\documentclass[11pt]{amsart}

\usepackage{enumerate,tikz-cd,mathtools,amssymb}
\usepackage{microtype}

\usepackage[a4paper,top=3cm,bottom=2.5cm,left=3cm,right=3cm,marginparwidth=1.75cm]{geometry}

\usepackage{verbatim}

\definecolor{xdxdff}{rgb}{0.49019607843137253,0.49019607843137253,1.}
\definecolor{uuuuuu}{rgb}{0.26666666666666666,0.26666666666666666,0.26666666666666666}
\definecolor{ududff}{rgb}{0.30196078431372547,0.30196078431372547,1.}

\definecolor{cite}{RGB}{44,123,182}
\definecolor{ref}{RGB}{215,25,28}

\usepackage[pdftex,
              pdfauthor={Giulia Gugiatti, Franco Rota},
              pdftitle={},
              pdfsubject={},
              pdfkeywords={},
              colorlinks=true,
              citecolor=cite,
              linkcolor=ref 
              ]{hyperref}

\usepackage[all,cmtip]{xy}

\usepackage{libertine}

\usepackage{adjustbox}
\usepackage{multirow}
\usepackage{makecell}
\makeatletter
\newtheorem*{rep@theorem}{\rep@title}
\newcommand{\newreptheorem}[2]{%
\newenvironment{rep#1}[1]{%
 \def\rep@title{#2 \ref{##1}}%
 \begin{rep@theorem}}%
 {\end{rep@theorem}}}
\makeatother

\theoremstyle{plain}
\newtheorem{thm}{Theorem}[section]
\newtheorem{corollary}[thm]{Corollary}
\newtheorem{lemma}[thm]{Lemma}
\newtheorem{ld}[thm]{Lemma/Definition}
\newtheorem{proposition}[thm]{Proposition}

\newtheorem*{thm*}{Theorem}
\newtheorem*{corollary*}{Corollary}
\newtheorem*{lemma*}{Lemma}
\newtheorem*{ld*}{Lemma/Definition}
\newtheorem*{proposition*}{Proposition}

\theoremstyle{definition}
\newtheorem{definition}[thm]{Definition}
\newtheorem{remark}[thm]{Remark}
\newtheorem{example}[thm]{Example}

\newtheorem*{definition*}{Definition}
\newtheorem*{remark*}{Remark}
\newtheorem*{example*}{Example}
\newtheorem*{xca*}{Exercise}
\newtheorem*{claim*}{Claim}
\newtheorem*{fact*}{Fact}
\newtheorem*{notation*}{Notation}
\newtheorem*{construction*}{Construction}
\newtheorem*{ack*}{Acknowledgements}
\newtheorem*{question*}{Question}
\newtheorem*{problem*}{Problem}
\newtheorem*{conjecture*}{Conjecture}
\newtheorem*{assumption*}{Assumption}

\newcommand{\C}{\mathbb{C}}
\newcommand{\CC}{\mathbb{C}}
\newcommand{\Z}{\mathbb{Z}}
\newcommand{\ZZ}{\mathbb{Z}}
\newcommand{\Q}{\mathbb{Q}}

\newcommand{\N}{\mathbb{N}}

\newcommand{\pr}[1]{\mathbb P^{#1}}

\newcommand{\HHom}{\mathcal{H}om}

\DeclareMathOperator{\GL}{GL}
\DeclareMathOperator{\SL}{SL}

\DeclareMathOperator{\coker}{coker}
\DeclareMathOperator{\Hom}{Hom}
\DeclareMathOperator{\Ext}{Ext}

\DeclareMathOperator{\im}{Im}

\DeclareMathOperator{\dR}{\mathbf{R}\!}
\DeclareMathOperator{\dL}{\mathbf{L}\!}

\newcommand{\rnat}{\rho_{\mathrm{nat}}}
\newcommand{\rdet}{\rho_{\mathrm{det}}}

\DeclareMathOperator{\Bl}{Bl}

\DeclareMathOperator{\Pic}{Pic\,}

\DeclareMathOperator{\Spec}{Spec}
\DeclareMathOperator{\Proj}{Proj}

\newcommand{\ch}[1]{\mathrm{ch}_{{#1}}}

\DeclareMathOperator{\Coh}{Coh}
\DeclareMathOperator{\QCoh}{QCoh}

\newcommand{\sX}{\mathcal{X}}       

\newcommand{\sO}{\mathcal{O}}

\newcommand{\sL}{\mathcal{L}}

\newcommand{\sR}{\mathcal{R}}

\newcommand{\sT}{\mathcal{T}}

\newcommand{\sZ}{\mathcal{Z}}

\newcommand{\sC}{\mathcal{C}}
\newcommand{\sD}{\mathcal{D}}

\newcommand{\sE}{\mathcal{E}}

\usepackage{leftidx}						

\usepackage{xcolor}						
\newcommand{\AAA}[1]{\textcolor{red}{#1}}

\renewcommand{\P}{\mathbb{P}}

\newcommand{\TT}{\mathbb{T}}

\newcommand{\RR}{\mathbb{R}}
\newcommand{\QQ}{\mathbb{Q}}

\usepackage{pgfplots}
\pgfplotsset{compat=1.15}
\usepackage{mathrsfs}
\usetikzlibrary{shapes.geometric, calc, arrows}
\usepackage[font=small,labelfont=bf]{caption}

\begin{document}
\title[Full exceptional collections for anticanonical log del Pezzo surfaces]{Full exceptional collections for \\ anticanonical log del Pezzo surfaces}

\author[G. Gugiatti]{Giulia Gugiatti}
\address{GG: Math Section \\ ICTP \\ Leonardo Da Vinci Building,
  Strada Costiera 11\\
  34151 Trieste, Italy} 
\email{ggugiatt@ictp.it}

\author[F. Rota]{Franco Rota}
\address{FR: School of Mathematics and Statistics \\ University of Glasgow \\ Glasgow G12 8QQ,
United Kingdom} 
\email{franco.rota@glasgow.ac.uk}


\begin{abstract}
Motivated by homological mirror symmetry, 
this paper constructs explicit full exceptional collections for the canonical stacks associated with the series of log del Pezzo surfaces
constructed by Johnson and Koll\'ar in \cite{MR1821068}.
These surfaces have cyclic quotient, non-Gorenstein, singularities. The construction involves both the $\GL(2,\C)$ McKay correspondence, and  the study of the minimal resolutions of the surfaces, which are birational to degree two del Pezzo surfaces. We show that a degree two del Pezzo surface arises in this way if and only if it admits a generalized Eckardt point, and in the course of the paper we classify the blow-ups of $\pr 2$ giving rise to them.  
Our result on the adjoints of the functor of Ishii--Ueda \cite{IU15} applies to any finite small subgroup of $\GL(2,\C)$. 
\end{abstract}
 
\maketitle
\setcounter{tocdepth}{1}
\tableofcontents

\section{Introduction}\label{sec_intro}
A log del Pezzo surface is a projective surface with quotient singularities and ample  anticanonical class.
For each positive integer $k$, consider the family of hypersurfaces  of degree $8k+4$ in weighted projective space 
\begin{equation}
\label{eq:series}
X_{8k+4}	 \subset \P(2,2k+1,2k+1, 4k+1).
\end{equation}
With the exception of 22 sporadic cases, all quasismooth and wellformed log del Pezzo weighted hypersurfaces $X \subset \P(a_1, a_2,a_3,a_4)$ which are anticanonical (that is,  with $-K_X \sim \sO_X(1)$) arise in this way \cite{MR1821068}.

Fix an integer $k>0$ and let $X$ be a general surface of the family \eqref{eq:series}. Then $X$ has isolated cyclic quotient log-terminal  
singularities.
Motivated by homological mirror symmetry, summarized in Section \ref{ssec:related} below, we study the derived category of coherent sheaves of the canonical stack of $X$. This is the unique smooth Deligne-Mumford stack with coarse moduli space $X$ and isomorphic to $X$ outside a locus of codimension 2 \cite[(2.8),(2.9)]{Vistoli}\cite[\S 4.1]{FMN10}\footnote{Sometimes, the expression ``canonical stack" refers to the index 1 cover stack of $X$ \cite[\S 5.2]{KM98}. This is a separate notion, which does not appear in the present paper.}.
The main result is the following.

\begin{thm}[{= (\ref{thm:stack-cat})}]
\label{thm:Main1}
Let $\sX$ be the canonical stack associated with $X$. Then $D^b(\Coh(\sX))$ admits full exceptional collections, given explicitly in Section \ref{sec:DerCat}.
\end{thm}

The exceptional collections of Theorem \ref{thm:Main1} are formed by $10+12k$ objects. In Section \ref{ssec:HomSpaces} we also exhibit a collection whose nontrivial morphisms are only in degree $0$ and $1$, which in turn implies the existence of a tilting complex on $\sX$ (see Section \ref{ssec:tilting}).

\smallskip

 The starting point of our construction is
the $\GL(2,\C)$ McKay correspondence, as formulated in \cite{IU15}.
For $X$ a surface with finite quotient singularities, denote by $\widetilde{X}$ its minimal resolution. The canonical stack $\sX$ of $X$ can also be interpreted as a stacky resolution:
 \begin{center}
        \begin{tikzcd}
     \widetilde{X} \arrow{dr} & & \sX \arrow{dl} \\
      & X & 
\end{tikzcd}
    \end{center}

The correspondence states that there is a fully faithful functor \mbox{$\Phi \colon D^b(\Coh(\widetilde{X})) \to D^b(\Coh(\sX))$}. Moreover, there exists a semiorthogonal decomposition  (SOD)
\begin{equation} \label{eq:sod}D^b(\Coh(\sX)) = \left\langle \textbf{e}_1,...,\textbf{e}_N,\Phi\left(D^b(\Coh(\widetilde{X}))\right) \right\rangle,
\end{equation}	
where $(\textbf{e}_1,...,\textbf{e}_N)$ is an exceptional collection. 
{If $X$ has only cyclic singularities, $\widetilde{X}$ coincides locally with the moduli space of clusters on $X$, $\Phi$ is the integral functor of the corresponding universal family, and the singularities determine the collection combinatorially (see Section \ref{ssec:gl2}).}

The exceptional sequences of Theorem \ref{thm:Main1} are constructed by applying the above decomposition to the surfaces \eqref{eq:series}, together with:
\begin{enumerate}[(i)]
\item a detailed description of $D^b(\Coh(\widetilde{X}))$;
\item a computation of the images of $\textbf{e}_1,...,\textbf{e}_N$ through the left adjoint functor $\Psi$ to $\Phi$. 
\end{enumerate}

To obtain (i) and (ii), it is necessary to prove two results which may be of independent interest. We give a brief overview here. \smallskip

The first result describes the geometry of the minimal resolution $\widetilde{X}$ of the surfaces \eqref{eq:series}.

\begin{thm}[{= (\ref{thm:min-resol})}]
\label{thm:MinRes}
Let $\widetilde{X}$ be the minimal resolution of $X$. Then $\widetilde{X}$ is obtained from a smooth degree 2 del Pezzo surface $X'$ via a sequence of $1+4k$ blow-ups. 
\end{thm}

As an immediate consequence, full exceptional collections for $\widetilde X$, and hence $\sX$, exist.  
In fact, this can be shown in general for log del Pezzo surfaces since they are rational (see Remark \ref{rem:rationality}).
In the case at hand, we describe them explicitly combining a detailed study of $\widetilde{X}$ and $X^\prime$ with Theorem \ref{thm:calcolo psi} below.

The surface $X'$ appearing in Theorem \ref{thm:min-resol} does not depend on the integer $k$ and it admits a generalized Eckardt point, i.e. a point of intersection of four $(-1)$-curves. Del Pezzo surfaces of degree 2 with a generalized Eckardt point form a stratum of codimension two in the moduli space of all del Pezzo surfaces 
 \cite[Table 8.9]{Dol_CAG}.
 Repeating the construction of Theorem \ref{thm:MinRes}
 in families, we obtain the following result about 
  the moduli space $M_k$ of quasismooth and wellformed hypersurfaces of the form \eqref{eq:series} (we define $M_k$ in Section \ref{sec:moduli}).

\begin{proposition}
    [{= (\ref{pro:moduli})}]
For all $k>0$ integer, $M_k$ is isomorphic to the moduli space of del Pezzo surfaces  with a generalized Eckardt point. 
\end{proposition}

In Proposition \ref{prop:ClassificationOfEckardtPoints} we classify all configurations of seven points on $\pr 2$ which, once blown up, give rise to a del Pezzo surface with a generalized  Eckardt point. This classification is then used to describe the Hom-spaces of the exceptional collections of $\sX$.\medskip


Our second result computes the image via $\Psi$ \footnote{Here $\Psi$ is the left adjoint to the (local version of) the functor $\Phi$ defined above. In the rest of the paper, we will distinguish locally and globally defined functors by denoting them with lower-case and upper-case letters, respectively.} of simple sheaves $\sO_0\otimes \rho \in D([\C^2/G])$, where $G$ is any finite small subgroup of $\GL(2,\C)$ and $\rho$ is an irreducible representation of $G$. 
Consider the minimal resolution of $\C^2/G$, let $E_1, \dots, E_m$ be the exceptional curves and let $E$ be the fundamental cycle. 
Write $\rho_i$, with  $i=0,...,n$, for the irreducible representations of $G$, and assume that $\rho_i$, with $i=1,...,m$, is the special representation corresponding to $E_i$ via the McKay correspondence (see Section \ref{ssec:gl2}). 
 
\begin{thm}[{= (\ref{thm:psiO0rho_general})}]
\label{thm:calcolo psi}
Let $\rnat$ be the natural representation of $G\subset \GL(2,\C)$ and $\rdet \coloneqq\Lambda^2\rnat$. 
For $i \in \{1, \dots, m\}$ let $-\alpha_i$ be the self intersection of the exceptional curve $E_i$. Then we have:
\begin{align}\label{eq:formula_calcoloPsi}
\Psi(\sO_0\otimes \rho) \simeq \begin{cases}
\sO_E(E)[1]& \mbox{ if } \rho = \rdet^\ast \\
\sO_{E_{i}}(-\alpha_{i}+1) & \mbox{ if } \rho\otimes \rdet=\rho_i \mbox{ for some }1\leq i\leq m \\
0 & \mbox{ otherwise}.
\end{cases}
\end{align}

\end{thm}

In Section \ref{ssec:alternative_pf}, we apply Theorem \ref{thm:calcolo psi} to give an alternative proof of \cite[Proposition 1.1]{IU15}, which states that the essential image of $\Phi$ is generated by $\{\sO_{\C^2}\otimes \rho_i\}_{i=1}^m$, and its right orthogonal by $\{\sO_{\C^2}\otimes \rho_i\}_{i=m+1}^n$.

To describe morphism spaces of the SOD \eqref{eq:sod}, one can compute either the images through $\Phi$ of objects in $D(\widetilde{X})$, or the images $\Psi(\textbf{e}_i)$ and use adjunction.  
The latter approach is simpler and gives a more general answer. It is simpler since the computations take place on the smooth surface $\widetilde{X}$ and not on $\sX$. On the other hand, our Theorem \ref{thm:calcolo psi} shows that the objects $\Psi(\sO_0\otimes \rho)$ do not depend on the global properties of $\widetilde{X}$. In fact, the formula \eqref{eq:formula_calcoloPsi} determines $\Psi(\mathbf{e}_i)$ for all surfaces with quotient singularities by finite small groups, since the $\textbf{e}_i$ all have a composition series whose factors are simple sheaves $\sO_0\otimes \rho$ \cite[Section 3]{IU15}.


\subsection{Related works and further problems} \label{ssec:related}
The main motivation for this work comes from mirror symmetry for Fano varieties -- also known as Fano/Landau--Ginzburg (LG) correspondence --
 from the point of view of Kontsevich's homological mirror symmetry (HMS) \cite{Kon94}. 
The  correspondence predicts that the mirror of a $n$-dimensional Fano orbifold $\mathcal{X}$ is a LG model $(Y,w)$ of the same dimension\footnote{In mirror symmetry, it is important to think of a Fano variety $X$ with quotient singularities as a canonical smooth Deligne-Mumford stack $\mathcal{X}$. This allows one to define Gromov--Witten invariants for $\mathcal{X}$. See \cite[Section 3]{MR3830183} for a concise summary of the quantum period of a Fano orbifold.}, and  can be interpreted at different levels.

Its Hodge-theoretic formulation identifies the regularised quantum differential operator of $\mathcal{X}$ and the Picard--Fuchs operator of $(Y,w)$ \cite{Gol1,Gol2, P}.
Recent works suggest that the Hodge-theoretic mirrors of  
Fano orbifolds with a $\Q$-Gorenstein (qG) degeneration\footnote{
    We refer the reader to \cite{KSB, Hack, Cor15} for the notion of $\QQ$-Gorenstein  (qG)  deformation of varieties with quotient singularities.}
to a toric variety 
should be LG models restricting  to specific 
Laurent polynomials built out of the combinatorics of the toric variety. This perspective, first laid out in \cite{ Gol1, P2,CC}, agrees with the  Givental/Hori--Vafa mirror construction \cite{Giv3,Hori-Vafa} for toric complete intersections  and the intrinsic mirror symmetry program \cite{2016arXiv160900624G} for log Calabi--Yau pairs. 

One formulation of HMS\footnote{A parallel formulation, which translates the Hodge-theoretic version of mirror symmetry mentioned above, is an equivalence between a suitable Fukaya category of $\sX$ and the category of matrix factorisations of $(Y,w)$.} predicts an equivalence  $D^b(\Coh(\sX)) \equiv D\mathrm{Lag}_{\mathrm{vc}}(Y,w)$ between the bounded derived category  of coherent sheaves of $\mathcal{X}$ and the derived category of Lagrangian vanishing cycles of $(Y,w)$ \cite{Seidel}\footnote{A rigorous definition of this category has been proposed by Seidel in the case where $(Y,w)$ is a Lefschetz fibration.}. 
It is expected that HMS implies Hodge-theoretic mirror symmetry (see \cite{Kon94}, and \cite{GPS}), even though the relation between these two formulations is not proven in full generality.
Unlike Hodge-theoretic mirror symmetry, HMS has been proven systematically only for few classes of Fano varieties; among them, smooth del Pezzo surfaces \cite{AKO06}, and weighted projective planes 
\cite{AKO08}.

The focus of our work is on the surfaces \eqref{eq:series}, 
which we think of 
as canonical stacks. These surfaces 
have isolated cyclic quotient singularities which are qG-rigid, that is, which are rigid with respect to qG-deformations.  
As such, they fit into the context of \cite{Cor15}. However, since their 
 anticanonical linear system is empty, they do not admit a qG-degeneration to a toric variety (see \cite[Remark 2.7]{ACGG}), thus they are out of the mirror symmetry setting of \cite[Conjecture A]{Cor15}.
The work \cite{ACGG} contains the only known mirror construction for the surfaces \eqref{eq:series}. The authors show that regularised quantum differential operator of the surfaces is hypergeometric, and use hypergeometric motives to build their Hodge-theoretic mirrors: these are pencils of hyperelliptic curves which do not restrict to Laurent polynomials. 

This paper is the first step towards the study of HMS for the surfaces \eqref{eq:series}, as it describes exceptional collections of $D^b(\Coh(\sX))$. 
A construction of other exceptional collections for $\sX$ using other methods appears to be challenging: for example, neither $X$ nor $\widetilde X$ are toric, so the Kawamata construction \cite{Kaw06} does not apply. 
Nevertheless, the surfaces $X$ are hypersurfaces in weighted projective spaces, so they inherit  a --- not full --- exceptional collection from the ambient space.
This collection appears in Elagin \cite{Ela07}. There, the author uses it to build  a full exceptional collection for the stack of a surface with one log-terminal singularity via an \textit{ad hoc} mutation of a set of generators for the residual category. 
There seems to be no systematic way of extending the construction  \cite{Ela07} to other cases.

\subsection*{Structure of the paper}

Section \ref{sec:preliminaries} recalls definitions and basic properties. We set some toric geometry notation in \ref{ssec:toric}. Then we recall properties of log del Pezzo surfaces in \ref{ssec:log-dp} and the $\GL(2,\C)$ McKay correspondence in \ref{ssec:gl2}.

Section \ref{sec:AdjointComputation} contains the computation of the adjoint functor $\Psi$, and the alternative proof of \cite[Prop. 1.1]{Ish02} in \ref{ssec:alternative_pf}. In Section \ref{ssec:CyclicSingAdjointComputation}, we rephrase Theorem \ref{thm:calcolo psi} in the cyclic singularities setting, and reprove it using toric geometry.

Section \ref{sec:min-res} is dedicated to the minimal resolutions and contains the proof of Theorem \ref{thm:MinRes}. In \ref{ssec:Eckardt} we classify degree 2 del Pezzo surfaces admitting a generalized Eckardt point.

In Section \ref{sec:DerCat} we show the existence of full exceptional collections for $\widetilde{X}$ and $\sX$ and describe them explicitly. We discuss the existence of a tilting complex for $\widetilde X$ and $\sX$ in \ref{ssec:tilting}.

\subsection*{Notation and conventions}

We work over the field of complex numbers. 
We use $D(-)=D^b(\mathrm{Coh}(-))$ to denote the bounded derived category of coherent sheaves and $D_G(-)$ to denote the bounded derived category of $G$-equivariant coherent sheaves.
We write $\sO(n)$ to denote the line bundle of degree $n$ in the (weighted) projective space $\mathbb P$. For $C\subset \mathbb{P}$ a rational curve, we write $\sO_C(1)$ for its line bundle of degree one, so that $\sO(1)_{|C}=\sO_C(\deg C)$. We recollect here some of the notation used in the paper:
\begin{center}
\begin{tabular}{clr}
 Notation & Description & Local alternative\\\hline
    $X$ & Singular surface & $\C^2/G$\\
    $\widetilde X$ & minimal resolution of $X$ & $Y$\\
    $\sX$ &  canonical stack of $X$ & $[\C^2/G]$ \\
    $\Phi$ & functor $\Phi\colon D(\widetilde X) \to D(\sX)$ & $\phi$  \\
    $\Psi$ & left adjoint to $\Phi$  & $\psi$  \\
\end{tabular}
\end{center}

\subsection*{Acknowledgements}
We are thankful to Arend Bayer, Alessio Corti, and Michael Wemyss for valuable discussions and advice. We thank  Akira Ishii for his remarks on Section \ref{ssec:alternative_pf}. We are grateful to Dominic Bunnett and to Stefano Filipazzi for helpful correspondence. 
The second author was partially supported by EPSRC grant EP/R034826/1.


\section{Preliminaries}\label{sec:preliminaries}
This section is organised as follows: in \ref{ssec:toric} we set up our notation for toric varieties and we state some well-known facts which we will use repeatedly; 
in \ref{ssec:log-dp}  we recall the notion of 
anticanonical log del Pezzo surface, the classification \cite{MR1821068}, and some basic facts about smooth del Pezzo surfaces of degree two; \ref{ssec:gl2} is a summary of the $\GL(2,\C)$ McKay correspondence, with a focus on the case of cyclic quotient singularities.

\subsection{Toric setup}  \label{ssec:toric}
For details on the content of this Section, we refer the reader to \cite[Chapters 1, 3 and 11]{cox2011toric} and
\cite[\S 10]{KM92}.

For a free abelian group of finite rank $L$, we denote by $L_{\RR}=L \otimes_{\ZZ} \RR$ the
associated real vector space. For a torus $\TT$, we denote by
$M=\mathrm{Hom}(\TT, \CC^\times)$ the character group and we write $N=\mathrm{Hom}(M, \ZZ)$. Given vectors $v_1, \dots, v_k \in L_\RR$, we write $\langle v_1, \dots, v_k \rangle $ for the cone they generate.

For a fan $\Sigma \subset N_\RR$, we denote by $G_{\Sigma}$ the
associated toric variety. The toric variety is proper if and only if $\Sigma$ is a complete fan, that is, its support $|\Sigma|$ is all $N_\RR$. The toric variety has finite quotient singularities 
if and only if the fan is simplicial. 
We write $\Sigma(r)$ for the set of $r$-dimensional cones of the fan $\Sigma$.  We denote by $\rho_i$ the primitive
generators of the $1$-dimensional cones (or \emph{rays}) of $\Sigma$, and by
$D_{\rho_i} \subset G_{\Sigma}$ the corresponding divisors under the toric orbit-cone correspondence~\cite[\S 3.2]{cox2011toric}. 
\smallskip

Below we recall the construction of star subdivision of a fan, as given in \cite[\S 11.1]{cox2011toric}.

\begin{ld}[{\cite[Lemma 11.1.3]{cox2011toric}}]
Let $\Sigma \subset N_\RR$ be a fan with rays $\rho_1, \dots, \rho_n$ and let $v \in |\Sigma| \cap N$ be a primitive element. 
Define \emph{the star subdivision}  of $\Sigma$ at $v$, denoted by $\Sigma^\ast(v)$, as the 
set of the cones:
\begin{itemize}
\item[(1)] $\sigma$, where $v \notin \sigma \in \Sigma$;	
\item[(2)] $\mathrm{cone}(\tau, v)$, where $v \notin \tau \in \Sigma$ and there exists $\sigma \in \Sigma$ such that $\{v\} \cup \tau \subset \sigma$. \end{itemize}
The set $\Sigma^\ast(v)$ is a refinement of $\Sigma$, thus  induces a toric morphism $G_{\Sigma^\ast(v)} \to G_\Sigma$.
\end{ld}

In Section \ref{sec:min-res} we will use toric morphisms coming from star subdivisions to resolve singularities of toric varieties. 
In particular, we will consider situations where $\Sigma$ contains a cone $\tau$ with rays $\rho_{i_1}, \dots, \rho_{i_r}$ such that the corresponding affine chart $U_\tau$ is of the form:
\[
 U_\tau \simeq \CC^r/{\ZZ_m(a_1, \dots, a_r) }\times (\CC^\times)^{n-r} \subset G_{\Sigma}
\]
 where $m,a_i$ are positive integers such that $a_i <m$ for all $i$. 
Then the toric variety $G_{\Sigma^\ast(v)}$ associated to the star subdivision $\Sigma^\ast(v)$, with
\begin{equation} \label{eq:vectorv} v=\frac{1}{m}\sum_{i=1}^{r} a_j \; \rho_{i_j}
\end{equation}
is the weighted blow-up of $G_\Sigma$ along $V(\tau)$ with weights $\frac{1}{m}(a_1, \dots , a_r)$ \cite[(10.3)]{KM92}, where we write $V(\tau)$ for the closure of the $\TT_N$-orbit in $G_\Sigma$ corresponding to $\tau$. As such, the variety $G_{\Sigma^\ast(v)}$ fits in a fiber square:
\begin{center}
\begin{tikzcd}
\mathcal{E}  \ar[d]  \arrow[hookrightarrow,r]&   G_{\Sigma^\ast(v)} \ar[d]  \\
 V(\tau)  \ar[hookrightarrow,r]  & \  G_\Sigma
\end{tikzcd}
\end{center}
  where the exceptional divisor $\mathcal{E}$ is a $\mathbb{P}(a_1, \dots, a_r)$-bundle over $V(\tau)$.

\subsection{Log del Pezzo weighted hypersurfaces} \label{ssec:log-dp}
For positive integers $a_i$, $i=1, \dots, m$, we denote by $S=S(a_1, \dots, a_n)$ the graded polynomial ring $\CC[x_1, \dots, x_n]$ where $\deg x_i=a_i$, and by $\P(a_1, \dots, a_m)$ the weighted projective space with weights 
$a_1, \dots, a_m$.  We simply write $\P$ if the weights and $m$ are clear from the context. We write $\P_{\mathrm{sing}}$ for the singular locus of $\P$. The space $\P$ is said \emph{wellformed} if the quotient stack $\pmb{\P}=[\Spec{S} \setminus \{ 0\}/\CC^\times]$  is canonical, that is, the natural morphism $s\colon \pmb{\P} \to \P$ is an isomorphism in codimension $1$.
  
A subvariety $X \subset \P$ is called \emph{quasismooth} of dimension $m$ if its affine cone $C_X \subset \CC^{n+1}$ is smooth of dimension $m+1$ outside of its vertex $\underline{0}$. This implies that the 
singularities of $X$ are only due to the $\CC^\times$-action defining $\P$, hence are cyclic quotient singularities.
  A subvariety $X \subset \P$ of codimension $c$  is called \emph{wellformed} if $\P$ is wellformed and 
 $\mathrm{codim}_{X}(\P_{\mathrm{sing}} \cap X) \geq 2$. 
 Observe that, if $X=\Proj(S/I) \subset \P$, where $I$ is a homogeneous ideal of $S$, is quasismooth and wellformed, then $X$ is the coarse moduli space of the canonical stack 
 $[\Spec (S/I) \setminus \{ 0\}/\CC^\times]$.

For $X=X_d\subset \P(a_1, \dots, a_m)$ a quasismooth and wellformed hypersurface of degree $d$, the adjunction formula for the canonical sheaf $K_X$ \cite[Section 6.14]{IF} states that
  \[ K_X=\sO_X(d-\sum_{i=1}^m a_i).
  \]
  By definition $X$ is Fano if $-K_X$ is ample, i.e., if and only if $\sum_{i=1}^m a_i-d>0$. $X$ is called \emph{anticanonical} if $-K_X=\sO_X(1)$, i.e., if and only if $\sum_{i=1}^m a_i-d=1$.
 \smallskip

A Fano surface with quotient singularities is called a \emph{log del Pezzo surface}. 
Johnson and Koll\'ar \cite{MR1821068} classify all anticanonical quasismooth wellformed log del Pezzo surfaces in weighted projective $3$-spaces.
They are the general members of either the series \eqref{eq:series}, where  $k \in \N, \  k > 0$, 
or of one of $22$ sporadic families. The sporadic families are all listed in~\cite[Theorem 8]{MR1821068}.
General surfaces of the form \eqref{eq:series} have isolated log-terminal singularities (see Section \ref{ssec:basic}).  
On the other hand, there are three sporadic families for which quasismoothness actually coincides with smoothness, namely cubic surfaces 
$X_3 \subset \P^3$, degree-two del Pezzo surfaces $X_4 \subset \P(1,1,1,2)$, and degree-one del Pezzo surfaces $X_6 \subset \P(1,1,2,3)$.
 
\subsubsection{Degree-two del Pezzo surfaces} \label{ssec:dP2}
We will prove in Section \ref{sec:min-res} that the geometry of the series \eqref{eq:series} is related to the geometry of some special
 degree-two smooth del Pezzo surfaces. Below we recall some basic facts about degree-two smooth del Pezzo surfaces which we will use repeatedly. 
 
Let $S$ be a smooth del Pezzo surface of degree two. Then $S$ is isomorphic to the blow-up of $\P^2$ at $7$ distinct points in general position \cite[IV.24.4]{Man86}\cite[Thm 1]{ Dem80}. Conversely, the choice of a set $\mathcal{L}=\{l_1, \dots, l_7\}$  of $7$ disjoint $(-1)$-curves  on $S$  is equivalent to the choice of a birational morphism $\pi  \colon S \to \P^2$ which expresses $S$ as the blowup of $\pr 2$ at the $7$ points $x_i=\pi(l_i)$. We refer to such a set $\mathcal{L}$ as to an \textit{exceptional set} for $S$.

On the other hand, $S$ is expressed by its anticanonical series as a double cover  $ \sigma \colon S \to \pr 2$ branched along a quartic curve $Q=(f_4(x_1,x_2,x_3)=0) \subset \P^2$ \cite[III.3.5]{Kol96}
Hence, $S$ can be defined by an equation of the form \[(x_4^2-f_4(x_1,x_2,x_3)=0) \subset \P(1,1,1,2).\] 

The surface $S$ contains $56$ $(-1)$-curves, which can be described in terms of  the morphisms $\sigma$ and $\pi$. 
On the one hand,  the $56$ $(-1)$-curves are the preimages through $\sigma$ of the $28$ lines bitangent to $Q$. Each preimage is a reducible curve of the form $m\cup m'$ where $m,m'$ are $(-1)$-curves on $S$ satisfying $m+m'\sim - K_S$ (we refer to $m$ and $m^\prime$ as \emph{dual} curves). Thus any two of the $56$ $(-1)$-curves are either disjoint, intersect at one point, or intersect at two points precisely if they are dual to each other. 
On the other hand, given an exceptional set $\sL=\{l_1, \dots, l_7\}$ for $S$ and the corresponding blow-down morphism  $\pi \colon S \to \P^2$, the $28$ pairs of $(-1)$-curves
 described above are the strict transforms through $\pi$ of:
\begin{itemize}
\item the line through two points $x_i,x_j$ and the conic through the remaining 5 points (21 pairs);
\item the cubic with a node at one of the $x_i$ and passing through the remaining 6 points, paired with the exceptional divisor $l_i$ (7 pairs).
\end{itemize}
 
By the proof of \cite[Lemma 4.1]{TVAV09}, at most four bitangents to $Q$ can meet at one point, thus at most four (-1)-curves on $S$ can meet at one point of $S$. Such points on $S$ are called \textit{generalized Eckardt points}.

Denote by $I^{1,7}$ the lattice $\ZZ^8$ together with the symmetric bilinear form defined by the diagonal matrix with diagonal entries $(1,-1,\dots,-1)$ with respect to the standard basis
\[ e_0=(1,0, \dots, 0)^T, \quad e_1=(1,0 \dots, 0)^T, \quad \dots \quad ,  e_7=(0,0 \dots, 1)^T.\] 
The Picard lattice $\Pic(S)$ is isomorphic to the lattice $I^{1,7}$\cite[Section 8.2.6]{Dol_CAG}. The set of $(-1)$-curves on $S$ is in bijection with the set of  exceptional vectors in $ I^{1,7}$, that is, the vectors $v \in I^{1,7}$ such that  $v \cdot k=-1$, where $k=(3,-1,\dots,-1)^T$, and  $v^2=-1$. These vectors are listed in \cite[Proposition 8.2.19]{Dol_CAG}.

\subsection{\texorpdfstring{$\GL(2,\C)$}{GL(2,C)} McKay correspondence} \label{ssec:gl2}
Suppose $G$ is a finite subgroup of $\GL(2,\C)$ which is small, i.e. it acts freely on $\C^2/\{0\}$. Let $G$ act naturally on $S= \C[x,y]$ and set $R=S^G$. We can associate two  resolutions to the singularity $X=\Spec R$: its minimal resolution $f\colon Y \to X$, and a \textit{non-commutative crepant resolution} (\cite[Definition 4.1]{VdB04_NCCR}) given by the skew group algebra $S\ast G$. The category of finitely generated $S\ast G$-modules is equivalent to the category of $G$-equivariant sheaves on $\C^2$, and therefore to the category of coherent sheaves over the quotient stack $[\C^2/G]$:
\[ \mathrm{mod }\;S\ast G \simeq \Coh[\C^2/G]. \]

There are bijective correspondences among the sets of irreducible representations of $G$, indecomposable Cohen-Macaulay $R$-modules, indecomposable projective $S\ast G$-modules, and indecomposable full sheaves on $Y$ \cite[Remark 1.4, Lemma 2.2]{Esn85}.
More explicitly, to an irreducible representation $\rho$ of $G$, we associate a free $S$-module\footnote{By \cite[Lemma 1.1]{Aus86}, a $S\ast G$-module is projective precisely if it is free as a $S$-module. Then, the $S\otimes \rho$ for $\rho\in \mathrm{Irrep}(G)$ are all the indecomposable projectives.} $S\otimes \rho$, the CM $R$-module 
\[ M_{\rho}\coloneqq (S\otimes \rho^*)^G \]
and the coherent sheaf $\widetilde{M_\rho}\coloneqq f^*(M_\rho)/\mathrm{tors.}$ on $Y$ (sheaves arising this way are called \textit{full}).

Moreover, irreducible exceptional curves of $Y$ are in bijection with 
\textit{special} representations of $G$, i.e. irreducible representations $\rho$ for which 
$H^1\left(\widetilde{M_\rho}^\vee\right)=0$.
Let $E_1,...,E_m$ be the exceptional curves on $Y$: Wunram \cite[Theorem 1.2]{Wun88} shows that there are exactly $m+1$ special representations $\rho_0,\rho_1,...,\rho_m$ of $G$, where $\rho_0$ is the trivial representation and the other satisfy the \textit{multiplication formula} 
\[\mathrm{c}_1(\widetilde{M}_{\rho_i})\cdot E_j=\delta_{ij}.\]
A refined version of the formula is given by Ishii in \cite{Ish02}, and we briefly recall it here.

Consider the equivariant Koszul complex of $0\in\C^2$:
\begin{equation}
\label{eq:local_Koszul}
0\to S \otimes \rdet \xrightarrow{(y,-x)^T} S 
\otimes \rnat \xrightarrow{(x,y)} S \to \sO_0 \to 0,
\end{equation}
where $\rnat$ is the representation induced by $G\subset \GL(2,\C)$ and $\rdet\coloneqq\Lambda^2\rnat$.
Tensoring with $\rho^*$ and taking invariants we obtain the Auslander-Reiten (AR) sequences 
\begin{equation}
\label{eq:ARSeq}
\begin{split}
A_\rho \coloneqq  \left[  0\to \tau(M_\rho)  \to  N_\rho  \to M_\rho \to 0 \right],
\end{split}
\end{equation}
where $N_\rho\coloneqq (S\otimes \rnat \otimes \rho^*)^G$ and $\tau(M_\rho)\coloneqq (S\otimes \rdet \otimes \rho^*)^G$. The $A_\rho$ are exact except for $\rho=\rho_0$, in which case
\[ 0\to\tau(R) \to N_{\rho_0} \to R \to \sO_0 \to 0 \]
is exact. 

Recall that the fundamental cycle of $Y\to \Spec R$ is $E=\sum_{i=1}^m r_i E_i$, where $r_i\geq 1$ are the unique smallest integers satisfying $E\cdot E_i \leq 0$ for all $i=1,...,m$ \cite{Art66}. An immediate consequence of the proof of {\cite[Theorem 5.1]{Ish02}} is:

\begin{proposition}
\label{prop:Ishii5.1}
Let $\rho\in \mathrm{Irrep}(G)$. Denote by $E$ the fundamental cycle of $ f \colon Y \to X$ and by $\widetilde{A_\rho}\in D^b(Y)$ the complex formed by the torsion-free pull-back of the terms of $A_\rho$. 
Then
\begin{align}
\widetilde{A_\rho} \simeq \begin{cases}
\sO_E & \mbox{ if } \rho=\rho_0, \\
\sO_{E_{i}}(-1)[1] & \mbox{ if } \rho=\rho_i \mbox{ for } 1\leq i\leq m, \\
0 & \mbox{ otherwise.}
\end{cases}
\end{align}
\end{proposition}


\subsubsection{Derived McKay correspondence}

In the above setting, there is a full and faithful functor $\phi\colon D(Y) \to D[\C^2/G]$ constructed as follows. One observes that $Y$ can be constructed as the $G$-Hilbert scheme parametrizing $G$-\textit{clusters} of $\C^2$ (i.e. finite, $G$-invariant subschemes $Z\subset \C^2$ such that $H^0(\sO_Z)$ is isomorphic as a $\CC[G]$-module to the regular representation of $G$) \cite[Theorem 3.1, Theorem 8.1]{Ish02}. Thus, $Y\times \C^2$ admits a universal family $Z$, and one can define the diagram
\begin{equation} \label{eq:diagram}
\begin{tikzcd}
 Z \rar[phantom,"\subseteq"] & Y\times \C^2 \rar{\pi_{\C^2}} \dar{\pi_Y} &  \C^2 \dar \\
&Y \rar{f} & X
\end{tikzcd}
\end{equation}
where we let $G$ act trivially on $X$ and $Y$, so that all morphisms are equivariant. Define the Fourier--Mukai functor 
\begin{equation}
 \label{eq:phi}	\phi(-)\coloneqq \dR\pi_{\C^2*}\left(\pi_Y^*(-  \otimes \rho_0)\otimes^{\dL}\sO_Z\right) \colon D(Y) \to D([\C^2/G]).
 \end{equation}
The functor $\phi$ is an equivalence if and only if $Y$ is a crepant resolution, which happens if and only if  $G\subset \SL(2,\C)$, equivalently $R$ is Gorenstein.
Even in the $G\subset \GL(2,\C)$ case, $\phi$ admits left and right adoints. So its essential image is admissible and there is a semiorthogonal decomposition: 
\begin{equation} \label{eq:LocalSOD}
	D([\C^2/G]) = \left \langle \phi(D(Y))^\perp, \phi(D( Y))  \right\rangle.
\end{equation}

The essential image of $\phi$ and its right orthogonal are described as follows.
Since the functor $(  -  )^G$ of taking invariant part defines an equivalence of categories between projective $S\ast G$-modules and CM $R$-modules \cite[Sec. 2]{Aus86}\if  Prop. 2.2 \fi, the $\{ S\otimes \rho \}_{\rho\in \mathrm{Irrep}(G)}$ are projective generators of $\mathrm{mod-}S\ast G$ and hence they generate $D([\C^2/G])$. 

\begin{proposition}{\cite[Prop. 1.1]{IU15}} \label{prop:1.1}
\footnote{In Section \ref{ssec:alternative_pf} we provide an alternative proof of Proposition \ref{prop:1.1} based on our Theorem \ref{thm:psiO0rho_general}.}
	The essential image of $\phi$ in $D([\C^2/G])$ is generated by the modules $\{ S\otimes \rho \}$ where $\rho$ is a \textit{special} representation. The right orthogonal of $\phi(D(Y))$ in $D([\C^2/G])$ is generated by sheaves $\sO_0\otimes\rho$ where $\sO_0 = S/(x,y)$ and $\rho$ is not special.
\end{proposition}

\begin{remark} \label{rem:quiver}
Let $\mu$ be an irreducible representation of $G$ and write $a_{\mu \nu}$ for the multiplicity of $\nu \in \mathrm{Irrep}(G)$ in the decomposition of $\mu \otimes \rho_{\mathrm{nat}}$ into irreducible representations. Then by \eqref{eq:local_Koszul} we have
\[
\begin{split}
&\dim \Hom(\mathcal{O}_0 \otimes \mu, \mathcal{O}_0 \otimes \nu)=\delta_{\mu \nu}\\
&\dim \Ext^1(\mathcal{O}_0 \otimes \mu, \mathcal{O}_0 \otimes\nu)=a_{\mu \nu}\\
& \dim \Ext^2(\mathcal{O}_0 \otimes\mu, \mathcal{O}_0 \otimes\nu)=\dim \Hom(\mathcal{O}_0 \otimes\nu, \mathcal{O}_0 \otimes\mu \otimes \rho_{\mathrm{det}})	
\end{split}
\]
This information is summarised in the \emph{McKay quiver of $G$}, whose vertices represent 
the irreducible representation of $G$, and whose solid (resp. dashed) arrows from $\mu$ to $\nu$ denote a basis of $\Ext^1(\mathcal{O}_0 \otimes\mu, \mathcal{O}_0 \otimes\nu)$  (resp. $\Ext^2(\mathcal{O}_0 \otimes\mu, \mathcal{O}_0 \otimes\nu)$).  We give two examples in Figure \ref{fig:quiver_1/5(1,1)} and Figure \ref{fig:quiver_1/7(1,3)}. \end{remark}


\subsubsection{Cyclic singularities}
\label{sssec:cyclic_singularities}
For relatively prime integers $1\leq a <r$, define the group $G=\frac {1}{r}(1,a)$ as the cyclic small subgroup of $\GL(2, \CC)$ generated by 
\[ \zeta\coloneqq \begin{pmatrix}
\epsilon & 0 \\ 0 & \epsilon^a
\end{pmatrix}  \]
where $\epsilon$ is a primitive $r$-th root of unity. We use the same notation $\frac 1r(1,a)$ for the corresponding cyclic quotient singularity $\Spec R$.

\begin{definition}
\label{def:iseries}
For integers $1\leq a<r$ as above, 
set $i_0=r$, $i_1=a$, and define $i_t$ inductively by
\begin{equation} \label{eq:iseries}
i_t \coloneqq  \alpha_{t-1}i_{t-1}-i_{t-2}
	\end{equation}
where $\alpha_{t-1} \geq 2$ and $0\leq i_t<i_{t-1}$, until $i_n=1$ and $i_{n+1}=0$. 
The $I$-series $I(r,a)$ is the sequence $I(r,a)\coloneqq \{i_0=r,...,i_{n+1}=0\}$.
\end{definition}

The equations \eqref{eq:iseries} can be rearranged and spliced together to give the Hirzebruch--Jung continued fraction expansion of $\frac ra$, i.e. the expression:
\begin{equation}
	\frac ra = \alpha_1 - \frac{1}{\alpha_2-\frac{1}{\alpha_3-\frac{1}{(...)}}} \eqqcolon [\alpha_1,...,\alpha_n] 
\end{equation}
It is well-known \cite{Riem74} that the dual graph of the minimal resolution $Y$ of $\frac 1r(1,a)$ is the labeled diagram

\begin{center}
\begin{tikzpicture}[scale=0.8]
	\fill (-4,0) circle (3pt);
	\fill (-2,0) circle (3pt);
	\fill (2,0) circle (3pt);
	\fill (4,0) circle (3pt);

	\draw[line width=1.0pt] (-3.7,0) -- (-2.3,0);
	\draw[line width=1.0pt] (-1.7,0) -- (-1,0);
	\draw[line width=1.0pt] (1.7,0) -- (1,0);	
	\draw[line width=1.0pt] (3.7,0) -- (2.3,0);

	\node at (-4,-0.5) {$-\alpha_1$};
	\node at (-2,-0.5) {$-\alpha_2$};
	\node at (0,0) {$\cdots$};
	\node at (2,-0.5) {$-\alpha_{n-1}$};
	\node at (4,-0.5) {$-\alpha_n$};
\end{tikzpicture}
\end{center}
\medskip

For $l\in \ZZ$, denote by $\rho_l$ the irreducible representation of $G$ acting with weight $l$.
Then the modules
\[ M_l \coloneqq(\C[x,y]\otimes \rho_l^*)^G = \left\lbrace f\in \C[x,y] \mid \zeta\cdot f = \epsilon^l f \right\rbrace\]
are the indecomposable CM $R$-modules\footnote{Clearly, $\rho_l=\rho_m$ and $M_l =M_m$ if $l \equiv m \mod r$.}, 
and $M_l$ is special if $l\in I(r,a)$ \cite{Wun87}. In \cite[Appendix A1]{Wun88}) it is shown that their torsion free pull-backs $\widetilde{M}_{l}$ satisfy
\begin{equation} \label{eq:formula-cyclic} c_1(\widetilde{M}_l) \cdot E_t=\delta_{l, i_t}.
\end{equation}

Moreover, since in this case $\rnat= \rho_1\oplus \rho_a$ and $\rdet=\rho_{a+1}$, the AR-sequences \eqref{eq:ARSeq} specialise as follows:
\begin{equation}
\label{eq:ARSeq-cycl}
\begin{split}
A_l \coloneqq  \left[  0\to M_{l-a-1} \to  M_{l-1} \oplus M_{l-a}  \to M_l \to 0 \right],
\end{split}
\end{equation}
and by \eqref{eq:formula-cyclic} Proposition \ref{prop:Ishii5.1} rewrites as: 

\begin{proposition}
\label{prop:Ishii5.1-cyclic}
Let $l\in \{0, \dots, r-1\}$. Denote by $E$ the fundamental cycle of $f \colon Y \to X$ and  by $\widetilde{A_l}\in D^b(Y)$ the torsion-free pull-back of the AR-sequence ending in $M_l$. Then we have
\begin{align}
\widetilde{A_l} \simeq \begin{cases}
\sO_E & \mbox{ if } l=0 \\
\sO_{E_{t}}(-1)[1] & \mbox{ if } l=i_t \in I(r,a)\setminus\{0\} \\
0 & \mbox{ otherwise.}
\end{cases}
\end{align}
\end{proposition}

In the cyclic case,  
the right orthogonal $\phi(D(Y))^\perp$ admits a full exceptional collection. Every non-special representation $\rho$ contributes one exceptional sheaf, which has finite length and socle $\sO_{0}\otimes \rho$ \cite[Sec. 3]{IU15}. We do not recall here the construction of \cite{IU15}, but we point out below that the exceptional sheaves are simple in the two cases we consider in this paper, namely $G=\frac 1n(1,1)$, $n>1$ integer, and $G=\frac{1}{2k+1}(1,k)$, $k \geq 1$ integer.

\begin{example} \label{exa:1n11}
Let $G=\frac 1n(1,1)$, for $n>1$ integer. Then $I(n,1)=\{n,1,0\}$ and $\frac{n}{1}=[n]$. It follows that the minimal resolution $Y=\mathrm{Tot}\left( \mathcal{O}_{\P^1}(-n) \right)$ has a single exceptional $(-n)$-curve, $\rho_0$ and $\rho_1$ are the only special representations, and the exceptional collection computed in \cite{IU15} reads 
\begin{equation} \label{eq:IU_1n11}
 \phi(D( Y))^\perp = \left( \sO_0\otimes\rho_{2},...,\sO_0\otimes\rho_{n-1} \right).  
 \end{equation}
 For all $i$, the McKay quiver of $G$ has two solid arrows from $\rho_i$ to $\rho_{i+1}$, and one dashed arrow from $\rho_i$ to $\rho_{i+1}$ (see Figure \ref{fig:quiver_1/5(1,1)}).
\end{example}

\begin{example}  \label{exa:12k+1}
Let $G=\frac{1}{2k+1}(1,k)$, for $k \geq 1$ integer. One computes
\[ \frac{2k+1}{k}=[3,\underbrace{2,...,2}_{k-1 \text{ times}}], \]
and the relevant terms of the $i$-series are $i_t = k+1-t$ for $t=1,...,k+1$. It follows that the minimal resolution contains one $(-3)$-curve $E_1$ and $k-1 \times (-2)$-curves $E_i$, $i\in \{2, \dots k\}$, and the special representations are $\rho_{i_t}=\rho_{k+1-t}$, $t=1,...,k+1$. The exceptional collection computed in \cite{IU15} reads 
\[  \phi(D( Y))^\perp = \left( \sO_0\otimes\rho_{k+1},...,\sO_0\otimes\rho_{2k} \right).  \]
For all $i$, the McKay quiver of $G$ has a solid arrow from $\rho_i$ to $\rho_{i+1}$, a solid arrow from $\rho_i $ to $\rho_{i+k}$, and a dashed arrow from $\rho_i$ to $\rho_{i+k+1}$. \end{example}
\medskip

Figure \ref{fig:quiver_1/5(1,1)} and \ref{fig:quiver_1/7(1,3)} depict the McKay quiver of $\frac 15(1,1)$ and of  $\frac 17(1,3)$ , with objects and morphisms in $\phi(D(Y))^\perp$ colored in blue.
\begin{figure}[!ht]
\definecolor{uuuuuu}{rgb}{0.26666666666666666,0.26666666666666666,0.26666666666666666}
\definecolor{zzttqq}{rgb}{0.6,0.2,0}
\definecolor{qqqqff}{rgb}{0,0,1}
\definecolor{uuuuuu}{rgb}{0.26666666666666666,0.26666666666666666,0.26666666666666666}
\definecolor{qqqqff}{rgb}{0,0,1}
\definecolor{ududff}{rgb}{0.30196078431372547,0.30196078431372547,1}
\noindent
\minipage{0.48\textwidth}%
\centering
\begin{tikzpicture}[line cap=round,line join=round,>=triangle 45,x=1cm,y=1cm, scale=1.5]
\clip(-4,-2) rectangle (0.2,2.335171284442755);
\draw [->,line width=0.4pt] (-3.5968826033715913,0.2561305573032108) -- (-3.114311074135772,-1.1501689777097202);
\draw [->,line width=0.4pt] (-3.427201152734883,0.34913538672724864) -- (-2.9399665435658475,-1.0777659686963625);
\draw [->,line width=0.4pt] (-2.2722038100784587,1.7075436242735051) -- (-3.4032841527922915,0.8983861483320672);
\draw [->,line width=0.4pt] (-2.144396508570919,1.5324194375663407) -- (-3.275476851284752,0.7058607255831514);
\draw [->,line width=0.4pt] (-0.4020880502757565,0.8613988221285287) -- (-1.5331683929895892,1.7749637143204748);
\draw [->,line width=0.4pt] (-0.517380938667175,0.7319625796457784) -- (-1.6310600453392567,1.6281262357959732);
\draw [->,line width=0.4pt,color=qqqqff] (-0.606533174536701,-1.1843582390669056) -- (-0.19760412755554602,0.27734558844020785);
\draw [->,line width=0.4pt,color=qqqqff] (-0.7762146251734088,-1.0919093227876933) -- (-0.384686814234005,0.32629141461504174);
\draw [->,line width=0.4pt,color=qqqqff] (-2.6055292636773277,-1.3387845093226325) -- (-1.1960291442954747,-1.330083891301757);
\draw [->,line width=0.4pt,color=qqqqff] (-2.5897569686107866,-1.538381609117214) -- (-1.1802568492289336,-1.5296809910963383);
\draw [->,line width=0.4pt,dash pattern=on 3pt off 3pt,color=qqqqff] (-2.727337915969585,-1.1821733849468705) -- (-0.4651772305419195,0.46224342099863236);
\draw [->,line width=0.4pt,dash pattern=on 3pt off 3pt] (-0.517380938667173,0.5840520732908919) -- (-3.345081795451755,0.6014533093326433);
\draw [->,line width=0.4pt,dash pattern=on 3pt off 3pt] (-1.9442822940907776,1.523718819545465) -- (-2.823044714199217,-1.156071530884243);
\draw [->,line width=0.4pt,dash pattern=on 3pt off 3pt] (-3.345081795451755,0.5144471291238865) -- (-1.0133161658570846,-1.199574620988622);
\draw [->,line width=0.4pt,dash pattern=on 3pt off 3pt] (-0.9611124577318305,-1.138670294842492) -- (-1.8224736417985186,1.549820673608092);
\begin{scriptsize}
\draw [fill=qqqqff] (-2.8895204560046355,-1.3116224835819166) circle (1.5pt);
\draw[color=qqqqff] (-3.0013827511814376,-1.5375684194198596) node {$\rho_2$};
\draw [fill=qqqqff] (-0.8895204560046355,-1.3116224835819166) circle (1.5pt);
\draw[color=qqqqff] (-0.7321524311071531,-1.5375684194198596) node {$\rho_3$};
\draw [fill=qqqqff] (-0.3714864672547408,0.5904905490083897) circle (1.5pt);
\draw[color=qqqqff] (-0.04934659040824107,0.6125011533665175) node {$\rho_4$};
\draw [fill=uuuuuu] (-1.889520456004635,1.7660610535933359) circle (1.5pt);
\draw[color=uuuuuu] (-1.8060508206876396,2.0231914690886532) node {$\rho_0$};
\draw [fill=uuuuuu] (-3.50755444475453,0.5904905490083907) circle (1.5pt);
\draw[color=uuuuuu] (-3.776618654154343,0.6125011533665175) node {$\rho_1$};
\end{scriptsize}
\end{tikzpicture}
\caption{
\label{fig:quiver_1/5(1,1)} The case $\frac 15(1,1)$.}
\endminipage\hfill
\minipage{0.48\textwidth}
\centering
\begin{tikzpicture}[line cap=round,line join=round,>=triangle 45,x=1cm,y=1cm,scale=0.78]
\clip(-5,-3) rectangle (4,5.148);
\draw [->,line width=0.4pt] (-1.56,-1.99) -- (0.6,-2.01);
\draw [->,line width=0.4pt,color=qqqqff] (1.42,-1.73) -- (2.68,0.01);
\draw [->,line width=0.4pt,color=qqqqff] (2.76,0.75) -- (2.26,2.93);
\draw [->,line width=0.4pt] (1.86,3.43) -- (-0.18,4.41);
\draw [->,line width=0.4pt] (-0.86,4.45) -- (-2.92,3.39);
\draw [->,line width=0.4pt] (-3.36,3.01) -- (-3.78,0.89);
\draw [->,line width=0.4pt] (-3.76,-0.01) -- (-2.22,-1.79);
\draw [->,line width=0.4pt] (-0.5733333333333335,4.06) -- (-1.813333333333334,-1.393333333333337);
\draw [->,line width=0.4pt] (-2.84,2.9) -- (0.8266666666666671,-1.6066666666666705);
\draw [->,line width=0.4pt] (-3.42,0.37) -- (2.413333333333334,0.3533333333333304);
\draw [->,line width=0.4pt] (-1.64,-1.5) -- (1.9733333333333334,2.913333333333334);
\draw [->,line width=0.4pt] (1.0133333333333336,-1.5533333333333368) -- (-0.3866666666666667,4.046666666666665);
\draw [->,line width=0.4pt] (2.4666666666666677,0.5533333333333305) -- (-2.68,3.06);
\draw [->,line width=0.4pt] (1.853333333333334,3.0866666666666647) -- (-3.44,0.5666666666666638);
\draw [->,line width=0.4pt,dash pattern=on 3pt off 3pt] (-0.14666666666666672,4.086666666666665) -- (1.226666666666667,-1.4333333333333371);
\draw [->,line width=0.4pt,dash pattern=on 3pt off 3pt] (-2.52,3.246666666666665) -- (2.5066666666666673,0.78);
\draw [->,line width=0.4pt,dash pattern=on 3pt off 3pt] (-3.5066666666666677,0.82) -- (1.666666666666667,3.246666666666665);
\draw [->,line width=0.4pt,dash pattern=on 3pt off 3pt] (-2.093333333333334,-1.2466666666666701) -- (-0.8266666666666669,4.086666666666666);
\draw [->,line width=0.4pt,dash pattern=on 3pt off 3pt] (0.5857031888939774,-1.7079768998850313) -- (-3,2.7133333333333307);
\draw [->,line width=0.4pt,dash pattern=on 3pt off 3pt] (2.44,0.1) -- (-3.426666666666667,0.15333333333333035);
\draw [->,line width=0.4pt,dash pattern=on 3pt off 3pt] (2.08,2.6466666666666643) -- (-1.373333333333334,-1.58);


\begin{scriptsize}
\draw [fill=black] (-2,-2) circle (2pt);
\draw[color=black] (-2.066753785443086,-2.493890077466386) node {$\rho_3$};
\draw [fill=ududff] (1,-2) circle (2pt);
\draw[color=ududff] (1.2814402811950834,-2.4938861821625687) node {$\rho_4$};
\draw [fill=qqqqff] (2.8704694055762,0.34549444740408874) circle (2pt);
\draw[color=qqqqff] (3.321292837301929,0.074735242159865) node {$\rho_5$};
\draw [fill=qqqqff] (2.2029066037072575,3.2702781839495594) circle (2pt);
\draw[color=qqqqff] (2.613970977478715,3.530950149261156) node {$\rho_6$};
\draw [fill=uuuuuu] (-0.5,4.571929401302234) circle (2pt);
\draw[color=uuuuuu] (-0.3868643683629522,5.0400910138549) node {$\rho_0$};
\draw [fill=uuuuuu] (-3.2029066037072567,3.2702781839495603) circle (2pt);
\draw[color=uuuuuu] (-3.5387562258860175,3.530950149261156) node {$\rho_1$};
\draw [fill=uuuuuu] (-3.8704694055762006,0.34549444740408997) circle (2pt);
\draw[color=uuuuuu] (-4.483436806505735,0.074735242159865) node {$\rho_2$};
\end{scriptsize}
\end{tikzpicture}
\caption{\label{fig:quiver_1/7(1,3)} The case $\frac{1}{7}(1,3)$.} 
\endminipage
\end{figure}


\subsubsection{Global case} \label{ssec:global-case}
Suppose $X$ is a complex quasi-projective surface with isolated  
quotient singularities, i.e. every singular point of $X$ is \'etale-locally isomorphic to a quotient $\C^2/G$, where $G$ is a finite subgroup of $\GL(2,\C)$. 
Let $\tilde{X}$ be the minimal resolution of $X$, and denote by $\sX$ the canonical stack associated to $X$.
By \cite[Theorem 1.4]{IU15}, there is a fully faithful functor $\Phi\colon D(\tilde{X}) \to D(\sX)$.

Suppose that $X$ has only cyclic quotient singularities and consider the diagram 
\begin{equation*} 
\begin{tikzcd}
\mathcal \sZ \rar[phantom,"\subseteq"] & \tilde{X}\times \sX \rar{\pi_{\sX}} \dar{\pi_{\tilde{X}}} &  \sX \dar{s} \\
&\tilde{X} \rar{f} & X
\end{tikzcd}
\end{equation*}
where $\mathcal Z$ is the substack $(\tilde{X} \times_X \sX)_{\mathrm{red}}$. In this case, \cite[Prop. 8.1]{IU15} shows that the Fourier-Mukai functor $\Phi=\pi_{\sX *}(\pi_{\tilde{X}}^*(-)\otimes \sO_\mathcal Z)$ is fully faithful.
Over the smooth locus of $X$, $\mathcal Z$ is the graph of $s$. On a singular chart of $X$, $\mathcal Z$ restricts to the quotient $[Z/G]$ of the universal subscheme of diagram \eqref{eq:diagram}. This yields a semiorthogonal decomposition
\[D(\sX) = \left\langle \textbf{e}_1,...,\textbf{e}_N,\Phi(D(\widetilde{X})) \right\rangle,\]
where $(\textbf{e}_1,...,\textbf{e}_N)$ is an exceptional collection, containing the local contributions described in Section \ref{sssec:cyclic_singularities}.


\section{Computation of left adjoint functor}
\label{sec:AdjointComputation}
Consider the fully faithful functor $\phi \colon D(Y) \to D([\C^2/G])$ of Equation \eqref{eq:phi} and the semiorthogonal decomposition \eqref{eq:LocalSOD}. By Proposition \ref{prop:1.1}, $\phi(D(Y))^\perp$ is generated by simple sheaves of the form $\sO_0\otimes \rho$, where $\rho$ is nonspecial. In this Section we introduce the functor $\psi\colon D([\C^2/G])\to D(Y)$ left adjoint to $\phi$, and compute the images $\psi(\sO_0\otimes \rho)$ for all $\rho$.
\smallskip

We follow the same setup and notation of Section \ref{ssec:gl2}.
The left adjoint of $\phi$ is:
\begin{equation} 
\psi(-)= [\dR\pi_{Y *} \left( \pi_{\C^2 }^*(- \otimes K_{\C^2}) \otimes^{\dL} \sO_\sZ^\vee \right) ]^G  [2]
\end{equation}
where $K_{\C^2}=\sO_{\C^2}\otimes \rho_{\det}$ is the $G$-equivariant canonical sheaf of $\CC^2$.

\begin{thm}
\label{thm:psiO0rho_general}
For $i \in \{1, \dots, m\}$ let $-\alpha_i$ be the self intersection of the exceptional curve $E_i$. Then we have:
\begin{align*}
\psi(\sO_0\otimes \rho) \simeq \begin{cases}
\sO_E(E)[1]& \mbox{ if } \rho=\rdet^\ast \\
\sO_{E_{i}}(-\alpha_{i}+1) & \mbox{ if } \rho\otimes \rdet=\rho_i \mbox{ for some }1\leq i\leq m \\
0 & \mbox{ otherwise}.
\end{cases}
\end{align*}
\end{thm}

\begin{proof}
By definition of $\psi$ we have:
\begin{equation}
 \label{eq:psi0} \begin{split}
\psi(\mathcal{O}_0\otimes \rho ) &= \left[\dR\pi_{Y \ast} \left( \pi_{\C^2}^\ast(\mathcal{O}_0 ) \otimes^{\dL} \sO_\mathcal Z^\vee \otimes \rho \otimes \rdet [2] \right) \right]^G \\
&= \left[\dR\pi_{Y \ast}  \left(\mathcal{O}_{Y \times \{0\} } \otimes^{\dL} \sO_\mathcal Z^\vee \otimes  \rho \otimes \rdet [2] \right) \right]^G\\
&= \left[\dR\pi_{Y \ast}  \bigl(\dR \HHom \left(\sO_\mathcal Z, \mathcal{O}_{Y \times \{0\} }  \otimes  \rho \otimes \rdet [2] \right) \bigr) \right]^G
\end{split}
\end{equation}

Consider the resolution obtained from pulling back \eqref{eq:local_Koszul} to $Y\times \C^2$:
\[ 0 \to  \mathcal{O}_{Y \times \C^2} \otimes \rdet  \xrightarrow{(y,-x)^T}  \mathcal{O}_{Y \times \C^2} \otimes \rnat \xrightarrow{(x,y)} \mathcal{O}_{Y \times \C^2} \to \mathcal{O}_{Y \times \{0\} } \to 0.
\]
By Grothendieck--Verdier duality 
\begin{equation}\label{eq:GV}   \dR\pi_{Y \ast}  \dR \HHom \left(\sO_\sZ, \mathcal{O}_{Y \times \C^2 } \otimes  \rdet [2]\right)= \dR \HHom (  \dR\pi_{Y \ast}  \sO_\sZ, \sO_Y)
\end{equation}
Denote by $\sR$ the vector bundle $\sR=\dR\pi_{Y \ast}  \sO_\sZ$. 
Then the right hand side of the equation above is $\sR^\vee=\dR \HHom (  \sR, \sO_Y)$. Combining \eqref{eq:psi0} and \eqref{eq:GV} we find that 
\begin{equation} \label{eq:psi0prime}
\psi(\sO_0\otimes \rho)= \left[ \sR^\vee \otimes (\rdet \otimes \rho) \xrightarrow{}  \sR^\vee  \otimes (\rnat\otimes\rho) \xrightarrow{}   \sR^\vee \otimes \rho  \right]^G.
\end{equation}

The vector bundle  $\sR$ decomposes as a sum 
\begin{equation}\label{eq:DecompositionOfR}
\sR = \oplus_{\rho} \sR_\rho \otimes \rho
\end{equation} according to the decomposition of the regular representation. Moreover, for all $\rho\in \mathrm{Irrep}(G)$, we have $(\sR^\vee \otimes \rho)^G= (\sR_\rho^\vee\otimes \rho^*\otimes \rho) = \sR_\rho^\vee$. Then by exactness of the functor $(-)^G$:
\begin{equation}
\label{eq:our_sequences}
\psi(\sO_0\otimes \rho )= \left[\sR_{\rdet\otimes \rho}^\vee \xrightarrow{}  \sR_{\rnat\otimes\rho}^\vee \xrightarrow{}   \sR_{\rho}^\vee \right].
\end{equation}

By \cite[Corollary 3.2]{Ish02},  $\sR_\rho$ is the torsion free pull-back
$ \sR_\rho = \widetilde{M_\rho}$. Therefore, $\psi(\sO_0\otimes \rho)$ is dual to the torsion-free pull-back of $\widetilde{A}_{\rho\otimes \rdet}$:
\[\dR\HHom(\widetilde{A}_{{\rho\otimes \rdet}},\sO_Y)[2] = \psi(\sO_0\otimes \rho),\]
since all summands of $\sR$ are locally free. The statement then follows from Proposition \ref{prop:Ishii5.1}.
\end{proof}


\begin{remark} \label{rem:perv}
Let $\Coh_G(\C^2)_0$ be the finite length category of $G$-equivariant sheaves supported at $0\in\C^2$, and let $\!^0\mathrm{Per}(Y/X)$ be the category of $0$-perverse sheaves on $Y$ (as defined in \cite{Bri02_flops}). 

The complexes $\tilde{A}_\rho$, for $\rho$ special, are quasi-isomorphic to simple $(-1)$-perverse sheaves. Dualizing them (with the functor $\dR\HHom(-,K_Y)[2]$), we obtain simple $0$-perverse sheaves (\cite[Prop. 3.5.8]{VdB04}). 
Theorem \ref{thm:psiO0rho_general} shows that $\psi$ sends simple sheaves supported at $0$ to simple $0$-perverse sheaves twisted by $-K_Y$. 
It follows that $\psi\otimes K_Y$ restricts to an exact functor of abelian categories $\Coh_G(\C^2)_0\to \!^0\mathrm{Per}(Y/X)$. 
\end{remark}


\subsection{Alternative proof of Prop.~\ref{prop:1.1} (or \texorpdfstring{\cite[Prop. 1.1]{Ish02}}{[Ish02, Prop. 1.1]})}
\label{ssec:alternative_pf}
As an application of Theorem \ref{thm:psiO0rho_general}, we give an alternative proof of \cite[Prop. 1.1]{IU15}.
\smallskip

Let $\theta$ be the right adjoint to $\phi$:
\[ \begin{split} \theta (-)&:= [\dR\pi_{Y \ast} \left( \pi_{\CC^2 }^\ast(-) \otimes^{\dL} \sO_\sZ^\vee \otimes \pi_{Y} ^\ast(K_Y)[2] \right) ]^G\\
& \ = [\dR\pi_{Y \ast} \left( \pi_{\CC^2 }^\ast(-) \otimes^{\dL} \sO_\sZ^\vee  \right) ]^G \otimes K_Y[2]. 
 \end{split}\]
Note that \begin{equation}\label{eq:relBetweenAdjoints}
\theta(- \otimes K_{\CC^2})=\psi(-) \otimes K_Y.
\end{equation}

\begin{proof}[Proof of Prop. {\ref{prop:1.1}}]

Consider the semiorthogonal decomposition \eqref{eq:LocalSOD}.
Let $\sC=
\ker(\theta)$, i.e. the full triangulated subcategory of $D^b_G(\C^2)$ containing objects $C$ such that $\theta(C)=0$.
 By the adjunction  between $\phi$ and $\theta$ one has $\sC=\phi (D(Y))\!^\perp$, thus $\!^\perp\sC \simeq \phi (D(Y))$.

\paragraph{\textbf{Step 1:} } we show that the category $\sC$ is generated by the simple sheaves $\sO_0 \otimes \rho$, with $\rho$ not special.

Suppose $E$ is a complex of $G$-equivariant sheaves on $\C^2$ which is not supported on the origin. Then by \cite[Lemma 5.3]{BM02} there exists a point in $X\setminus \{0\}$, equivalently a free orbit $Gx$ in $\C^2$, and an integer $i$ such that 
\[ \Ext^i_{\C^2}(\sO_{Gx},E)^G\neq 0. \]
The functor $\phi$ maps points of $Y\setminus F$ to free orbits, so we may write $\sO_{Gx}\simeq \phi(\sO_p)$ for some $p\in Y\setminus F$. Then, 
\[ \Ext^i(\phi(\sO_p),E)^G= \Ext^i_{[\C^2/G]}(\phi(\sO_p),E) = \Ext^i_Y(\sO_p, \theta(E))\neq 0, \]
thus $E\notin \sC$. Therefore, all objects of $\sC$ are supported at the origin.

Theorem \eqref{thm:psiO0rho_general}, combined  with  \eqref{eq:relBetweenAdjoints}, implies that $\sO_0\otimes \rho\in \sC $ if and only if $\rho$ is not special.
Then sheaves in $\sC$ have a composition series with factors the non-special simples. Since $\theta$ is exact on $\Coh_G(\C^2)_0$ by Remark \ref{rem:perv} and \eqref{eq:relBetweenAdjoints}, a complex $E \in D_G(\C^2)$ lies in $\sC$ precisely if its cohomologies do, which proves the claim. 

\paragraph{\textbf{Step 2:}} we show that the sheaf $\sO_{\C^2}\otimes \rho$ is in the essential image of $\Phi$ for $\rho$ special. 

One readily checks that, for two irreducible representations $\sigma$ and $\tau$,
\begin{equation*}
\dR\Hom(\sO_{\C^2}\otimes \sigma, \sO_0 \otimes \tau)=0
\end{equation*}
unless $\sigma=\tau$ (see Remark \ref{rem:quiver}). Then by Step 1 we have that $\sO_{\C^2}\otimes \rho \in ~^\perp\sC = \im \Phi$ for $\rho$ special.

\paragraph{\textbf{Step 3:}} we show that the sheaves 
  $\sO_{\C^2}\otimes \rho$ for $\rho$ special generate the essential image of $\Phi$. 
  
 We have
\begin{equation}\label{eq:psi(Orho)}
\begin{split}
\psi(\sO_{\C^2}\otimes \rho) & = \left[ \dR\pi_{Y \ast}  \dR \HHom \left(O_\sZ, \mathcal{O}_{Y \times \C^2 } \otimes \rho\otimes  \rdet [2]\right) \right]^G\\
 &= \left[ \dR \HHom (  \dR\pi_{Y \ast}  \sO_\sZ, \sO_Y) \otimes\rho\right]^G \\
 &= \left[ \sR^\vee \otimes\rho\right]^G =\sR_\rho^\vee 
 \end{split}
\end{equation}
by \eqref{eq:GV} and \eqref{eq:DecompositionOfR}. Then by Step 2,  Equation \eqref{eq:psi(Orho)}, and the fact that 
\[
\phi \circ \psi_{|{\phi (D(\Coh(Y)))}}=\mathrm{id}_{\phi(D(\Coh(Y)))}
\] 
we get that $\Phi(\sR_\rho^\vee)=\sO_{\C^2}\otimes \rho$ for all $\rho$ special. Since the $\sR_\rho^\vee$ are generators for $D^b(Y)$ by \cite[Theorem B]{VdB04},  the special $\{\sO_{\C^2}\otimes \rho\}$ generate $\Phi(D^b(\Coh(Y)))$. \end{proof}


\subsection{Cyclic singularities case}
\label{ssec:CyclicSingAdjointComputation}
Suppose we are in the cyclic setting of Section \ref{sssec:cyclic_singularities}. Then we can give a more precise statement of Theorem \ref{thm:psiO0rho_general}:
\begin{proposition} \label{thm:PsiO0rhoi} 
We have:
\begin{align}
\psi(\sO_0\otimes \rho_i) \simeq \begin{cases}
\sO_E(E)[1] & \mbox{ if } a+1+i=0 \\
\sO_{E_{t}}(-\alpha_t+1) & \mbox{ if } a+1+i=i_t \in I(r,a)\setminus\{0,r\} \\
0 & \mbox{ otherwise}.
\end{cases}
\end{align}
\end{proposition}
\begin{proof} The proof goes as the proof of Theorem \ref{thm:psiO0rho_general}.  Recall that in this case $\rho_{\det}=\rho_{a+1}$, and that  Proposition \ref{prop:Ishii5.1-cyclic} computes the torsion-free pullbacks $\widetilde{A_l}$. 
\end{proof}



\begin{example}
Suppose $G=\frac{1}{5}(1,3)$. We have $\frac 53=[2,3]$ and $I(5,3)=\{ 5,3,1,0 \}$. The SOD \eqref{eq:sod} from \cite{IU15} reads
\[ D^b(\sX)=\left\langle G, \sO_0\otimes \rho_4 ,\phi(D^b(\Coh(Y)))  \right\rangle \]
where $G=\C[x,y]/(x,y^2)\otimes \rho_4$ fits in an extension $\sO_{0}\otimes \rho_2 \xrightarrow{y} G \to \sO_0\otimes \rho_4$.
By Proposition \ref{thm:PsiO0rhoi} there is a short exact sequence
\[ \sO_{E_2}(-2) \to \psi(G) \to \sO_{E_1}(-1). \]
It is straightforward to check that $\psi(G)$ can also be written as an extension of $\sO_p$ by $\sO_E(E)$, where $p\in E_2$. 
\end{example}
%


\begin{example}
\label{ex:1/n11ImagesOfPsi}
Let $G=\frac 1n (1,1)$, $n>1$ integer.  Let $E$ the single exceptional $(-n)$-curve in the minimal resolution of $X$ (see Example \ref{exa:1n11}). Then Theorem \ref{thm:PsiO0rhoi} reads
\begin{equation}
\psi(\sO_0\otimes \rho_i) \simeq \begin{cases}
0 & \mbox{ if } i=0,...,n-3\\
\sO_E(-n)[1] & \mbox{ if } i=n-2 \\
\sO_{E}(-n+1) & \mbox{ if } i=n-1.
\end{cases}
\end{equation} 
\end{example}

\begin{example}
\label{ex:1/(2k+1)1kImagesOfPsi}
Let $G=\frac{1}{2k+1}(1,k)$, with $k \geq 1$. Let $E_1, \dots, E_k$ be the $k$ exceptional curves in $Y$, where $E_1^2=-3$ and $E_i^2=-2$ for $i=2,...,k$ (see Example \ref{exa:12k+1}). The fundamental cycle is $E=\sum_{i=1}^k E_i$. 
Theorem \ref{thm:PsiO0rhoi} reads
\begin{equation}
\label{eq:thmPsiO0rhoi_2k+1}
\psi(\sO_0\otimes \rho_i) \simeq \begin{cases}
0 & \mbox{ if } i=0,...,k-1 \\
\sO_E(E)[1] & \mbox{ if } i=k \\
\sO_{E_{2k+1-i}}(-\alpha_{2k+1-i}+1) & \mbox{ if } i=k+1,...,2n.
\end{cases}
\end{equation} 
\end{example}
\smallskip


\subsubsection{Alternative (toric) approach to Proposition \ref{thm:PsiO0rhoi} }
\label{ssec:ToricApproach}

Here we present another possible approach, based on toric geometry, to the computation in Proposition \ref{thm:PsiO0rhoi}. \smallskip

First, we briefly recall how to describe the singularity $X=\frac 1r(1,a)$ and its minimal resolution 
torically. The singularity $X$ is the affine toric surface $U_\sigma$ associated to the cone 
 \begin{equation}   \label{eq:normal-form} \sigma=\langle  re_1 -a e_2, e_2  \rangle \subset N_{\RR}\simeq \RR^2 
\end{equation}
Its minimal resolution can be obtained by performing $n$ consecutive weighted blow-ups with weights $\frac{1}{{i}_{t-1}}(1, i_{t})$, where $t=1, \dots, n$. More explicitly, write $\rho_0=e_2$ and $\rho_{n+1}= re_1 -a e_2$. Then $Y$ is the smooth affine toric surface with fan $\Sigma$, where $\Sigma$ is the refinement of $\sigma$ obtained by adding the rays generated by the vectors \begin{equation} \rho_{t}=\frac{1}{{i}_{t-1}}\left(\rho_{n+1} +i_{t} \cdot \rho_{t-1}\right) \quad  t=1, \dots, n
\end{equation}
 Indeed, if ${\sigma}_{t-1}=\langle \rho_{n+1}, \rho_{t-1} \rangle$, the associated surface is the cyclic singularity $U_{{\sigma}_{t-1}}=\frac{1}{{i}_{t-1}}(1, i_{t})$.  
The surface associated to the refinement of ${\sigma}_{t-1}$ obtained by adding the ray $\rho_t$ has two charts: the chart corresponding to $\tau_t=\langle  \rho_{t-1}, \rho_{t} \rangle$ is $U_{\tau_t}=\C^2$, and the chart corresponding to ${\sigma}_{t}=\langle \rho_{n+1}, \rho_{t} \rangle$ is $U_{{\sigma}_{t}}=\frac{1}{{i}_{t}}(1, i_{t+1})$ if $t<n$, while is $\CC^2$  if $t=n$. 
Note that the vectors $\rho_t$ satisfy the relation
 $\rho_{t-1}+\rho_{t+1}=\alpha_t \cdot \rho_t $,
 which implies that the self-intersection of the exceptional curve $E_t$ associated to the ray $\rho_t$ is $-\alpha_t$ (see \cite[Theorem 10.4.4]{cox2011toric}).

The map $f \colon Y \to X$ can be described locally as follows. Let $u_t, v_t$ be the coordinates on $U_{\tau_t}$, and, for $t<n$, let $x_t, y_t$ be the coordinates on $\C^2$ such that  $U_{{\sigma}_{t}}=\Spec\CC[x_t,y_t]^G$, $G= \frac{1}{{i}_{t}}(1, i_{t+1})$. Then:
\begin{equation} \label{eq:blow-up maps}
\begin{cases}
x_{t-1}^{i_{t-1}}=u_t \\
y_{t-1}^{i_{t-1}}=u_t^{i_t}v_t^{i_{t-1}} \end{cases} \quad \text{and} \quad   \begin{cases}
x_{t-1}^{i_{t-1}} = x_t^{i_{t-1}} y_t \\
y_{t-1}^{i_{t-1}}= y_t^{i_t}
\end{cases} 
\end{equation}
For $t=n$, the same expressions hold,  $x_n, y_n$ being the coordinates on $U_{{\sigma}_{n}}=\C^2$.

\begin{remark}
Any 2-dimensional strictly convex cone can be written in the form \eqref{eq:normal-form}, called \emph{normal form}. 
There is an ambiguity for normal forms of 2-dimensional cones: if \mbox{$\sigma=\langle  re_1 -a e_2, e_2\rangle$} and $\sigma^\prime=\langle  r^\prime e_1^\prime -a^\prime e_2^\prime, e_2^\prime\rangle$ are lattice equivalent cones,
 then  $r=r^\prime$ and either $a=a^\prime$ or $aa^\prime =1$ mod $r$. In terms of the associated singularity, this corresponds to writing $X$ as either $\frac 1r(1,a)$ or $\frac 1r(a^\prime, 1)$.
Note that, if  $r/a=[\alpha_1,...,\alpha_n]$, then $r/a^\prime=[\alpha_n,...,\alpha_1]$, that is, 
  $E_1, \dots, E_n$ appear in reverse order when resolving $\frac{1}{r}(a^\prime,1)$ as indicated above. Define the $j$-series $J(r,a)$ to be the sequence $J(r,a) \coloneqq \{j_0, j_1, \dots, j_n,j_{n+1}\}$,
 where  $j_0=0, j_1=1$, and $j_t=j_{t-1}\alpha_{t-1}-j_{t-2}$. Then, $I(r, a^\prime)=\{j_{n+1}, j_n, \dots, j_1, j_0 \}$. \end{remark}\smallskip

The Picard group of the minimal resolution $Y$ is generated by divisors $D_1,...,D_n$ determined by the condition $D_i \cdot E_j = \delta_{ij}$.
For $t=1, \dots, n$, formula \eqref{eq:formula-cyclic} implies $\widetilde{M}_{\rho_{i_t}}=\mathcal{O}_Y(D_t)$.
Now let $l \in \{0, \dots, r-1\}$. By \cite[\S A.1, Theorem a)]{Wun88} the summand $\sR_{\rho_l}= \widetilde{M}_{\rho_{l}}$ appearing in \eqref{eq:DecompositionOfR} is the line bundle  $\sO_Y(\sum_i d^{(l)}_i D_i)$, where the $d^{(l)}_i\geq 0$ are uniquely determined by the I-series $I(r,a)$ by requiring 
\[
l=d_1 i_1+ \dots + d_n i_n  \quad \text{and} \quad 0 \leq \sum_{t>t_0} d_t i_t < i_{t_0} \  \text{for any $t_0$}.\]
We can then interpret the sequences \eqref{eq:our_sequences} by studying the $D_i$'s and the divisors associated to the invariant coordinates $x^r$ and $y^r$ on $X$, without recurring to Proposition \ref{prop:Ishii5.1}.

Consider the divisors on $X$ defined by  $x^r$ and $y^r$. By \eqref{eq:blow-up maps}, their strict transforms under \mbox{$f \colon Y \to X$} are nonreduced curves $A'=r A$, $B'=rB$; by construction, the curve $A$ (resp. $B$) meets the curve $E_n$ (resp. $E_1$) at exactly one point (see top of Figure \ref{fig:AandB}).
In what follows, we write the pull-backs $f^*(x^r)$, $f^*(y^r)$ and the $D_i$'s as linear combinations of $A,B$, and $E_1,...,E_n$.

For $s \in \{0,1,...,n,n+1\}$ define a sequence $H(s)\coloneqq\{ h^{(s)}_0,...,h^{(s)}_{n+1} \}$ by letting $h^{(s)}_s=0$, $h^{(s)}_{s-1}=h^{(s)}_{s+1}=1$ and 
\begin{equation*}
\begin{split}
h^{(s)}_t=h^{(s)}_{t+1}\alpha_{t+1} - h^{(s)}_{t+2} & \ \  \mbox{ for }0\leq t<s-1,\\
h^{(s)}_t=h^{(s)}_{t-1}\alpha_{t-1} - h^{(s)}_{t-2} & \ \  \mbox{ for }s+1<t \leq n+1.
\end{split}
\end{equation*} 
Observe that $H(0)=J(r,a)$ and $H(n+1)=I(r,a)$. 
More generally,  
the truncation $H{(s)}_{\geq  s}$ is the $j$-series $J(i_{s}, i_{s+1})$
and the truncation $H{(s)}_{\leq s}$ is the $i$-series $I(j_s, j_{s-1}^\prime)$, with $j_{s-1} j_{s-1}^\prime =1 \mod j_s$. 

By \eqref{eq:blow-up maps}, the terms of $H(0)$ and $H(n+1)$ are the coefficients of the pull-back of the Cartier divisors $(x^r=0)$ and $(y^r=0)$ to $Y$: 
\begin{equation*}
\begin{split}
(f^*(x^r)=0)= \sum_{t=1}^{n}h^{(0)}_tE_t + h^{(0)}_{n+1}A \qquad \
(f^*(y^r)=0)=h^{(n+1)}_0 B + \sum_{t=1}^n h^{(n+1)}_t E_t.
\end{split}
\end{equation*}
Moreover, one checks that there are numerical equivalences \footnote{In this setting, linear and numerical equivalence coincide. In fact, since we are in a toric setting, they coincide up to torsion of the divisor class group, which is zero because $Y$ is smooth. Again given the toric setting, it is not surprising that working up to equivalence does not suffice to compute the cohomologies of the sequences \eqref{eq:our_sequences}, which is why the expressions \eqref{eq:NumEquivClassesDi} are relevant to our computation.} 
\begin{equation}
\label{eq:NumEquivClassesDi}
D_s \equiv \sum_{t=s+1}^nh^{(s)}_t E_t +h^{(s)}_{n+1}A\equiv h^{(s)}_0 B + \sum_{t=1}^{s-1}h^{(s)}_tE_t  \qquad \qquad s=1,...,n. 
\end{equation}
An alternative way to obtain the expressions \eqref{eq:NumEquivClassesDi} is to look at the fibers of maps $Y\to \pr 1$ constructed as follows. For each $s$, contract all the exceptional curves  on $Y$ except for $E_s$,
to obtain a morphism $\tau_s\colon Y\to T_s$, where $T_s$ is the total space of an orbibundle over the projective line 
$\P(j_{s-1}, i_{s+1}) \simeq \pr 1$ (see Figure \ref{fig:AandB}). 
Project further  $\sigma_s \colon T_s \to \pr 1$.  
Then, by \eqref{eq:blow-up maps}, the two expressions for $D_s$ written in \eqref{eq:NumEquivClassesDi} are the fibers of $\sigma_s \circ \tau_s \colon Y \to \P^1$ over the points $(0:1)$ and $(1:0)$ in $\P^1$.

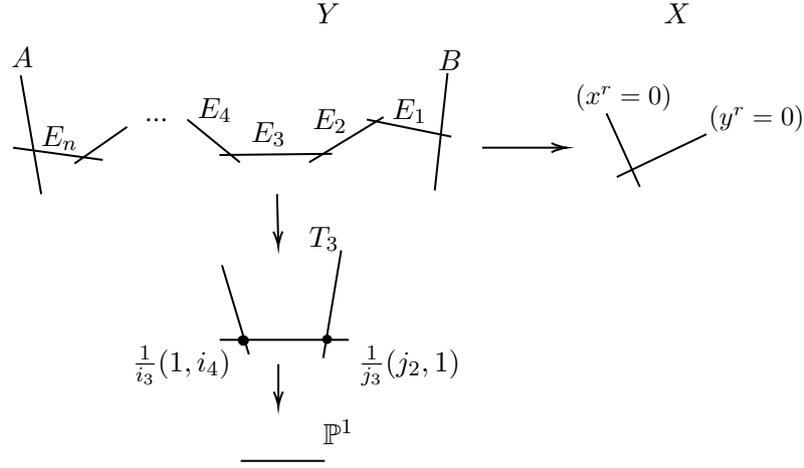
\begin{figure}[ht!]

\tikzset{every picture/.style={line width=0.75pt}} 

\begin{tikzpicture}[x=0.75pt,y=0.75pt,yscale=-0.7,xscale=0.7]
\clip(20,80) rectangle (700, 430);

\draw    (514.62,214.05) -- (490.85,161.32) ;
\draw    (562.15,176.26) -- (498.17,207.02) ;
\draw    (377.51,132.31) -- (368.37,216.69) ;
\draw    (73.14,134.07) -- (87.76,219.33) ;
\draw    (278.8,196.48) -- (331.81,163.95) ;
\draw    (214.82,191.2) -- (295.25,190.32) ;
\draw    (319.93,165.71) -- (380.26,178.02) ;
\draw    (191.96,165.71) -- (229.44,196.48) ;
\draw    (67.66,185.93) -- (131.64,194.72) ;
\draw    (111.53,197.35) -- (149.01,170.99) ;
\draw    (402.19,185.93) -- (460.52,185.93) ;
\draw [shift={(462.52,185.93)}, rotate = 180] [color={rgb, 255:red, 0; green, 0; blue, 0 }  ][line width=0.75]    (10.93,-3.29) .. controls (6.95,-1.4) and (3.31,-0.3) .. (0,0) .. controls (3.31,0.3) and (6.95,1.4) .. (10.93,3.29)   ;
\draw    (255.95,219.33) -- (256.81,255.12) ;
\draw [shift={(256.86,257.12)}, rotate = 268.61] [color={rgb, 255:red, 0; green, 0; blue, 0 }  ][line width=0.75]    (10.93,-3.29) .. controls (6.95,-1.4) and (3.31,-0.3) .. (0,0) .. controls (3.31,0.3) and (6.95,1.4) .. (10.93,3.29)   ;
\draw    (215,324) -- (307,324) ;
\draw    (216.64,270.31) -- (228.03,308.36) -- (235.84,334.47) ;
\draw    (301.65,259.76) -- (288.85,337.11) ;
\draw    (256.86,341.5) -- (256.86,370.27) ;
\draw [shift={(256.86,372.27)}, rotate = 270] [color={rgb, 255:red, 0; green, 0; blue, 0 }  ][line width=0.75]    (10.93,-3.29) .. controls (6.95,-1.4) and (3.31,-0.3) .. (0,0) .. controls (3.31,0.3) and (6.95,1.4) .. (10.93,3.29)   ;
\draw  [fill={rgb, 255:red, 0; green, 0; blue, 0 }  ,fill opacity=1 ] (228.73,323.49) .. controls (228.73,322.79) and (230.17,321.42) .. (231.93,321.42) .. controls (233.7,321.42) and (235.13,322.79) .. (235.13,324.49) .. controls (235.13,326.19) and (233.7,327.57) .. (231.93,327.57) .. controls (230.17,327.57) and (228.73,326.19) .. (228.73,324.49) -- cycle ;
\draw  [fill={rgb, 255:red, 0; green, 0; blue, 0 }  ,fill opacity=1 ] (288.85,323.75) .. controls (288.82,322.29) and (290.01,321.08) .. (291.53,321.05) .. controls (293.04,321.01) and (294.3,322.16) .. (294.34,323.62) .. controls (294.37,325.07) and (293.18,326.28) .. (291.66,326.32) .. controls (290.15,326.35) and (288.89,325.2) .. (288.85,323.75) -- cycle ;
\draw    (230,411) -- (290,411) ;

\draw (160,162.93) node [anchor=north west][inner sep=0.75pt]   [align=left] {...};
\draw (466.33,138.1) node [anchor=north west][inner sep=0.75pt]   [align=left, font=\small] {$(x^r=0)$};
\draw (561.76,153.26) node [anchor=north west][inner sep=0.75pt]   [align=left, font=\small ] {$(y^r=0)$};
\draw (529.64,80.12) node [anchor=north west][inner sep=0.75pt]   [align=left] {$X$};
\draw (281.94,80.12) node [anchor=north west][inner sep=0.75pt]   [align=left] {$Y$};
\draw (368.77,113.7) node [anchor=north west][inner sep=0.75pt]   [align=left] {$B$};
\draw (64.4,111.95) node [anchor=north west][inner sep=0.75pt]   [align=left] {$A$};
\draw (85.6,168.2) node [anchor=north west][inner sep=0.75pt]   [align=left] {$E_n$};
\draw (336.05,148.86) node [anchor=north west][inner sep=0.75pt]   [align=left] {$E_1$};
\draw (279.38,155.8) node [anchor=north west][inner sep=0.75pt]   [align=left] {$E_2$};
\draw (235.51,165.32) node [anchor=north west][inner sep=0.75pt]   [align=left] {$E_3$};
\draw (197.12,147.98) node [anchor=north west][inner sep=0.75pt]   [align=left] {$E_4$};
\draw (276.72,240.58) node [anchor=north west][inner sep=0.75pt]   [align=left] {$T_3$};
\draw (286.78,379.12) node [anchor=north west][inner sep=0.75pt]   [align=left] {$\P^1$};
\draw (150.73,326.41) node [anchor=north west][inner sep=0.75pt]   [align=left] {$\frac{1}{i_3}(1, i_4)$};
\draw (311.81,326.41) node [anchor=north west][inner sep=0.75pt]   [align=left] {$\frac{1}{j_3}(j_2,1)$};
\end{tikzpicture}
\caption{A picture of the surface $Y$ and the maps $ f \colon Y \to X$, $p_3 \circ \tau_3 \colon Y \to \P^1$. The toric surface $T_3$ has two singular points of type  $\frac{1}{i_3}(1, i_4)$ and $\frac{1}{j_3}(j_2,1)$.}
\label{fig:AandB}
\end{figure}

This setup can be used to interpret the maps in the complexes \eqref{eq:our_sequences}. Indeed, invariant functions pull-back to maps:
\[ \sO_Y \xrightarrow{f^*(x^r)} \sO_Y \ \ \qquad \ \ \sO_Y \xrightarrow{f^*(y^r)} \sO_Y \]
which must factor through rational maps
\begin{equation}
x_{(l)}\colon\sO_Y(\sum_i d^{(l)}_i D_i)\simeq\sR_l \xrightarrow{} \sR_{l+1}\simeq\sO_Y(\sum_i d^{(l+1)}_i D_i)
\end{equation} 
(and similarly for $y$) satisfying $(x_{(l)})\leq (f^*(x^r))$. This last condition uniquely determines a divisor in the linear series $|\sum_i d^{(l)}_i D_i|$, which allows one to compute the cohomoologies of \eqref{eq:our_sequences} without ambiguity: the maps $x_{(l)}$ and $y_{(l)}$ are precisely those appearing in \eqref{eq:our_sequences} (after taking duals). We illustrate this in an example.


\begin{example}
Suppose $G=\frac 15(1,2)$. The continued fraction expansion is $\frac 73  = [3,2]$ and the $i$-series $I(7,3)=\{ 5,2,1,0 \}$. Equation  \eqref{eq:thmPsiO0rhoi_2k+1} reads
\begin{equation}
\psi(\sO_0\otimes \rho_i) \simeq \begin{cases}
0 & \mbox{ if } i=0,1, \\
\sO_F(F)[1] & \mbox{ if } i=2 \\
\sO_{E_{5-i}(-\alpha_{5-i} + 1) } & \mbox{ if } i=3,4.
\end{cases}
\end{equation} 
We verify it using the approach of Section \ref{ssec:ToricApproach}. 
Applying \cite[\S A.1]{Wun88} we have 
\[ \sR_i=\begin{cases}
\sO_Y & i=0 \\
\sO_Y(D_{3-i}) & i=1,2 \\
\sO_Y(D_1+D_{5-i}) & i=3,4
\end{cases} \]

From Section \ref{ssec:ToricApproach} we get the following equalities:
\begin{align*}
(f^* x^5=0) = E_1 + 3E_2 +5A,  & &  D_1 \equiv  B \equiv  E_2 + 2A, \\
 (f^* y^5=0) = 5B + 2E_1 + E_2, & &  D_2 \equiv 3B + E_1 \equiv  A. 
\end{align*}

The maps $x_{(l)}$ are uniquely determined by the above. Indeed, denote by $a,b,e_{1},e_2$ the sections defining the divisors $A,B,E_1,E_2$. Then we have:
\begin{align*}
x_{(0)}\colon &\sR_0 \xrightarrow{a} \sR_1 & & y_{(0)} \colon \sR_0 \xrightarrow{b} \sR_2\\
x_{(1)}\colon &\sR_1 \xrightarrow{e_2a} \sR_2 & & y_{(2)} \colon \sR_2 \xrightarrow{b} \sR_4\\
x_{(2)}\colon &\sR_2 \xrightarrow{a} \sR_3 & & y_{(4)} \colon \sR_4 \xrightarrow{be_1} \sR_1\\
x_{(3)}\colon &\sR_3 \xrightarrow{e_2a} \sR_4 & & y_{(1)} \colon \sR_1 \xrightarrow{b} \sR_3\\
x_{(4)}\colon &\sR_4 \xrightarrow{e_1e_2a} \sR_0 &  & y_{(3)} \colon \sR_3 \xrightarrow{be_1e_2} \sR_0,\\
\end{align*}
so that $f^*x^5 = x_{(4)} x_{(3)} x_{(2)} x_{(1)} x_{(0)}$ (and similarly for $y$).
The dual maps $x_{(l)}^\vee$ (resp. $y_{(l)}^\vee$) to  $x_{(l)}$ (resp $y_{(l)}$) correspond to multiplication by the same sections. Then, consider for example the sequence:
\begin{align*}
\psi(\sO_0\otimes \rho_2)=\left[\sR_{0}^\vee \xrightarrow{(-y_{(3)}^\vee, x_{(4)}^\vee)^T}  \sR_{3}^\vee  \oplus \sR_{4}^\vee \xrightarrow{(x_{(2)}^\vee,y_{(2)}^\vee)}   \sR_2^\vee \right].
\end{align*}
The kernel of the map $(x_{(2)}^\vee,y_{(2)}^\vee)=(a,b)$ is 
\[\sO_Y(-2D_1-D_2) \xrightarrow{(-b,a)^T}  \sR_{3}^\vee  \oplus \sR_{4}^\vee.\]
We have  $(-y_{(3)}^\vee, x_{(4)}^\vee)^T=e_1e_2(-b,a)^T$, so
\[\psi(\sO_0\otimes \rho_2) = \coker\left( \sR_0 \xrightarrow{e_1e_2} \sO(-2D_1-D_2)  \right)[1] \simeq \sO_F(F)[1]. \] 

Similarly, consider:
\begin{align*}
\psi(\sO_0\otimes \rho_3)=\left[\sR_{1}^\vee \xrightarrow{(-y_{(4)}^\vee, x_{(0)}^\vee)^T}  \sR_{4}^\vee  \oplus \sR_{0}^\vee \xrightarrow{(x_{(3)}^\vee,y_{(3)}^\vee)}   \sR_3^\vee \right].
\end{align*}
with $(-y_{(4)}^\vee, x_{(0)}^\vee)^T=(-be_1,a)^T$ and $(x_{(3)}^\vee,y_{(3)}^\vee)=(e_2a,be_1e_2)$. This shows that the complex is exact in the middle. However, the cokernel of $(-be_1,a)^T$ is $(a,be_1)$, and hence
\[ \psi(\sO_0\otimes \rho_3) = \coker\left( \sO(-2D_1+D_2) \xrightarrow{e_2} \sR_3^\vee  \right) \simeq \sO_{E_2}(-1). \]
One proves the remaining cases in the same way.

\end{example}




\section{The minimal resolution of the Johnson-Koll\'ar surfaces} \label{sec:min-res}

For $k >0$ integer, let 
$X^{(k)}$ 
be a sufficiently general log del Pezzo surface 
of the family \eqref{eq:series}.
The surface $X^{(k)}$ is the vanishing locus of a generic section of the sheaf $\mathcal{O}_\P(8k+4)$ on $\P=\mathbb{P}(2,2k+1, 2k+1,4k+1)$. 
Let \mbox{$f^{(k)} \colon \widetilde{X}^{(k)} \to X^{(k)} $} be the minimal resolution. 
In this Section we describe in details the geometry of the surface $\widetilde{X}^{(k)} $. We will omit the superscript $- ^{(k)}$ from our notation whenever the context allows it.

\subsection{Basic properties} \label{ssec:basic}
Fix $k>0$ integer. We denote by $x, y_1, y_2, z$ the weighted homogeneous coordinates of $\P$.
The singular locus $\P_\mathrm{sing}$ of  $\mathbb{P}$ has three components:
\begin{itemize}
\item[1.] the point $(1:0:0:0)$, which is a singularity of type $1/2(1,1,1)$;
\item[2.] the line $(x=z=0)$, whose points all have a neighbourhood analytically isomorphic to $\mathbb{A}^1 \times 1/(2k+1)(1,k)$;
\item[3.] the point $(0:0:0:1)$, which is a singularity of type $1/(4k+1)(4,1,1)$.
\end{itemize}
\smallskip

It is easy to check that a section of the sheaf $\mathcal{O}_{\P}(8k+4)$ is of the form: 
\begin{equation}
\label{eq:generic_X8k+4}	
h^{(k)}(x,y_1,y_2,z)=a x^{4k+2}+h_4(y_1, y_2)+ b xz^2+x^{2k+1}h_2(y_1, y_2)+x^{k+1}h_1(y_1,y_2)z
\end{equation}
where $a,b \in \CC$ and $h_i(y_1, y_2)$ is a homogeneous polynomial of degree $i$ in the variables $y_1, y_2$. 

Let $X$ be the vanishing locus of $h=h^{(k)}$\footnote{
If $h$ is irreducible, then $X$ is wellformed. Indeed, if $h$ is irreducible, then $h_4(y_1,y_2) \neq 0$, that is, $X$ does not contain the line $(x=z=0)$. 
By \cite[Corollary 8.5]{IF}, $X$ is quasismooth if and only if $a \neq 0$, $b \neq 0$, and  for $i=1,2$ the polynomial $h_4(y_1, y_2)$ has a monomial of the form $y_i^m y_j$, $m>0$, $j \in \{1,2\}$ (these conditions imply that $h$ is irreducible).\label{footnote}}.
For a generic choice of $h$, $(1:0:0:0) \notin X$, and $(0:0:0:1) \in X$ is a singularity of type $1/(4k+1)(1,1)$. Moreover, since $h(0, y_1, y_2, 0)=h_4(y_1, y_2)$, 
$X$ intersects the line $(x=z=0)$ in four points, each of which is a singularity of type $1/(2k+1)(1,k)$. We write $p=(0:0:0:1)$ and denote by $p_1, p_2, p_3, p_4$ the four points in $X \cap (x=z=0)$.

By Equation \eqref{eq:generic_X8k+4} the curve $(x=0) \cap X$ is the union of four lines $c_{i} \in |\mathcal{O}_\P(2)| \oplus |\mathcal{O}_\P(2k+1)|$, $i=1,2,3,4$; for each $i$, the line $c_i$ passes through the points $p_i$ and  $p$ (see the bottom left of Figure \ref{fig:surfaces}). 

\begin{remark} \label{rem:min-names}
By the general theory of surface cyclic quotient singularities, 	the surface $\widetilde{X}$ is obtained from $X$ by a sequence of $1+4k$ weighted blow-ups: a single weighted blow-up at $p$, and a sequence of $k$ weighted blow-ups centered at each $p_i$ (see Examples \ref{exa:1n11} and \ref{exa:12k+1}). 

We denote by $F$ the $-(4k+1)$- exceptional curve over $p$, and for $i=1,2,3,4$  we denote by $C_{i,0}$ the $(-3)$-exceptional curve, and by $C_{i,j}$,  $j=1, \dots, k-1$, the $(-2)$-exceptional curves over $p_i$.   
Moreover, for $i=1,2,3,4$ we denote by $C_{i,k}$ the strict transform in $\widetilde{X}$ of the line $c_i \subset X$  (see the top of Figure \ref{fig:surfaces}).  
\end{remark} 

\begin{remark} \label{rem:canonical-min}
  The minimal resolution $f \colon \widetilde{X} \to X$ is noncrepant. 
 More precisely, the singularities of $X$ are strictly log-terminal: 
 \begin{equation} K_{\widetilde{X}}= f^\ast(K_X)+ d \cdot F+ \sum_{i=1}^4 \sum_{j=1}^k d_j \cdot C_{ij}
\end{equation}  where
\begin{equation}d=-\frac{4k-1}{4k+1} \quad \text{and} \quad  d_j= -\frac{k-(j-1)}{2k+1}
\end{equation}	
\end{remark}

\begin{remark}[Rationality and  existence of exceptional collections]
\label{rem:rationality} Note that, for any log del Pezzo surface $X$, its minimal resolution $\widetilde{X}$ is rational.
Indeed, by Castelnuovo's criterion \cite[Theorem 3.2]{CKS}, it suffices to show that
\[ H^1\left(\widetilde{X},\sO_{\widetilde{X}}\right)=0, \quad \mbox{and} \quad H^0\left(\widetilde{X}, \sO_{\widetilde{X}}(2K_{\widetilde{X}})\right)=0. \]

By the Leray spectral sequence (and since $X$ has rational singularities), we have $H^1(\widetilde{X},\sO_{\widetilde{X}})=H^1\left(X,\sO_{X}\right)$, which vanishes by the Kawamata-Viehweg vanishing theorem \cite[Theorem 9.1.18]{Laz04_PosII} combined with the del Pezzo assumption. 
For the second vanishing, we may argue as follows.
Since the singularities are log-terminal, we may write 
\[ f^*K_X = K_{\widetilde{X}} + \Delta \]
where $\Delta$ is effective. Since $-f^*K_X$ is big, so is $-K_{\widetilde{X}}= -f^*K_X + \Delta$. This forces the vanishing of all groups $H^0(\widetilde{X}, \sO_{\widetilde{X}}(mK_{\widetilde{X}}))$, $m>0$. Thus  $\widetilde{X}$ is a smooth rational surface.

It follows that both $\widetilde{X}$ and the canonical stack $\sX$ admit full exceptional collections. 
Indeed,  $\widetilde{X}$ is obtained through a sequence of blow-ups from a minimal rational surface (i.e. $\pr 2$, $\pr 1 \times \pr 1$, or $\mathbb{F}_n$ with $n\geq 2$). Since minimal rational surfaces admit full exceptional collections, so does $\widetilde{X}$ by Orlov's formula \cite{Orl92}. Finally, in virtue of the decomposition \eqref{eq:sod}, $D(\sX)$ admits a full exceptional collection if $D(\widetilde{X})$ does. 
\end{remark}

\subsection{Explicit geometry of the minimal resolution}
Fix a generic polyonomial $h=h^{(k)}$, equivalently, 
$a,b\in \C\setminus \{0\}$, and polynomials $h_4,h_2,h_1$ as in Equation \eqref{eq:generic_X8k+4}. 
Let $X$ be the corresponding log del Pezzo surface and $ f \colon \widetilde{X} \to X$ its minimal resolution. Let $X^\prime  \subset \P(2,1,1,1)$ be the smooth del Pezzo surface of degree two defined by the polynomial $h^{(0)}$, i.e.  
\begin{equation}
    \label{eq:h0}h^{(0)}(x,y_1,y_2,z)=a x^{2}+h_4(y_1, y_2)+ x\left(b z^2+h_2(y_1, y_2)+h_1(y_1,y_2)z\right)
\end{equation}
where $x,y_1, y_2, z$ are now weighted homogeneous coordinates on $\P(2,1,1,1)$.
The hyperplane \mbox{$(x=0) \subset \P(2,1,1,1)$} intersects $X'$ in four lines $c_{i,0}$, $i=1,2,3,4$, which meet at the point $q\coloneqq(0:0:0:1)$.
The next Theorem shows that $\widetilde{X}$ is obtained from $X^\prime$ by a sequence of $4k+1$ smooth blow-ups.

\begin{thm} \label{thm:min-resol} 
There is a birational morphism $\tau \colon \widetilde{X}  \to X^\prime$, which contracts the $4k+1$ curves $C_{i,k}, \dots, C_{i,1},$ $i=1,2,3,4$, and $F$ to the point $q$, as pictured in  Figure \ref{fig:widetildeX-Xprime}.
\end{thm}

\begin{figure}[ht!]
\tikzset{every picture/.style={line width=0.75pt}} 

\begin{tikzpicture}[x=0.75pt,y=0.75pt,yscale=-0.74,xscale=0.74]

\draw    (433.2,146.23) -- (433.5,229) ;
\draw    (474.19,140.54) -- (472.26,229.39) ;
\draw    (449.44,135.98) -- (455.5,226) ;
\draw    (494.42,141.82) -- (493.55,227.06) ;
\draw [color={rgb, 255:red, 65; green, 117; blue, 5 }  ,draw opacity=1 ][line width=1.5]    (416.18,218) .. controls (447.12,183.83) and (483.47,235.09) .. (514.41,200.91) ;
\draw    (553.02,141.68) -- (610.61,226.41) ;
\draw    (645.87,157.12) -- (569.82,214.75) ;
\draw    (620.13,137.57) -- (580.73,222.98) ;
\draw    (586.6,130.36) -- (594.5,241.51) ;

\draw [color={rgb, 255:red, 65; green, 117; blue, 5 }  ,draw opacity=1 ][line width=1.5]    (302.52,155.67) -- (290.5,233.09) ;
\draw [color={rgb, 255:red, 65; green, 117; blue, 5 }  ,draw opacity=1 ][line width=1.5]    (344.06,142.72) -- (331.69,234.76) ;
\draw [color={rgb, 255:red, 65; green, 117; blue, 5 }  ,draw opacity=1 ][line width=1.5]    (323.45,148.16) -- (312.5,229.66) ;
\draw [color={rgb, 255:red, 65; green, 117; blue, 5 }  ,draw opacity=1 ][line width=1.5]    (366.06,142.45) -- (355.21,236) ;
\draw [color={rgb, 255:red, 65; green, 117; blue, 5 }  ,draw opacity=1 ][line width=1.5]    (272.49,224.43) .. controls (306.4,194.15) and (346.24,239.57) .. (380.16,209.29) ;
\draw    (293.78,128.73) -- (301.93,181.21) ;
\draw    (335.37,111.57) -- (342.83,181.71) ;
\draw    (315.76,110.44) -- (324.3,186.22) ;
\draw    (359.19,113.82) -- (366.82,182.63) ;
\draw [color={rgb, 255:red, 65; green, 117; blue, 5 }  ,draw opacity=1 ][line width=1.5]    (15.5,311) .. controls (31.47,284.16) and (171.53,322.87) .. (187.5,296.03) ;
\draw    (56.96,23.32) -- (73.5,87) ;
\draw [color={rgb, 255:red, 65; green, 117; blue, 5 }  ,draw opacity=1 ][line width=1.5]    (33.21,215.88) -- (22.76,267.75) ;
\draw [color={rgb, 255:red, 65; green, 117; blue, 5 }  ,draw opacity=1 ][line width=1.5]    (60.3,101.98) -- (72.5,154) ;
\draw [color={rgb, 255:red, 65; green, 117; blue, 5 }  ,draw opacity=1 ][line width=1.5]    (74.01,68.26) -- (58.5,117) ;
\draw [color={rgb, 255:red, 65; green, 117; blue, 5 }  ,draw opacity=1 ][line width=1.5]    (23.8,246.4) -- (31.2,318.09) ;
\draw [color={rgb, 255:red, 65; green, 117; blue, 5 }  ,draw opacity=1 ][line width=1.5]    (68.5,254.45) -- (75.07,295.27) -- (77.92,312.86) ;
\draw [color={rgb, 255:red, 65; green, 117; blue, 5 }  ,draw opacity=1 ][line width=1.5]    (78.9,215.3) -- (66.86,273.6) ;
\draw [color={rgb, 255:red, 65; green, 117; blue, 5 }  ,draw opacity=1 ][line width=1.5]    (114.7,242.96) -- (122.13,314.83) ;
\draw [color={rgb, 255:red, 65; green, 117; blue, 5 }  ,draw opacity=1 ][line width=1.5]    (127.58,217.59) -- (109.82,270.69) ;
\draw [color={rgb, 255:red, 65; green, 117; blue, 5 }  ,draw opacity=1 ][line width=1.5]    (173.5,218) -- (153.1,287.52) ;
\draw [color={rgb, 255:red, 65; green, 117; blue, 5 }  ,draw opacity=1 ][line width=1.5]    (157.35,253.46) -- (164.74,330.13) ;
\draw  [dash pattern={on 0.84pt off 2.51pt}]  (122.26,170.12) -- (122.26,204.62) ;
\draw  [dash pattern={on 0.84pt off 2.51pt}]  (161.74,173.88) -- (162.69,208.38) ;
\draw  [dash pattern={on 0.84pt off 2.51pt}]  (80.27,172.5) -- (81.23,207) ;
\draw  [dash pattern={on 0.84pt off 2.51pt}]  (29.85,172.32) -- (30.81,206.82) ;
\draw    (515.32,175.14) -- (537.06,175.9) ;
\draw [shift={(539.06,175.97)}, rotate = 182] [color={rgb, 255:red, 0; green, 0; blue, 0 }  ][line width=0.75]    (10.93,-3.29) .. controls (6.95,-1.4) and (3.31,-0.3) .. (0,0) .. controls (3.31,0.3) and (6.95,1.4) .. (10.93,3.29)   ;
\draw    (390.37,175.22) -- (400.51,175.31) -- (410.72,175.39) ;
\draw [shift={(412.72,175.41)}, rotate = 180.49] [color={rgb, 255:red, 0; green, 0; blue, 0 }  ][line width=0.75]    (10.93,-3.29) .. controls (6.95,-1.4) and (3.31,-0.3) .. (0,0) .. controls (3.31,0.3) and (6.95,1.4) .. (10.93,3.29)   ;
\draw    (253.77,174.09) -- (275.4,174.09) ;
\draw [shift={(277.4,174.09)}, rotate = 180] [color={rgb, 255:red, 0; green, 0; blue, 0 }  ][line width=0.75]    (10.93,-3.29) .. controls (6.95,-1.4) and (3.31,-0.3) .. (0,0) .. controls (3.31,0.3) and (6.95,1.4) .. (10.93,3.29)   ;
\draw [line width=0.75]  [dash pattern={on 0.84pt off 2.51pt}]  (243.84,174.09) -- (232.61,174.09) ;
\draw    (190.5,175.19) -- (214.88,174.18) ;
\draw [shift={(216.88,174.09)}, rotate = 177.62] [color={rgb, 255:red, 0; green, 0; blue, 0 }  ][line width=0.75]    (10.93,-3.29) .. controls (6.95,-1.4) and (3.31,-0.3) .. (0,0) .. controls (3.31,0.3) and (6.95,1.4) .. (10.93,3.29)   ;
\draw    (19.96,25.32) -- (36.5,89) ;
\draw [color={rgb, 255:red, 65; green, 117; blue, 5 }  ,draw opacity=1 ][line width=1.5]    (23.3,103.98) -- (35.5,156) ;
\draw [color={rgb, 255:red, 65; green, 117; blue, 5 }  ,draw opacity=1 ][line width=1.5]    (37.01,70.26) -- (21.5,119) ;
\draw    (99.96,23.32) -- (116.5,87) ;
\draw [color={rgb, 255:red, 65; green, 117; blue, 5 }  ,draw opacity=1 ][line width=1.5]    (103.3,101.98) -- (115.5,154) ;
\draw [color={rgb, 255:red, 65; green, 117; blue, 5 }  ,draw opacity=1 ][line width=1.5]    (118.01,71.26) -- (102.5,120) ;
\draw    (140.96,26.32) -- (157.5,90) ;
\draw [color={rgb, 255:red, 65; green, 117; blue, 5 }  ,draw opacity=1 ][line width=1.5]    (144.3,104.98) -- (156.5,157) ;
\draw [color={rgb, 255:red, 65; green, 117; blue, 5 }  ,draw opacity=1 ][line width=1.5]    (158.01,71.26) -- (142.5,120) ;

\draw (108.84,-3) node [anchor=north west][inner sep=0.75pt]   [align=left] {$\widetilde{X}$};
\draw (583.78,78) node [anchor=north west][inner sep=0.75pt]   [align=left] {$X^\prime$};
\draw (598,195) node [anchor=north west][inner sep=0.75pt]   [align=left, font=\tiny] {$q$};
\draw (620,140) node [anchor=north west][inner sep=0.75pt]   [align=left, font=\tiny] {$c_{3,0}$};
\draw (631,165) node [anchor=north west][inner sep=0.75pt]   [align=left, font=\tiny] {$c_{4,0}$};
\draw (177.57,311.28) node [anchor=north west][inner sep=0.75pt]   [align=left, font=\tiny] {F};
\draw (545,125) node [anchor=north west][inner sep=0.75pt]   [align=left, font=\tiny] {$c_{1,0}$};
\draw (590,125) node [anchor=north west][inner sep=0.75pt]   [align=left, font=\tiny] {$c_{2,0}$};
\draw (76.45,280) node [anchor=north west][inner sep=0.75pt]   [align=left, font=\tiny] {$C_{2,k}$};
\draw (115,130.28) node [anchor=north west][inner sep=0.75pt]   [align=left, font=\tiny] {$C_{3,2}$};
\draw (157,130.28) node [anchor=north west][inner sep=0.75pt]   [align=left, font=\tiny] {$C_{4,2}$};
\draw (30,90) node [anchor=north west][inner sep=0.75pt]   [align=left, font=\tiny] {$C_{1,1}$};
\draw (111,45) node [anchor=north west][inner sep=0.75pt]   [align=left, font=\tiny] {$C_{3,0}$};
\draw (66,90) node [anchor=north west][inner sep=0.75pt]   [align=left, font=\tiny] {$C_{2,1}$};
\draw (35,130.28) node [anchor=north west][inner sep=0.75pt]   [align=left, font=\tiny] {$C_{1,2}$};
\draw (123.69,280) node [anchor=north west][inner sep=0.75pt]   [align=left, font=\tiny] {$C_{3,k}$};
\draw (162,280) node [anchor=north west][inner sep=0.75pt]   [align=left, font=\tiny] {$C_{4,k}$};
\draw (111,90) node [anchor=north west][inner sep=0.75pt]   [align=left, font=\tiny] {$C_{3,1}$};
\draw (155,90) node [anchor=north west][inner sep=0.75pt]   [align=left, font=\tiny] {$C_{4,1}$};
\draw (29,45) node [anchor=north west][inner sep=0.75pt]  [rotate=-359.17,font=\tiny] [align=left] {$C_{1,0}$};
\draw (150,45) node [anchor=north west][inner sep=0.75pt]  [rotate=-359.17, font=\tiny] [align=left] {$C_{4,0}$};
\draw (70,45) node [anchor=north west][inner sep=0.75pt]  [rotate=-359.17, font=\tiny] [align=left] {$C_{2,0}$};
\draw (31,280) node [anchor=north west][inner sep=0.75pt]   [align=left, font=\tiny] {$C_{1,k}$};
\draw (30.7,220) node [anchor=north west][inner sep=0.75pt]  [font=\tiny,rotate=-359.17] [align=left] {$C_{1,k-1}$};
\draw (77,220) node [anchor=north west][inner sep=0.75pt]  [font=\tiny,rotate=-359.17] [align=left] {$C_{2,k-1}$};
\draw (124.61,220) node [anchor=north west][inner sep=0.75pt]  [font=\tiny,rotate=-359.17] [align=left] {$C_{3,k-1}$};
\draw (171.56,220) node [anchor=north west][inner sep=0.75pt]  [font=\tiny,rotate=-359.17] [align=left] {$C_{4,k-1}$};
\draw (70,130.28) node [anchor=north west][inner sep=0.75pt]   [align=left, font=\tiny] {$C_{2,2}$};

\end{tikzpicture}
\caption{A picture of the morphism $\tau \colon \widetilde{X} \to X^\prime$. 
The surface $\widetilde{X}$ is obtained from the surface $X^\prime$ by a sequence of $4k+1$ smooth blow-ups. 
}
\label{fig:widetildeX-Xprime}

\end{figure}
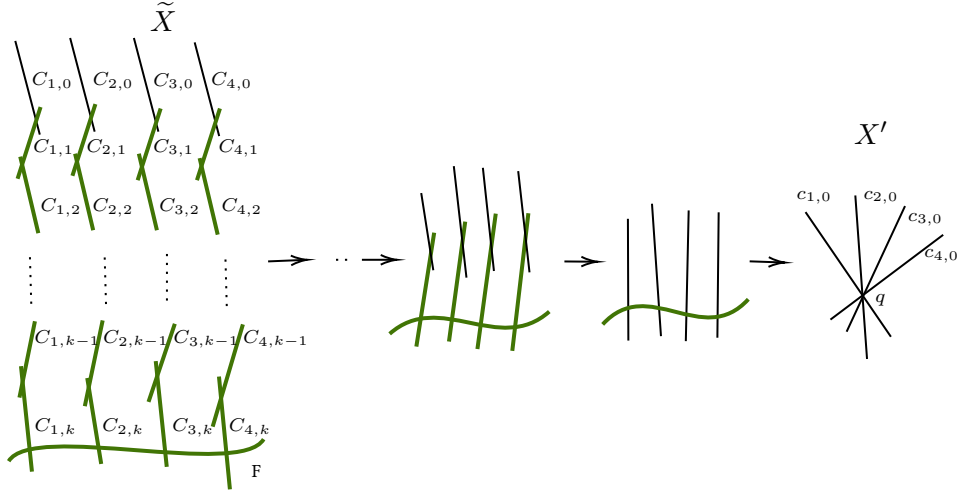

\begin{proof}
Let $\Delta \subset N_\RR \simeq \RR^3$ be the fan of $\P$ and let $\rho_1, \rho_2, \rho_3, \rho_4$ be the primitive generators of the rays of $\Delta$. 
One may choose for instance: 
\[ \rho_1=(k,k,4k+1)^T \quad  \rho_2=(1,0,0)^T \quad   \rho_3=(0,1,0)^T \quad   \rho_4=(-1,-1,-2)^T
\]

Consider the star subdivision $\Sigma=\Delta^\ast(v)$, where
\begin{equation} \label{eq:v}
v=\frac{1}{4k+1}\left( 4\rho_1 +\rho_2+\rho_3\right)		
\end{equation}
 The toric variety $G=G_{\Sigma}$ is the weighted blow-up of $\P$ at $p$ with weights $1/(4k+1)(4,1,1)$. 
 Note that the ray $\rho_1$ belongs to the cone generated by $\rho_4$ and $v$ (see Figure \ref{fig:fan-vectors}); more precisely,
\[ \rho_1= \frac{1}{2}(\rho_4 +(2k+1)v)
\]

Let $Y\subset G$ be the strict transform of $X$.  The exceptional divisor of the morphism $ G \to \P$ is the torus-invariant divisor $D_{v} \subset G$. 
By construction, 
the exceptional curve $D_v \cap Y \subset Y$ over the point $p \in X$ is a $-(4k+1)$-curve. The line $D_{\rho_1} \cap D_{\rho_4} \subset G$ is a line of $1/(2k+1)(1,k)$ points, and the surface $Y$ intersects this line in four points (see the left of Figure \ref{fig:surfaces}).

Now consider 
the sequence of $k$ star-subdivisions of the fan $\Sigma$ defined by setting $\Sigma_0=\Sigma^\ast(v_0) $, where 
\begin{equation} \label{eq:v0}
	 v_0=\frac{1}{2k+1}(\rho_1 +k \rho_4)
\end{equation} and, for $i=1, \dots, k-1$,  $\Sigma_i=\Sigma_{i-1}^\ast(v_i) $, where 
\begin{equation}\label{eq:vi}
	 v_i=\frac{1}{k-(i-1)}(\rho_1+(k-i) v_{i-1})
\end{equation}
and let $G_i=G_{\Sigma_i}$ (see Figure \ref{fig:fan-vectors}). The variety $G_0=G_{\Sigma_0}$ is the weighted blow-up of $G$ along the line $D_{\rho_1} \cap D_{\rho_4}$ with weights $1/(2k+1)(1,k)$. For $i=1, \dots, k-1$, the line $D_{\rho_1} \cap D_{v_{i-1}} \subset G_{i-1}$ is a line of $\frac{1}{k-(i-1)}(1,k-i)$ points, and the variety $G_i$ is the weighted blow-up of $G_{i-1}$ along this line.

\definecolor{xdxdff}{rgb}{0.49019607843137253,0.49019607843137253,1}
\definecolor{zzttqq}{rgb}{0.6,0.2,0}
\definecolor{ccqqqq}{rgb}{0.8,0,0}
\definecolor{qqqqff}{rgb}{0,0,1}
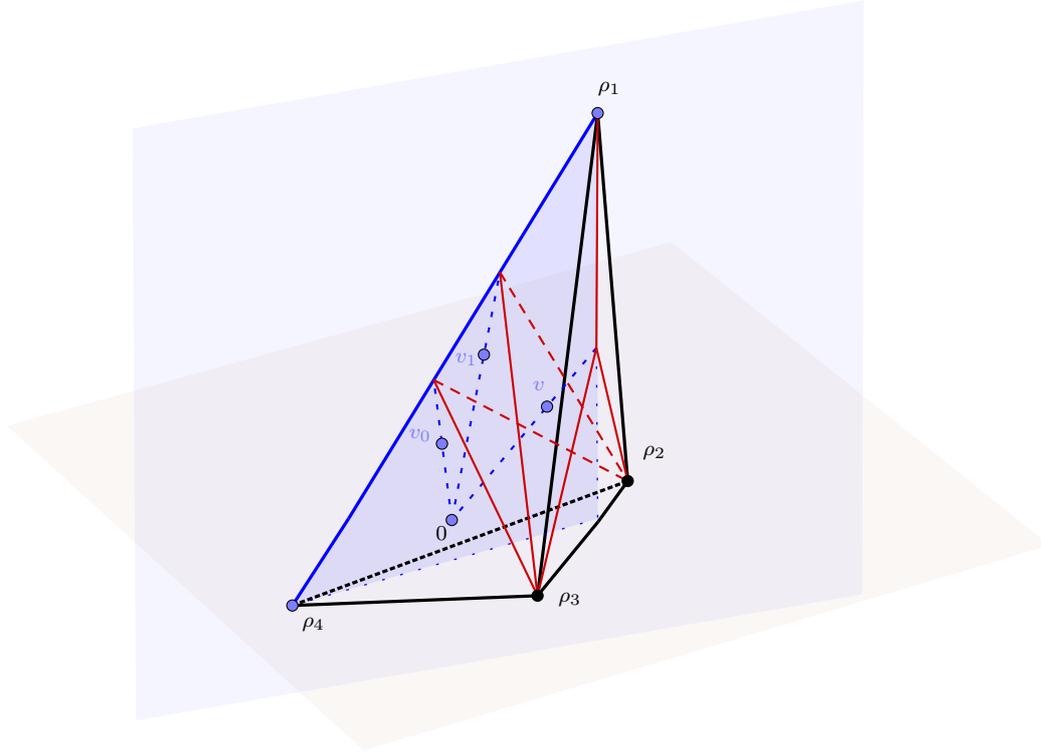
\begin{figure}[ht!]

\begin{tikzpicture}[scale=1.4]
\clip(1.9914948960418688,-2.9) rectangle (13.648297679014641,4.3);
\fill[line width=0pt,color=zzttqq,fill=zzttqq,fill opacity=0.04] (2.216457828145528,0.2319764088220546) -- (8.460428702931418,1.9736181365071375) -- (12.017824572245635,-0.8982379038246481) -- (5.570044559113618,-2.8436887698558575) -- cycle;
\fill[line width=0pt,color=qqqqff,fill=qqqqff,fill opacity=0.08]
(7.76871283945365,3.1933849493362287)-- (4.898973399687215,-1.4699590406016534)-- (7.76871283945365,-0.6494231031212336 ) -- cycle;
%

\fill[line width=2pt,color=qqqqff,fill=qqqqff,fill opacity=0.04]
(3.3960804167549288,3.0492280539165743) --  (3.428944429080794,-2.56023940829986283) -- (10.255855434742006,-1.36421245957177206) -- (10.269098938981704,4.2612251340788507) -- cycle;



\draw [line width=1.2pt] (7.76871283945365,3.1933849493362287)-- (8.050927408735216,-0.29128716326857784);
\draw [line width=1.2pt] (7.77314533253249,-0.6765582486396625)-- (7.203343137394577,-1.377254510923771);
\draw [line width=1.2pt] (7.77314533253249,-0.6765582486396625)-- (8.050927408735216,-0.29128716326857784);
\draw [line width=1.2pt] (7.76871283945365,3.1933849493362287)-- (7.203343137394577,-1.377254510923771);
\draw [line width=1.2pt] (4.898973399687215,-1.4699590406016534)-- (7.203343137394577,-1.377254510923771);
\draw [line width=1.2pt,color=qqqqff] (4.898973399687215,-1.4699590406016534)-- (5.4285203927986805,-0.6458927073068963);
\draw [line width=1.2pt,color=qqqqff] (5.4285203927986805,-0.6458927073068963)-- (7.76871283945365,3.1933849493362287);


\draw [line width=1.2pt,dash pattern=on 2pt off 1pt] (4.898973399687215,-1.4699590406016534)-- (8.050927408735216,-0.29128716326857784);

\draw [line width=0.5pt,dash pattern=on 1pt off 6pt,color=qqqqff] (4.898973399687215,-1.4699590406016534)-- (7.76871283945365,-0.6494231031212336 );
\draw [line width=0.5pt,dash pattern=on 1pt off 6pt,color=qqqqff] (7.756360770462977,0.9700125310148471)-- (7.76871283945365,-0.6494231031212336 );

\draw [line width=0.8pt,dash pattern=on 2pt off 4pt,color=qqqqff] (7.756360770462977,0.9700125310148471)-- (6.397633181488796,-0.6604605757541664);
\draw [line width=.8pt,dash pattern=on 2pt off 4pt,color=qqqqff] (6.397633181488796,-0.6604605757541664)-- (6.228070752252944,0.6658352383163377);
\draw [line width=.8pt,dash pattern=on 2pt off 4pt,color=qqqqff] (6.397633181488796,-0.6604605757541664)-- (6.849903766083704,1.686003301663429);

\draw [line width=.8pt,color=ccqqqq] (7.756360770462977,0.9700125310148471)-- (7.76871283945365,3.1933849493362287);
\draw [line width=.8pt,color=ccqqqq] (7.756360770462977,0.9700125310148471)-- (8.050927408735216,-0.29128716326857784);
\draw [line width=.8pt,color=ccqqqq] (7.756360770462977,0.9700125310148471)-- (7.203343137394577,-1.377254510923771);
\draw [line width=.8pt,color=ccqqqq] (6.228070752252944,0.6658352383163377)-- (7.203343137394577,-1.377254510923771);
\draw [line width=.8pt,color=ccqqqq] (6.849903766083704,1.686003301663429)-- (7.203343137394577,-1.377254510923771);
\draw [line width=.8pt,dash pattern=on 4pt off 3pt,color=ccqqqq] (6.849903766083704,1.686003301663429)-- (8.050927408735216,-0.29128716326857784);
\draw [line width=.8pt,dash pattern=on 4pt off 3pt,color=ccqqqq] (6.228070752252944,0.6658352383163377)-- (8.050927408735216,-0.29128716326857784);

\begin{scriptsize}
\draw [fill=xdxdff] (6.397633181488796,-0.6604605757541664) circle (1.5pt);
\draw[color=black] (6.3,-0.784605757541664) node {$0$};
\draw [fill=xdxdff] (7.76871283945365,3.1933849493362287) circle (1.5pt);
\draw[color=black] (7.8787618536191495,3.423515775966565) node {$\rho_1$};
\draw [fill=black]
 (8.050927408735216,-0.29128716326857784) circle (1.5pt);
\draw[color=black] (8.30255398928947,-0.01979532635477954) node {$\rho_2$};
\draw [fill=black] (7.203343137394577,-1.377254510923771) circle (1.5pt);
\draw[color=black] (7.50794373490762,-1.4103632715230148) node {$\rho_3$};
\draw [fill=xdxdff] (4.898973399687215,-1.4699590406016534) circle (1.5pt);
\draw[color=black] (5.0976259632826775,-1.6487463478375692) node {$\rho_4$};

\draw [fill=xdxdff] (6.305210379539784,0.06245885812077345) circle (1.5pt);
\draw[color=xdxdff] (6.100349806698174,0.1391267245215902) node {$v_0$};

\draw [fill=xdxdff] (6.699543457104778,0.9059060642266625) circle (1.5pt);
\draw[color=xdxdff] (6.534680916930308,0.8602239873828339) node {$v_1$};

\draw [fill=xdxdff] (7.292596070919589,0.4134948915627823) circle (1.5pt);
\draw[color=xdxdff] (7.216586641634276,0.6123798856300944) node {$v$};
\end{scriptsize}

\end{tikzpicture}
\caption{A picture of the fan $\widetilde{\Sigma}$ for $k=2$. The segments colored in blue all belong to the same plane. The three star-subdivisions of the fan $\Delta$ are represented by the segments in red. (For sake of clarity, the picture is out of scale.)}\label{fig:fan-vectors}
\end{figure}

Write $\widetilde{G}=G_{k-1}$ and $\widetilde{\Sigma}=\Sigma_{k-1}$, and 
let $\widetilde{X}$ be the strict transform of $X$ in $\widetilde{G}$. 
By construction, 
the minimal resolution $f \colon \widetilde{X} \to X$ is the restriction of the birational morphism $\widetilde{G} \to \P$. The curve $D_{v} \cap \widetilde{X} \subset \widetilde{X}$ is the $-(4k+1)$-exceptional curve $F$ over the point $p \in X$, and the intersection $\widetilde{X} \cap (D_{v_0}  \cup D_{v_1} \cup \dots \cup D_{v_{k-1}})$ is given by the union of the four chains 
 \[C_i=
 \bigcup_{j=0}^{k-1} C_{i,j} \quad  i=1,2,3,4\] where $C_{i,j}$ is the $j$th exceptional curve over the point $p_i \in X$ (see Remark \ref{rem:min-names} and the top of Figure \ref{fig:surfaces}).  The intersection $\widetilde{X} \cap D_{\rho_1} \subset \widetilde{X}$ is the union of the proper transforms  $C_{i,k}$ of the lines $c_i \subset X$, $i=1,2,3,4$.
Each $C_{i,k}$ has self-intersection $-1$; indeed $K_X.c_i=-\frac{1}{(2k+1)(4k+1)}$ since $c_i \in |\mathcal{O}_\P(2)| \oplus |\mathcal{O}_\P(2k+1)|$, thus by Remark \ref{rem:canonical-min} 
\[  K_{\widetilde{X}}. C_{i,k}= K_X.c_i -\frac{4k-1}{4k+1} -\frac{1}{2k+1}=-1
\]  and by  adjunction  $C_{i,k}^2=-1$.\smallskip
 
Now, observe that the vectors $v_1, \dots, v_{k-1}, \rho_1$ all belong to the cone 
$\langle v_0, v\rangle$ (see Figure \ref{fig:fan-vectors}). In particular, it follows from \eqref{eq:v} and \eqref{eq:v0} that $\rho_1=v_0+kv$; combining this and \eqref{eq:vi} one finds that:
\begin{equation}    v_i=v_{i-1}+v \quad  i=1, \dots, k-1 \qquad \text{and}  \quad \rho_1=v_{k-1}+v
\end{equation}
This implies that   
the fan $\widetilde{\Sigma}$ is obtained from the fan $\Sigma^\prime_0$ with rays $\rho_2, \rho_3, \rho_4, v_0,v$ and maximal cones 
\[\langle \rho_2, \rho_3, \rho_4 \rangle \quad \langle \rho_2,v_0, v\rangle  \quad \langle   \rho_2, \rho_4 , v_0\rangle \quad \langle \rho_3,v_0,  v\rangle  \quad \langle  \rho_3, \rho_4, v_0\rangle \quad \langle  \rho_2, \rho_3, v\rangle
\] by the sequence of $k$ star subdivisions: 
\[
\begin{split}
 &\Sigma^\prime_i=(\Sigma^\prime_{i-1})^\ast(v_{i-1}+v)  \quad i=1, \dots, k\\
\end{split} \] where  $\widetilde{\Sigma}={\Sigma}^\prime_k$. 
For $i=1, \dots, k$ the toric variety $G^\prime_i=G_{\Sigma^\prime_i}$ is the blow up of the toric variety $G^\prime_{i-1}$ along $D_{v_{i-1}} \cap D_{v}$. 

Let $Y^\prime$ be the proper transform of $\widetilde{X}$  in $G^\prime_0$. The morphism $\widetilde{G} \to G^\prime_0$ contracts the divisors $D_{\rho_1}, D_{v_{k-1}}, \dots,  D_{v_1}$ in $\widetilde{G}$, thus contracts the curves $C_{i,k}, C_{i, k-1}, \dots, C_{i,1}$, $i=1,2,3,4$, in $\widetilde{X}$. For $i=1,2,3,4$ the proper transform of $C_{i,0}$  in $Y^\prime$ is a $(-2)$-curve, and the proper transform  of $F$ in $Y^\prime$ is a $(-1)$-curve (see the right of Figure \ref{fig:surfaces}). 

Now, note that  the vector $v$ can be written in terms of $\rho_2, \rho_3, v_0$ as:
\begin{equation}
 v=\rho_2	+\rho_3+4v_0
\end{equation} and that $\rho_2, \rho_3, \rho_4, v_0$ span the lattice $N$ and satisfy:
\begin{equation}\label{eq:relation1112} 2v_0+\rho_2+\rho_3+\rho_4=0
\end{equation} This implies that $G^\prime_0$ is the weighted blow-up of $\P(2,1,1,1)$ at the smooth point $q=(0:0:0:1)$ with weights $(4,1,1)$. 

Let $X^\prime$ be the proper transform of $Y^\prime$ in $\P(2,1,1,1)$. The morphism $G_0^\prime \to \P(2,1,1,1)$ contracts the divisor $D_v \subset G_0^\prime$ to the point $q$, thus contracts the proper transform of $F$ in $Y^\prime$ to $q$. The proper transforms $c_{i,0} \subset X^\prime$ of the curves $C_{i,0} \subset \widetilde{X}$, $i=1,2,3,4$, are $(-1)$-curves meeting at the point $q$ (see the bottom right of Figure \ref{fig:surfaces}).

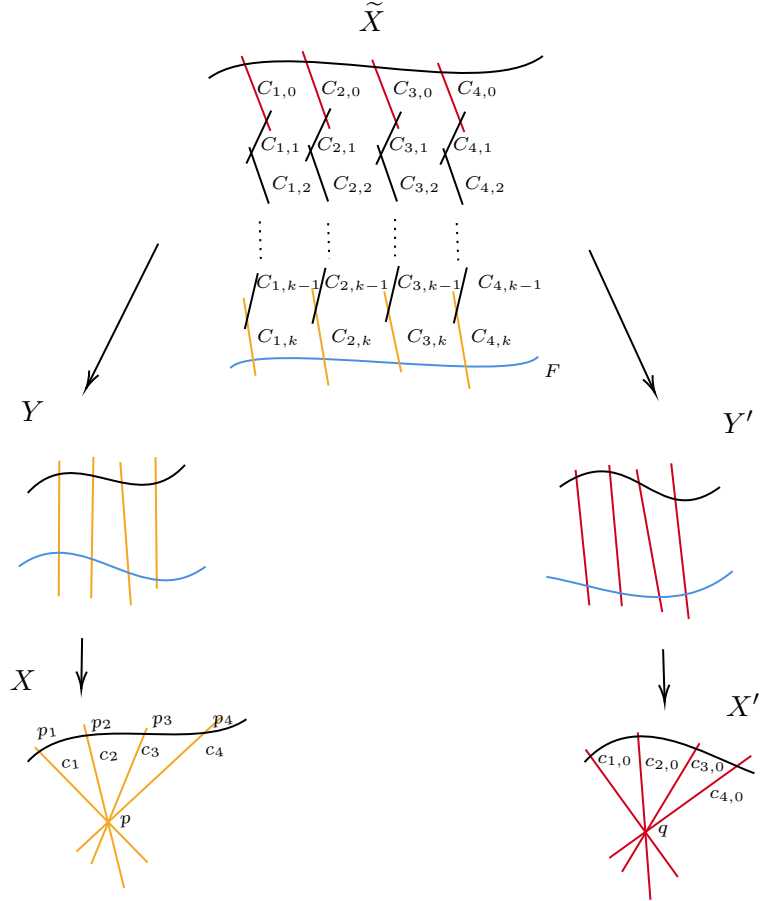
\begin{figure}[ht!]
\tikzset{every picture/.style={line width=0.75pt}} 

\begin{tikzpicture}[x=0.75pt,y=0.75pt,yscale=-0.8,xscale=0.8]

\draw [color={rgb, 255:red, 74; green, 144; blue, 226 }  ,draw opacity=1 ][line width=0.75]    (218.5,245.23) .. controls (234.47,226.92) and (394.53,256.31) .. (410.5,238) ;
\draw [color={rgb, 255:red, 208; green, 2; blue, 27 }  ,draw opacity=1 ]   (263.5,46) -- (280.5,94.77) ;
\draw [color={rgb, 255:red, 0; green, 0; blue, 0 }  ,draw opacity=1 ][line width=0.75]    (267.3,104.99) -- (279.5,140.47) ;
\draw [color={rgb, 255:red, 0; green, 0; blue, 0 }  ,draw opacity=1 ][line width=0.75]    (281.01,81.98) -- (265.5,115.23) ;
\draw [color={rgb, 255:red, 245; green, 166; blue, 35 }  ,draw opacity=1 ][line width=0.75]    (226.8,201.17) -- (234.2,250.06) ;
\draw [color={rgb, 255:red, 245; green, 166; blue, 35 }  ,draw opacity=1 ][line width=0.75]    (314.5,197) -- (325.13,247.84) ;
\draw [color={rgb, 255:red, 245; green, 166; blue, 35 }  ,draw opacity=1 ][line width=0.75]    (357.5,197) -- (367.74,258.27) ;
\draw [color={rgb, 255:red, 0; green, 0; blue, 0 }  ,draw opacity=1 ][line width=0.75]  [dash pattern={on 0.84pt off 2.51pt}]  (321.26,151.46) -- (321.26,174.99) ;
\draw [color={rgb, 255:red, 0; green, 0; blue, 0 }  ,draw opacity=1 ][line width=0.75]  [dash pattern={on 0.84pt off 2.51pt}]  (359.74,154.03) -- (360.69,177.56) ;
\draw [color={rgb, 255:red, 0; green, 0; blue, 0 }  ,draw opacity=1 ][line width=0.75]  [dash pattern={on 0.84pt off 2.51pt}]  (279.5,154.11) -- (280.23,176.61) ;
\draw [color={rgb, 255:red, 0; green, 0; blue, 0 }  ,draw opacity=1 ][line width=0.75]  [dash pattern={on 0.84pt off 2.51pt}]  (236.85,152.96) -- (237.5,176.62) ;
\draw [color={rgb, 255:red, 208; green, 2; blue, 27 }  ,draw opacity=1 ]   (225.5,49) -- (243.5,96.13) ;
\draw [color={rgb, 255:red, 0; green, 0; blue, 0 }  ,draw opacity=1 ][line width=0.75]    (230.3,106.35) -- (242.5,141.83) ;
\draw [color={rgb, 255:red, 0; green, 0; blue, 0 }  ,draw opacity=1 ][line width=0.75]    (244.01,83.35) -- (228.5,116.6) ;
\draw [color={rgb, 255:red, 208; green, 2; blue, 27 }  ,draw opacity=1 ]   (306.96,51.33) -- (323.5,94.77) ;
\draw [color={rgb, 255:red, 0; green, 0; blue, 0 }  ,draw opacity=1 ][line width=0.75]    (310.3,104.99) -- (322.5,140.47) ;
\draw [color={rgb, 255:red, 0; green, 0; blue, 0 }  ,draw opacity=1 ][line width=0.75]    (325.01,84.03) -- (309.5,117.28) ;
\draw [color={rgb, 255:red, 208; green, 2; blue, 27 }  ,draw opacity=1 ]   (347.96,53.38) -- (364.5,96.82) ;
\draw [color={rgb, 255:red, 0; green, 0; blue, 0 }  ,draw opacity=1 ][line width=0.75]    (351.3,107.04) -- (363.5,142.51) ;
\draw [color={rgb, 255:red, 0; green, 0; blue, 0 }  ,draw opacity=1 ][line width=0.75]    (365.01,84.03) -- (349.5,117.28) ;
\draw [color={rgb, 255:red, 0; green, 0; blue, 0 }  ,draw opacity=1 ][line width=0.75]    (235.65,185.56) -- (227.5,220.69) ;
\draw [color={rgb, 255:red, 0; green, 0; blue, 0 }  ,draw opacity=1 ][line width=0.75]    (323.63,181.03) -- (315.48,216.17) ;
\draw [color={rgb, 255:red, 0; green, 0; blue, 0 }  ,draw opacity=1 ][line width=0.75]    (366.06,181.97) -- (357.91,217.1) ;
\draw    (205,63) .. controls (245,33) and (373.5,79) .. (413.5,49) ;

\draw [color={rgb, 255:red, 245; green, 166; blue, 35 }  ,draw opacity=1 ]   (172.61,383.95) -- (171.82,301.19) ; 
\draw [color={rgb, 255:red, 245; green, 166; blue, 35 }  ,draw opacity=1 ]   (131.65,389.89) -- (133.06,301.02) ;
\draw [color={rgb, 255:red, 245; green, 166; blue, 35 }  ,draw opacity=1 ]   (156.43,394.3) -- (149.83,304.32) ;
\draw [color={rgb, 255:red, 245; green, 166; blue, 35 }  ,draw opacity=1 ]   (111.42,388.73) -- (111.78,303.49) ;
\draw [color={rgb, 255:red, 0; green, 0; blue, 0 }  ,draw opacity=1 ][line width=0.75]    (190.2,306.08) .. controls (159.46,340.44) and (122.81,289.4) .. (92.07,323.75) ;
\draw [color={rgb, 255:red, 74; green, 144; blue, 226 }  ,draw opacity=1 ]   (86.5,370) .. controls (126.5,340) and (163,400) .. (203,370) ;
\draw [color={rgb, 255:red, 208; green, 2; blue, 27 }  ,draw opacity=1 ]   (504.5,403) -- (493.57,305.11) ;
\draw [color={rgb, 255:red, 208; green, 2; blue, 27 }  ,draw opacity=1 ]   (462.79,394.94) -- (454.87,306.14) ;
\draw [color={rgb, 255:red, 208; green, 2; blue, 27 }  ,draw opacity=1 ]   (487.98,398.59) -- (471.96,308.91) ;
\draw [color={rgb, 255:red, 208; green, 2; blue, 27 }  ,draw opacity=1 ]   (442.48,394.4) -- (433.9,309.24) ;
\draw [color={rgb, 255:red, 74; green, 144; blue, 226 }  ,draw opacity=1 ]   (415.65,376.46) .. controls (449.5,382) and (495.13,404.09) .. (531.9,372.9) ;
\draw    (424,320) .. controls (469.5,287) and (484,350) .. (524,320) ;
\draw [color={rgb, 255:red, 208; green, 2; blue, 27 }  ,draw opacity=1 ]   (440.15,485.03) -- (497.74,565.22) ;
\draw [color={rgb, 255:red, 208; green, 2; blue, 27 }  ,draw opacity=1 ]   (543.5,489) -- (454.94,553.24) ;
\draw [color={rgb, 255:red, 208; green, 2; blue, 27 }  ,draw opacity=1 ]   (511.26,481.13) -- (462.85,561.97) ;
\draw [color={rgb, 255:red, 208; green, 2; blue, 27 }  ,draw opacity=1 ]   (472.72,474.31) -- (480.62,579.51) ;
\draw [color={rgb, 255:red, 245; green, 166; blue, 35 }  ,draw opacity=1 ][line width=0.75]    (269.5,195) -- (279.74,256.27) ;
\draw [color={rgb, 255:red, 0; green, 0; blue, 0 }  ,draw opacity=1 ][line width=0.75]    (277.85,182.2) -- (269.7,217.33) ;
\draw [color={rgb, 255:red, 245; green, 166; blue, 35 }  ,draw opacity=1 ]   (96.55,483.58) -- (166.86,556.08) ;
\draw [color={rgb, 255:red, 245; green, 166; blue, 35 }  ,draw opacity=1 ]   (212.5,465) -- (122.43,549.26) ;
\draw [color={rgb, 255:red, 245; green, 166; blue, 35 }  ,draw opacity=1 ]   (166.47,471.39) -- (131.72,556.97) ;
\draw [color={rgb, 255:red, 245; green, 166; blue, 35 }  ,draw opacity=1 ]   (127.11,469.17) -- (152.23,572.21) ;
\draw    (173.5,167) -- (129.39,256.21) ;
\draw [shift={(128.5,258)}, rotate = 296.31] [color={rgb, 255:red, 0; green, 0; blue, 0 }  ][line width=0.75]    (10.93,-3.29) .. controls (6.95,-1.4) and (3.31,-0.3) .. (0,0) .. controls (3.31,0.3) and (6.95,1.4) .. (10.93,3.29)   ;
\draw    (442.5,171) -- (482.67,259.18) ;
\draw [shift={(483.5,261)}, rotate = 245.51] [color={rgb, 255:red, 0; green, 0; blue, 0 }  ][line width=0.75]    (10.93,-3.29) .. controls (6.95,-1.4) and (3.31,-0.3) .. (0,0) .. controls (3.31,0.3) and (6.95,1.4) .. (10.93,3.29)   ;
\draw    (126,415) -- (125.53,445) ;
\draw [shift={(125.5,447)}, rotate = 270.9] [color={rgb, 255:red, 0; green, 0; blue, 0 }  ][line width=0.75]    (10.93,-3.29) .. controls (6.95,-1.4) and (3.31,-0.3) .. (0,0) .. controls (3.31,0.3) and (6.95,1.4) .. (10.93,3.29)   ;
\draw    (488.5,422) -- (489.44,453.34) ;
\draw [shift={(489.5,455.34)}, rotate = 268.28] [color={rgb, 255:red, 0; green, 0; blue, 0 }  ][line width=0.75]    (10.93,-3.29) .. controls (6.95,-1.4) and (3.31,-0.3) .. (0,0) .. controls (3.31,0.3) and (6.95,1.4) .. (10.93,3.29)   ;
\draw [color={rgb, 255:red, 0; green, 0; blue, 0 }  ,draw opacity=1 ][line width=0.75]    (228.5,466) .. controls (195.5,489) and (122.81,458.4) .. (92.07,492.75) ;
\draw [color={rgb, 255:red, 0; green, 0; blue, 0 }  ,draw opacity=1 ][line width=0.75]    (545.5,500) .. controls (512.62,487.54) and (470.23,458.64) .. (439.5,493) ;

\draw (167.19,462.2) node [anchor=north west][inner sep=0.75pt]  [rotate=-352.88] [align=left, font=\tiny] {$p_3$};
\draw (205,462.2) node [anchor=north west][inner sep=0.75pt]  [rotate=-352.88] [align=left, font=\tiny] {$p_4$};
\draw (95.47,469.88) node [anchor=north west][inner sep=0.75pt]  [rotate=-352.88] [align=left, font=\tiny] {$p_1$};
\draw (129.2,464.85) node [anchor=north west][inner sep=0.75pt]  [rotate=-352.88] [align=left, font=\tiny] {$p_2$};
\draw (147.19,525.2) node [anchor=north west][inner sep=0.75pt]  [rotate=-352.88] [align=left, font=\tiny] {$p$};
\draw (296.84,13.34) node [anchor=north west][inner sep=0.75pt]   [align=left] {$\widetilde{X}$};
\draw (412.5,241) node [anchor=north west][inner sep=0.75pt]   [align=left, font=\tiny] {$F$};

\draw (320.44,185) node [anchor=north west][inner sep=0.75pt]  [rotate=-358.45] [align=left,font=\tiny] {$C_{3, k-1}$};
\draw (370.71,185) node [anchor=north west][inner sep=0.75pt]  [rotate=-358.45] [align=left,font=\tiny] {$C_{4, k-1}$};
\draw (232.44,185) node [anchor=north west][inner sep=0.75pt]  [rotate=-358.45] [align=left,font=\tiny] {$C_{1, k-1}$};
\draw (275.03,185) node [anchor=north west][inner sep=0.75pt]  [rotate=-358.45] [align=left ,font=\tiny] {$C_{2, k-1}$};

\draw (279.45,219) node [anchor=north west][inner sep=0.75pt]   [align=left, font=\tiny] {$C_{2,k}$};
\draw (326.69,219) node [anchor=north west][inner sep=0.75pt]   [align=left,font=\tiny] {$C_{3,k}$};
\draw (233.45,219) node [anchor=north west][inner sep=0.75pt]   [align=left, font=\tiny] {$C_{1,k}$};
\draw (366.69,219) node [anchor=north west][inner sep=0.75pt]   [align=left,font=\tiny] {$C_{4,k}$};

\draw (242.47,122.75) node [anchor=north west][inner sep=0.75pt]   [align=left, font=\tiny] {$C_{1,2}$};
\draw (322.08,122.75) node [anchor=north west][inner sep=0.75pt]   [align=left ,font=\tiny] {$C_{3,2}$};
\draw (362.99,122.75) node [anchor=north west][inner sep=0.75pt]   [align=left,font=\tiny] {$C_{4,2}$};

\draw (235.53,98) node [anchor=north west][inner sep=0.75pt]   [align=left,font=\tiny] {$C_{1,1}$};
\draw (270.59,98) node [anchor=north west][inner sep=0.75pt]   [align=left,font=\tiny] {$C_{2,1}$};
\draw (315.5,98) node [anchor=north west][inner sep=0.75pt]   [align=left, ,font=\tiny] {$C_{3,1}$};
\draw (355.5,98) node [anchor=north west][inner sep=0.75pt]   [align=left, ,font=\tiny] {$C_{4,1}$};

\draw (318,63) node [anchor=north west][inner sep=0.75pt]   [align=left,font=\tiny] {$C_{3,0}$};

\draw (233.03,63) node [anchor=north west][inner sep=0.75pt]  [rotate=-359.17] [align=left,font=\tiny] {$C_{1,0}$};
\draw (357.94,63) node [anchor=north west][inner sep=0.75pt]  [rotate=-359.17] [align=left,font=\tiny] {$C_{4,0}$};
\draw (273.04,63) node [anchor=north west][inner sep=0.75pt]  [rotate=-359.17] [align=left,font=\tiny] {$C_{2,0}$};

\draw (280.47,122.75) node [anchor=north west][inner sep=0.75pt]   [align=left,font=\tiny] {$C_{2,2}$};

\draw (503.67,490) node [anchor=north west][inner sep=0.75pt]   [align=left, font=\tiny] {$c_{3,0}$};
\draw (515.67,508.83) node [anchor=north west][inner sep=0.75pt]   [align=left, font=\tiny] {$c_{4,0}$};
\draw (445.73,486.46) node [anchor=north west][inner sep=0.75pt]   [align=left, font=\tiny] {$c_{1,0}$};
\draw (475.29,488.83) node [anchor=north west][inner sep=0.75pt]   [align=left, font=\tiny] {$c_{2,0}$};
\draw (483.29,530.83) node [anchor=north west][inner sep=0.75pt]   [align=left, font=\tiny] {$q$};

\draw (160.19,482.2) node [anchor=north west][inner sep=0.75pt]  [rotate=-352.88] [align=left, font=\tiny] {$c_3$};
\draw (200.19,482.2) node [anchor=north west][inner sep=0.75pt]  [rotate=-352.88] [align=left, font=\tiny] {$c_4$};
\draw (110.47,489.88) node [anchor=north west][inner sep=0.75pt]  [rotate=-352.88] [align=left, font=\tiny] {$c_1$};
\draw (134.2,484.85) node [anchor=north west][inner sep=0.75pt]  [rotate=-352.88] [align=left, font=\tiny] {$c_2$};

\draw (86,264) node [anchor=north west][inner sep=0.75pt]   [align=left] {$Y$};
\draw (524,270) node [anchor=north west][inner sep=0.75pt]   [align=left] {$Y^\prime$};
\draw (79,433) node [anchor=north west][inner sep=0.75pt]   [align=left] {$X$};
\draw (526,448) node [anchor=north west][inner sep=0.75pt]   [align=left] {$X^\prime$};
\end{tikzpicture}
\caption{A picture of the surfaces $X, Y, \widetilde{X}$, $Y^\prime$, and $X^\prime$. The surface $Y$ is the weighted blow-up of $X$ at $p$ with weights $1/(4k+1)(1,1)$; the surface $Y^\prime$ is the blow-up of $X^\prime$ at the smooth point $q$. The curves $F$, $C_{i,0},\dots, C_{i, k-1}$, $i=1,2,3,4$, are the exceptional curves of the minimal resolution $\widetilde{X} \to X$, while the curves $F$, $C_{i,1},\dots, C_{i, k}$, $i=1,2,3,4$, are the exceptional curves of the morphism $\widetilde{X} \to X^\prime$. We use the same color for a curve and its proper transforms under the maps considered.}
\label{fig:surfaces}
\end{figure}
It follows from \eqref{eq:generic_X8k+4} and \eqref{eq:relation1112} that the surface $X^\prime$ is the smooth del Pezzo surface of degree two cut out of $\P(2,1,1,1)$ by the section of $\mathcal{O}(4)$ given by the polynomial $h^{(0)}$ in \eqref{eq:h0}. Indeed, by \eqref{eq:relation1112},
$\Delta$ and the fan of $\P(2,1,1,1)$ share the cone  $\sigma=\langle \rho_2, \rho_3, \rho_4 \rangle$.
Writing $U_\sigma \simeq  \CC^3/\Z_2(1,1,1)$ for the affine toric variety corresponding to $\sigma$, one has that $X \cap U_\sigma \simeq  X^\prime \cap U_\sigma$, and,  by \eqref{eq:generic_X8k+4},  $X \cap U_\sigma$ is defined by the equation:
\[ (a+h_4(y_1,y_2)+bz^2+h_2(y_1,y_2)+h_1(y_1,y_2)z =0) \subset U_\sigma
\]
where now $y_1,y_2, z$ denote affine coordinates on $\CC^3$. It follows that $X^\prime$ is cut out of $\P(2,1,1,1)$ by $h^{(0)}$, thus is a smooth del Pezzo surface of degree two.

The union of the four curves $c_{i,0}$ is given by $\left( x=h_4(y_1,y_2)=0 \right) \subset \P(2,1,1,1)$. The morphism $\tau \colon \widetilde{X} \to X^\prime$ in the statement of the Theorem is the restriction of the birational morphism $\widetilde{G} \to \P(2,1,1,1)$ constructed above.
 \end{proof}

An immediate consequence of Theorem \ref{thm:min-resol} is:

\begin{corollary} For $k>0$ integer, let $X=X^{(k)}$ be a generic log del Pezzo surface \eqref{eq:series} and let  $\widetilde{X}^{(k)}$  be its minimal resolution. Then $\widetilde{X}^{(k)}$ is the blow-up of $\P^2$ in $8+4k$ points. 
Moreover, for all $k>0$, $X^{(k)}$ is a rational surface of Picard rank $\rho(X^{(k)})=8$. 
\end{corollary}

Theorem \ref{thm:min-resol} highlights a relation between the surfaces of the series \eqref{eq:series} as $k$ varies. Fix a generic choice of $a,b\in \C\setminus \{0\}$, and of polynomials $h_4,h_2,h_1$. This defines a sequence of polynomials $\{h^{(k)}\}_{k>0}$, and with it sequences $\{X^{(k)}\}$ and $\{\widetilde{X}^{(k)}\}$ of log del Pezzo surfaces and their minimal resolutions. By Theorem \ref{thm:min-resol}, every surface $\widetilde{X}^{(k)}$ admits a birational morphism $\tau^{(k)} \colon  \widetilde{X}^{(k)} \to X^\prime$.

\begin{corollary} 
For $0<k_1 < k_2$, the map $\tau^{(k_2) }$
factors through the map $\tau^{(k_1) }$, as illustrated in 
Figure \ref{fig:morphisms}. 
\end{corollary}

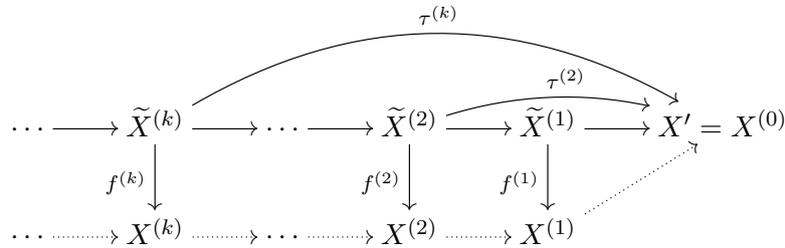
\begin{figure}[ht!] \begin{equation*}
\xymatrix{
 \cdots \ar[r] & \widetilde{X}^{(k)}  \ar@/^3pc/^{\tau^{(k)}}[rrrr] \ar[r] \ar_{f^{(k)}}[d] & \cdots \ar[r] & \widetilde{X}^{(2)} \ar@/^1pc/^{\tau^{(2)}}[rr] \ar[r] \ar[d]_{f^{(2)}}
&  \widetilde{X}^{(1)}  \ar[r]  \ar[d]_{f^{(1)}} & X^\prime=X^{(0)} \\
\cdots  \ar@{.>}[r] & {X}^{(k)} \ar@{.>}[r]   &  \cdots \ar@{.>}[r]  & {X}^{(2)} \ar@{.>}[r] & {X}^{(1)}  \ar@{.>}[ur] &  }
\end{equation*}
\caption{The morphims $f^{(k)} \colon \widetilde{X}^{(k)} \to X^{(k)}$ and $\tau^{(k)} \colon  \widetilde{X}^{(k)} \to X^\prime$. For all $k>0$, the surface $\widetilde{X}^{(k+1)}$ is the blow-up at four points of the surface $\widetilde{X}^{(k)}$. We use dashed arrows to denote rational maps and solid arrows to denote morphisms.}
\label{fig:morphisms}
\end{figure}

\begin{remark}\label{rem:surjectivity}
Any del Pezzo surface $X'$ as in Theorem \ref{thm:min-resol}
has a generalized Eckardt point, that is, a point where four $(-1)$-curves of $X'$ meet. 

By \cite[Theorem 5.1]{Kuw05} a del Pezzo surface $S$ of degree two admits a generalized Eckardt point if and only if,  after a suitable change of coordinates, it can be defined by an equation of the form
\begin{equation}\label{eq:normal-ourdp2}
\left( x^2+z^4+g_2(y_1,y_2)z^2+g_4(y_1, y_2)=0 \right) \subset \P(2,1,1,1)
	\end{equation}
where $x, y_1, y_2,z$ are weighted homogeneous coordinates on $\P(2,1,1,1)
$ and $g_i(y_1,y_2)$ is a polynomial of degree $i$ in the two variables $y_1, y_2$.

It is easy to see that, after a change of coordinates,  any equation of the form \eqref{eq:normal-ourdp2} can be rewritten as $(h^{(0)}=0)$, where $h^{(0)}$ is a polynomial as in \eqref{eq:h0}. Indeed, for instance, by sending
 $x \mapsto x+\mathtt{i} z^2+\frac{\mathtt{i}}{2}g_2(y_1,y_2)$,
 the polynomial in  \eqref{eq:normal-ourdp2}  takes the form:
\[ x^2-\frac{g_2(y_1,y_2)^2}{4} +g_4(y_1, y_2) +x
\left( 2 \mathtt{i}z^2+ \mathtt{i}g_2(y_1,y_2)\right)
\] 

Thus 
the surfaces $X^\prime$ appearing in Theorem \ref{thm:min-resol} 
are all del Pezzo surfaces of degree two with a generalized Eckardt point.
\end{remark}

\begin{remark} \label{rem:bijective}
Under the assignment $X \mapsto X'$ constructed in Theorem \ref{thm:min-resol}, to each surface $X$ corresponds bijectively a del Pezzo surface $X'$ with a generalized Eckardt point. Surjectivity is Remark \ref{rem:surjectivity}. For injectivity we may argue as follows. Suppose two log del Pezzo hypersurfaces, $X_1$ and $X_2$, get mapped to isomorphic del Pezzo surfaces of degree two $X_1'$ and $X_2'$ embedded in $\mathbb{P}(2,1,1,1)$. For $i=1,2$, write defining polynomials $h_i^{(0)}$ for $X_i'$. Since the surfaces are isomorphic, the $h_i^{(0)}$ are related by a change of coordinates $\psi$ of $\mathbb{P}(2,1,1,1)$. Then, $\psi$ must be a composition of transformations
\begin{equation}
\begin{split}
        x &\mapsto  r x,\\
        y_j &\mapsto  p_{1,j}y_1 + q_{2,j}y_2, \qquad  j=1,2\\
        z &\mapsto  s z + t_1y_1 +t_2 y_2
\end{split}
\end{equation}
(these are the projective transformation preserving the form \eqref{eq:h0}). It is a straightforward check that the same sequence of transformations, where $z$ is instead mapped to $s z + t_1x^ky_1 +t_2x^k y_2$, relates the polynomials $h_1^{(k)}$ and $h_2^{(k)}$ defining $X_1$ and $X_2$.
\end{remark}

\subsection{Moduli spaces of  Johnson--Koll\'ar surfaces} \label{sec:moduli} In this Section we use the construction of Theorem \ref{thm:min-resol} to obtain a statement about moduli spaces. As above, we fix $k>0$ integer. We start with the following remark.

\begin{remark}\label{rem:families_map}
The construction of Theorem \ref{thm:min-resol} can be carried out in families. More precisely, let $\mathscr X \to B$ be a flat family of Johnson-Koll\`ar hypersurfaces with base $B$. In other words, $\mathscr{X} \subset \mathbb P_B\coloneqq B \times \mathbb P$, and, if $\sO(1)$ be the tautological sheaf on $\mathbb P_B$, the family $\mathscr X$ is the zero locus of a general section of $\sO(8k+4)$. Equivalently, it is determined by a polynomial $h^{(k)}$ of the form \eqref{eq:generic_X8k+4}, where $a,b \in \sO_B(B)$ and $h_i \in H^0  (B \times \mathbb P , \sO(i(2k+1)))$ for $i=1,2,4$.

In this setting, the same argument as in the proof of Theorem \ref{thm:min-resol} gives a birational correspondence
\begin{equation}\label{eq:bir_correspondence}
    \begin{tikzcd}
         & \quad B \times \widetilde G \ar{dr} \ar{dl}  & \\
          \mathbb P_B   \arrow[rr,dashed] &   & B\times \mathbb P(2,1,1,1)
    \end{tikzcd}
\end{equation}
The correspondence, restricted to $\mathscr X$, maps it to a family $\mathscr X' \to B$ of smooth del Pezzo surfaces with a generalized Eckardt point. On each fiber $\mathscr X_b$ it coincides with $\tau\circ f^{-1} \colon \mathscr X_b \dashrightarrow \mathscr X'_b$. In other words, $\mathscr X'$ is the zero locus of the polynomial $h^{(0)}$, where now $h_i \in H^0(B \times \mathbb P(2,1,1,1) , \sO(i))$ for $i=1,2,4$.
\end{remark}

Now let $\mathscr{M}=\mathscr{M}_k$ be the moduli stack of quasismooth and wellformed hypersurfaces of degree $8k+4$ in the weighted projective $3$-space $\P=\P(2,2k+1,2k+1,4k+1)$, i.e. the open substack of the quotient stack $[|\mathcal{O}_\P(8k+4)| / \mathrm{Aut}(\P)]$ parametrizing quasismooth and wellformed hypersurfaces. 
The moduli stack $\mathscr M$ admits a coarse moduli space $M=M_k$ which  is an algebraic space as a consequence of \cite{KeelMori} and \cite[Theorem 3.13]{bunnett2021moduli}\footnote{ \cite{bunnett2021moduli} only refers to quasismooth hypersurfaces. 
However, we observed in Footnote \ref{footnote} that any quasismooth surface of the form \eqref{eq:series} is wellformed.}.
Similarly, one defines the moduli stack $\mathscr{P}$ of smooth del Pezzo surfaces of degree two, whose coarse moduli space $P$ is a quasiprojective variety by \cite[Example 5.22]{bunnett2021moduli}.
We denote by $\mathscr{P}_E$ the closed substack of $\mathscr P$ parametrizing surfaces with a generalized Eckardt point, and likewise define $P_E \subset P$.  The moduli space $P_E$ is a stratum of dimension $4$ of the moduli space $P$.  Surfaces in this stratum are called \emph{of type XII} in \cite[Table 8.9]{Dol_CAG}.

\begin{proposition}
\label{pro:moduli}
The moduli space $M$ is isomorphic to $P_E$.
\end{proposition}

\begin{proof}
We begin by constructing a morphism $M \to P_E$, using the universal property of moduli spaces and Yoneda's Lemma.

Let $[\mathscr X \to B]\in \mathscr{M}(B)$ be a flat family of log del Pezzo surfaces with base $B$. 
The construction of Remark \ref{rem:families_map} gives rise to a family $[\mathscr X'\to B]\in \mathscr P_E(B)$. In turn, this induces a natural transformation from $\mathscr M$ to the functor of points $h_{P_E}$, and hence a natural transformation $h_{M}\to h_{P_E}$ by the universal property of $M$. By Yoneda's Lemma, this is equivalent to a morphism $\phi \colon M \to P_E$. 

Conversely, consider a family $[\mathscr{X}'\to B] \in \mathscr{P}_E(B)$, written as the zero locus of a polynomial $h^{(0)}$ as in Remark \ref{rem:surjectivity}. Then, the inverse of the birational correspondence \eqref{eq:bir_correspondence} maps the family $\mathscr X'$ to one in $\mathscr{M}(B)$, defined by the polynomial $h^{(k)}$. Arguing as in Remark \ref{rem:bijective}, we see that the induced map $P_E \to M$ inverts $\phi$.
\end{proof}


\subsection{Classification of generalized Eckardt points}
\label{ssec:Eckardt}
Suppose that $S$ is a del Pezzo surface of degree two with a generalized Eckardt point $q$. Denote by $m_l$, $l=1,2,3,4$, the four $(-1)$-curves meeting at the point $q$. 
In this Section we classify all the possible configurations $n_l=\pi(m_l)$, where $\pi \colon S \to \P^2$ is the blow-down of an exceptional set. Keeping this notation:


\begin{lemma}
\label{lem:imagesOfEckardtLines}
Let $\pi \colon S\to \pr 2$ be a blowup of $\pr 2$ at seven points $x_i$ in general position 
(equivalently, a choice of  
an exceptional set for $S$). 
Up to reindexing, we have the following possibilities for the curves $n_l$:
\begin{itemize}
\item $n_1$ and $n_2$ are lines, and $n_3$ and $n_4$ are conics;
\item $n_1$ is a point, and the remaining are conics, or $n_1$ is a cubic and the remaining are lines;
\item $n_1$ is a point, and the remaining are a line, a conic, and a cubic.
\end{itemize}
\end{lemma}

\begin{proof}
Observe first of all that $n_k$ is a point for at most one $k$, since the $m_k$ are not disjoint. 

If $n_1=x_1$ is a point, then at most one of the other $n_k$ is a line: two lines in $\P^2$ containing $x_1$ are disjoint in the blowup, unless they coincide.

Suppose $n_1=x_1$ and $n_2$ is a line (containing $x_2$). At most one between $n_3,n_4$ is a conic: otherwise, $n_3$ and $n_4$ pass through $x_1$ with the same tangent direction of $n_2$, and therefore they intersect in $x_1$ (with multiplicity 2). Each conic must contain four points among ($x_3,...,x_7$), but if they share three they coincide, a contradiction.  

If $n_1,n_2,n_3$ are lines, $n_4$ cannot be a line: each line contains $\pi(q)$ and two of the $x_i$, and no point $x_i$ can be common to distinct lines. Similarly, $n_4$ is not a conic: if it was, it would have to share $\pi(q)$ and two of the $x_i$ with one among $n_1,n_2,n_3$, say $n_1$. But then $n_1\subset n_4$, a contradiction.  

At this point, we may argue dually on the curves $m_k'=-K_S-m_k$. In fact, they also intersect at a generalized Eckardt point $q'$: the bitangents corresponding to the $m_k$ (and hence to the $m_k'$) all meet at the point $\sigma(q)\in \pr 2\setminus Q$ (see Section \ref{ssec:dP2} for notation), and $q'$ is the other point in the preimage of $\sigma(q)$. Since the image of $m_k$ determines that of $m_k'$ and conversely, we can exclude all the configurations dual to those excluded above. For example, no more than one $n_k$ is a cubic, otherwise more than one $n_k'$ would be a point.
The only possibilities left are those in the statement: Example \ref{ex:EckardtPoints} shows that these do occur.
\end{proof}

\begin{example}[Construction of generalized Eckardt points]
\label{ex:EckardtPoints}
We present three constructions of seven points $x_1,...,x_7 \in \pr 2$ in general position. The blowup $\Bl_{\{x_i\}}\pr 2$ is a del Pezzo surface of degree two with a generalized Eckardt point (which is the preimage of $q$ in the first example, and lies in the preimage of $x_1$ in the other cases). In particular, these constructions are examples of the configurations listed in Lemma \ref{lem:imagesOfEckardtLines}. Observe moreover that these examples are not in general degenerations of one another: the families of examples all have dimension 4, as shown by a quick parameter count.
\begin{enumerate}[(a)]
\item \label{itm:EckA} Let $t_1,t_2$ be two lines in $\pr 2$ intersecting at a point $q$. Pick general points $x_1,x_2\in t_1$, $x_3,x_4\in t_2$. Let $s_1$ be a generic conic through $q,x_1,x_3$ and $s_2$ a generic conic through $q,x_2,x_4$. Then $s_1\cap s_2 = \{ q,x_5,x_6,x_7 \}$ (Figure \ref{fig:EckA}).
\item \label{itm:EckB}Let $x_1\in \pr 2$, and fix a conic $s_1$ through $x_1$. Define two general conics $s_2$ and $s_3$ which pass through $x_1$ with the same tangent direction of $s_1$. We have $s_1\cap s_2 = \{x_1,x_2,x_3\}$, $s_1\cap s_3=\{x_1,x_4,x_5\}$, and $s_2\cap s_3=\{x_1,x_6,x_7\}$. For generic choices, the points $x_i$, $i=1,...,7$ are in general position, and the blowup $\Bl_{\{x_i\}}\pr 2$ is a del Pezzo surface of degree two with a generalized Eckardt point at the intersection of the strict transforms of the $s_k$ with the exceptional divisor of $x_1$ (Figure \ref{fig:EckB}).
\item \label{itm:EckC}Fix a point $x_1\in \pr 2$, a line $t$ through it and a conic $s$ tangent to $t$ at $x_1$. Now pick generic points $x_2\in t$, and $x_3,...,x_6\in s$. Then, pick a nodal cubic in the 2-dimensional linear system of cubics containing $x_2,...,x_6$ and passing through $x_1$ with tangent direction $t$ (Figure \ref{fig:EckC}).  
\end{enumerate}
\end{example}

\begin{figure}[!ht]
\minipage{0.32\textwidth}
\begin{tikzpicture}[x=0.5pt,y=0.5pt,yscale=-0.9,xscale=0.9]

\draw    (251,53) -- (373.67,273.5) ;
\draw    (127.67,227.5) -- (305.67,65.5) ;
\draw   (136.68,181.08) .. controls (137.38,148.35) and (190.78,110.82) .. (255.95,97.24) .. controls (321.11,83.66) and (373.37,99.18) .. (372.67,131.91) .. controls (371.97,164.63) and (318.57,202.17) .. (253.4,215.75) .. controls (188.23,229.33) and (135.97,213.8) .. (136.68,181.08) -- cycle ;
\draw   (158.55,118.79) .. controls (151.18,83.11) and (196.75,70.17) .. (260.33,89.89) .. controls (323.91,109.61) and (381.43,154.52) .. (388.8,190.19) .. controls (396.17,225.87) and (350.6,238.81) .. (287.02,219.09) .. controls (223.44,199.38) and (165.92,154.47) .. (158.55,118.79) -- cycle ;


\draw (271,57.4) node [anchor=north west][inner sep=0.75pt]  [font=\small]  {$q$};
\draw  [fill=black] (273.99, 94.33) circle [x radius= 3.3, y radius= 3.3]   ;

\draw   (326.18, 188.14) circle [x radius= 3.3, y radius= 3.3]   ;
\draw   (348.4, 228.07) circle [x radius= 3.3, y radius= 3.3]   ;
\draw   (150.64, 206.59) circle [x radius= 3.3, y radius= 3.3]   ;
\draw   (193.16, 167.89) circle [x radius= 3.3, y radius= 3.3]   ;
\draw   (365.98, 152.33) circle [x radius= 3.3, y radius= 3.3]   ;
\draw   (165.85, 136.64) circle [x radius= 3.3, y radius= 3.3]   ;
\draw   (267.8, 212.31) circle [x radius= 3.3, y radius= 3.3]   ;
\end{tikzpicture}
\caption{
}
\label{fig:EckA}
\endminipage\hfill
\minipage{0.32\textwidth}
\begin{tikzpicture}[x=0.6pt,y=0.6pt,yscale=-1,xscale=1]
 \path (55,150);

\draw   (107.4,225.25) .. controls (101.21,212.45) and (129.77,185.82) .. (171.2,165.78) .. controls (212.63,145.73) and (251.24,139.86) .. (257.43,152.66) .. controls (263.63,165.46) and (235.06,192.09) .. (193.63,212.14) .. controls (152.2,232.18) and (113.6,238.06) .. (107.4,225.25) -- cycle ;
\draw   (124.21,193.79) .. controls (112.97,170.55) and (125.81,141.08) .. (152.91,127.97) .. controls (180,114.86) and (211.08,123.08) .. (222.33,146.32) .. controls (233.58,169.56) and (220.73,199.03) .. (193.63,212.14) .. controls (166.54,225.25) and (135.46,217.03) .. (124.21,193.79) -- cycle ;
\draw   (127.83,156.49) .. controls (108.88,117.34) and (107.63,78.78) .. (125.03,70.36) .. controls (142.43,61.94) and (171.9,86.86) .. (190.84,126.01) .. controls (209.78,165.16) and (211.03,203.72) .. (193.63,212.14) .. controls (176.23,220.56) and (146.77,195.64) .. (127.83,156.49) -- cycle ;


\draw   (201.34, 153.45) circle [x radius= 2.7, y radius= 2.7]   ;
\draw   (142.83, 181.76) circle [x radius= 2.7, y radius= 2.7]   ;
\draw   (126.18, 153) circle [x radius= 2.7, y radius= 2.7]   ;
\draw   (189.12, 122.54) circle [x radius= 2.7, y radius= 2.7]   ;
\draw   (124.9, 195.14) circle [x radius= 2.7, y radius= 2.7]   ;
\draw   (222.96, 147.69) circle [x radius= 2.7, y radius= 2.7]   ;
\draw   (193.63, 212.14) circle [x radius= 2.7, y radius= 2.7]   ;
\end{tikzpicture}
\caption{
}
\label{fig:EckB}
\endminipage\hfill
\minipage{0.32\textwidth}%
 \begin{tikzpicture}[x=0.5pt,y=0.5pt,yscale=-1,xscale=1]

\draw  [color=cite  ,draw opacity=1 ] (118.67,55.5) .. controls (152.67,17.5) and (255.01,7.29) .. (197.01,60.29) .. controls (139.01,113.29) and (249.67,215.5) .. (265.67,185.5) .. controls (281.67,155.5) and (74.67,161.5) .. (38.67,131.5) .. controls (2.67,101.5) and (84.67,93.5) .. (118.67,55.5) -- cycle ;
\draw    (177.67,-0.5) -- (14.67,154.5) ;
\draw   (92.05,150.41) .. controls (77.76,122.88) and (89.68,80.39) .. (118.67,55.5) .. controls (147.65,30.61) and (182.73,32.76) .. (197.01,60.29) .. controls (211.29,87.82) and (199.37,130.31) .. (170.39,155.19) .. controls (141.4,180.08) and (106.33,177.94) .. (92.05,150.41) -- cycle ;

\draw   (38.77, 131.58) circle [x radius= 3, y radius= 3]   ;
\draw   (117.97, 56.26) circle [x radius= 3, y radius= 3]   ;
\draw   (197.01, 60.29) circle [x radius= 3, y radius= 3]   ;
\draw   (189.64, 132.35) circle [x radius= 3, y radius= 3]   ;
\draw   (91.01, 148.28) circle [x radius= 3, y radius= 3]   ;
\draw   (166.97, 157.98) circle [x radius= 3, y radius= 3]   ;
\draw   (211.3, 163.2) circle [x radius= 3, y radius= 3]   ;
\end{tikzpicture}
\caption{
 }
 \label{fig:EckC}
\endminipage
\end{figure}

In fact, the configurations of Example \ref{ex:EckardtPoints} are the only possibility for generalized Eckardt points. Keeping the same notation as Lemma \ref{lem:imagesOfEckardtLines}, we have:

\begin{proposition}
\label{prop:ClassificationOfEckardtPoints} 
The map $\pi:S\to \pr 2$ is one of the blowups described in Example \ref{ex:EckardtPoints}. 
\end{proposition}

\begin{proof}
By Lemma \ref{lem:imagesOfEckardtLines}, there are only three possibilities for the $n_l$ (after possibly considering the dual curves $n_l'$). 
We only show one case, the others being similar. 

Suppose $n_1,n_2$ are lines and $n_3,n_4$ are conics. Since $n_1$ and $n_2$ intersect at $\pi(q)$, they cannot share the $x_i$'s they contain, so we may assume $x_1,x_2\in n_1$ and $x_3,x_4\in n_2$. The conic $n_3$ must contain 5 of the $x_i$ and pass through $q$, moreover it intersects $n_1$ and $n_2$ at two points ($n_3$ is smooth, otherwise the $x_i$ are not in general position). The only possibility is that $n_3$ contains, say, $x_1$ and $x_3$, whereas $x_2,x_4\in n_4$. Then, $n_3$ and $n_4$ must meet at $q$ and at $x_5,x_6,x_7$. This is the situation of Example \ref{ex:EckardtPoints} (a). 
\end{proof}


Later it will be useful to know that it is always possible to choose $\pi$ as in Example \ref{ex:EckardtPoints} (a). This is Corollary \ref{prop:casea} below, which follows from more general statements about \textit{cliques on} $S$. 

In graph theory, a clique is a complete subgraph of an undirected graph. Given a degree two del Pezzo surface $S$,  
by a clique on $S$ we mean a set of four $(-1)$-curves on $S$ whose dual graph is a clique in the dual graph of all $(-1)$-curves. The $(-1)$-curves meeting at a generalized Eckardt point form a clique, but the converse does not hold.

\begin{remark}
Lemma \ref{lem:imagesOfEckardtLines} holds if $\{m_1,...,m_4\}$ is only a clique on $S$, with a nearly identical proof. 
\end{remark}

We split the following Lemma in two for ease of reference, but part (ii) clearly implies part (i).

\begin{lemma} \label{lem:clique} Let $S$ be a del Pezzo surface of degree two and let $\{m_1, \dots, m_4\}$ be a clique on $S$. Then:
\begin{enumerate}[(i)]
    \item There exists an exceptional set $\sL$ on $S$ such that $\sL \cap \{m_1, \dots, m_4\}=\varnothing$.
\item Let $\{m_1^\prime, \dots, m_4^\prime\}$ be the dual clique. There exists an exceptional set $\sL$ on $S$ such that $\sL \cap \{m_1, \dots, m_4, m_1^\prime, \dots, m_4^\prime\}=\varnothing$.
\end{enumerate}
\end{lemma}
\begin{proof}
\begin{enumerate}[(i)]
    \item The set of $(-1)$-curves on $S$ is in bijection with the set of exceptional vectors in $I^{1,7}$ 
(see Section \ref{ssec:dP2}). Identifying the two sets,  
one checks with a \textsc{SageMath} computation that on $S$ there are $576$ exceptional sets and $630$ cliques, and that every clique admits a disjoint exceptional set. disjoint from itself and its dual.
\item An analogous computation shows that some of the exceptional sets found in (i) are also disjoint from the dual clique. \qedhere
\end{enumerate}
\end{proof}

\begin{corollary}
\label{prop:casea}
Suppose that $S$ is a del Pezzo surface of degree two with a generalized Eckardt point $q$. Then $S$ can always be realized as the blow-up of $\P^2$ described in Example \ref{ex:EckardtPoints} (a).
\end{corollary}

\begin{proof}
An immediate consequence of Proposition \ref{prop:ClassificationOfEckardtPoints} and Lemma \ref{lem:clique} (ii).	
\end{proof}


\section{The derived category of \texorpdfstring{$\mathcal{X}$}{sX}}
\label{sec:DerCat}
We construct full exceptional collections for $D(\widetilde X)$ in Section \ref{ssec_FullExcColl}; this directly implies the existence of full exceptional collections for $D(\sX)$. In Section \ref{ssec:HomSpaces}, after some additional choices, we compute the morphism spaces between objects in the above collections. The only non-vanishing spaces are $\Hom$ and $\Ext^1$. In Section \ref{ssec:tilting}, we observe that this is sufficient to produce tilting complexes for $D(\widetilde X)$ and $D(\sX)$.

\subsection{Full exceptional collections for \texorpdfstring{$\widetilde{X}$}{tildeX}}
\label{ssec_FullExcColl}
Several choices are involved in the construction of full exceptional collections for $\widetilde{X}$. Our guiding principle is to get exceptional collections whose Hom-spaces are as simple as possible, and we will make a choice explicit when it is made. 

We start by letting $\pi \colon X^\prime \to \P^2$ be a blow-down map which does not contract any of the lines $c_{i,0}$: such a $\pi$ always exists by Lemma \ref{lem:clique} (i). Let $\tau \colon \widetilde{X} \to X^\prime$ be the morphism in Theorem \ref{thm:min-resol} and write $\alpha=\tau \circ \pi$. 
For  $i=1,2,3,4$ and for $1 \leq l \leq k$, write $b_{i,l}=\mathcal{O}_{\sum_{j=l}^k C_{ij}}$ and 
$\mathcal{B}_i= \left(b_{i,1}, \dots, b_{i,k}\right)$.
Let $D\coloneqq F+\sum_{i=1}^4 \sum_{j=1}^k j C_{ij}$.

\begin{thm} \label{thm:Xtilde-cat}
  Let $(A_0, A_1, A_2)$ be a full exceptional collection on $\P^2$ and let $L_1, \dots, L_7$ be the strict transforms in $\widetilde{X}$ of the exceptional curves of $\pi\colon X^\prime \to \P^2$. Then \begin{equation} \label{eq:coll-sigmaX}
\sigma_{\widetilde{X}}=	 \left( \alpha^\ast(A_1), \alpha^\ast(A_1) ,\alpha^\ast(A_2), \mathcal{O}_{L_1}, \dots, \mathcal{O}_{L_7},\\
  \mathcal{O}_{D}, \mathcal{B}_1, \mathcal{B}_2,\mathcal{B}_3, \mathcal{B}_4 \right)
\end{equation} is a full exceptional collection on $\widetilde{X}$.
\end{thm}

\begin{proof} 
Let $l_1, \dots, l_7$ be the exceptional curves of $\pi\colon X^\prime \to \P^2$. Then, because of our choice of the map $\pi$, we have that for all $i,j$ $c_{i,0} \neq l_j$, thus none of the $l_j$ contains the point $q$. This implies that $\tau^\ast{\sO_{l_i}}=\mathcal{O}_{L_i}$. Then the statement follows immediately by an iterative application of Orlov's blow-up formula \cite{Orl92} (as stated in \cite{Kuz14}). 
\end{proof}

Let $p, p_1, \dots, p_4$ be the five singular points on $X$, as in Section \ref{ssec:basic}.
Let $G=\frac{1}{4k+1}(1,1)$ and $H=\frac{1}{2k+1}(1,k)$. For $j \in \ZZ/_{4k+1}$ let $\rho_j$ be the irreducible representation of $G$  of weight $j$
and write $e_{j}=\sO_{p}\otimes \rho_j$. For $j \in \ZZ/_{2k+1}$ let $\eta_j$,  be the irreducible representation of $H$ of weight $j$ and write $e_{i,j}=\sO_{p_i}\otimes \eta_j$.  
 Then:
\begin{corollary}\label{thm:stack-cat}

The derived category $D(\mathcal{X})$ admits a full exceptional collection  of the form 
\begin{equation}
\sigma_{\mathcal{X}}= \left( \mathcal{E}_p, 
\mathcal{E}_{p_1}, \mathcal{E}_{p_2}, \mathcal{E}_{p_3}, \mathcal{E}_{p_4},  \Phi(\sigma_{\widetilde{X}}) \right),
\end{equation}
where
$
\mathcal{E}_p=\left( e_{2}, \dots, e_{4k} \right)$, $ \mathcal{E}_{p_i}=\left( e_{i,k+1}, \dots, e_{i,2k}\right) $,  $i=1,2,3,4$. Here, $\sigma_{\widetilde{X}}$ is the full exceptional collection in \eqref{eq:coll-sigmaX}, and  $\Phi$ is the fully faithful functor from Section \ref{ssec:global-case}.
\end{corollary} 
\begin{proof}
An immediate consequence of the global $\GL(2,\CC)$ McKay correspondence \ref{ssec:global-case}, Theorem \ref{thm:Xtilde-cat}, Example \ref{exa:1n11} and Example \ref{exa:12k+1}.
\end{proof}

\subsection{Hom-spaces}\label{ssec:HomSpaces} There are various standard full exceptional collections $(A_0,A_1, A_2)$ on $\P^2$.
We choose $(A_0,A_1, A_2)=(\sO, \sT(-1),\sO(1))$, where $\sT$ is
the tangent bundle on $\P^2$, following \cite{AKO06}.
 Moreover, we choose the map $\pi \colon X^\prime \to \P^2$ to be a blow-up morphism as described in Example \ref{ex:EckardtPoints} (a): this is always possible by Corollary \ref{prop:casea}. 
 This choice implies that the images $\alpha(C_{i,0}) \subset \P^2$ of the  curves $C_{i,0}$ are 
two lines and two conics. Without loss of generality, we assume that $\alpha(C_{i,0})$ is a line if $i=1,2$, is a conic if $i=3,4$. Moreover, setting $x_j=\alpha(L_j)$, we also assume that  $x_1, x_2 \in \alpha(C_{1,0})$, $x_3, x_4 \in \alpha(C_{2,0})$, $x_1, x_3 \in \alpha(C_{3,0})$, and  $x_2, x_4  \in \alpha(C_{4,0})$ (see Figure \ref{fig:EckA}).
\smallskip

With these choices, we describe below the spaces of morphisms 
 of the exceptional collections $\sigma_{\widetilde{X}}$ \footnote{
 The dimension of the Hom spaces of $\sigma_{\widetilde{X}}$ are computed in Section \ref{ssec:morphismsOfCollections} for arbitrary vector bundles $A_0,A_1,A_2$. For $\sigma'_\sX$ they depend on our choices of $A_1,A_2,A_3$ and $\pi$.} and 
\begin{equation}
\sigma^\prime_{\mathcal{X}}= \left( \mathcal{E}^\prime_p, 
\mathcal{E}_{p_1}, \mathcal{E}_{p_2}, \mathcal{E}_{p_3}, \mathcal{E}_{p_4},  \Phi(\sigma_{\widetilde{X}}) \right) \end{equation}
where $\mathcal{E}^\prime_p= (e_{2}[-4k+1], \dots, e_{4k-1}[-2], e_{4k}[-1] )$.	

More precisely, in Lemma \ref{lem:v-bundle and L_i} and Lemma \ref{lem:D and Bi} we describe the spaces of morphisms of $ \sigma_{\widetilde{X}}$. Since $\Phi$ is a fully faithful functor, the same two lemmas also describe all morphisms between objects of $\sigma^\prime_{\mathcal{X}}$ of the form  $\Phi(E)$, where $E \in \sigma_{\widetilde{X}}$.  Lemma \ref{lem:ep} and Lemma \ref{lem:epi} complete the description of the spaces of morphisms of $\sigma^\prime_{\mathcal{X}}$. We omit the computations of the morphisms between objects in $\mathcal{E}^\prime_p$ and $\mathcal{E}_{p_i}$ since they follow directly from the McKay quivers in Example \ref{exa:1n11} and Example \ref{exa:12k+1}.
\smallskip

Our computations show that the collections $\sigma_{\widetilde{X}}$ and $\sigma^\prime_{\mathcal{X}}$ are not strong. However, they have non-trivial morphisms only in degree 0 and 1, which implies the existence of a tilting complex in $D(\widetilde{X})$ and $D(\sX)$ (see Section \ref{ssec:tilting}).

In fact, one can change the exceptional collection $\sigma_{\widetilde{X}}$ to another one with simpler spaces of morphisms. The  exceptional collection for $\widetilde{X}$ obtained by the left mutation of $\sigma_{\widetilde{X}}$ in the pair $(\alpha^\ast \sO (1), \sO_D)$ \footnote{More precisely, one left-mutates  $\mathcal{O}_D$ through  $\alpha^\ast \mathcal{O}(1), \mathcal{O}_{L_1}, \dots, \mathcal{O}_{L_7}$. However, the left mutation  satisfies $L_{\mathcal{O}_{L_i}}(\mathcal{O}_D)=\mathcal{O}_D$ since $\mathcal{O}_D$ and $\mathcal{O}_{L_i}$ are orthogonal.}, that is, 
\begin{equation}
L \sigma_{\widetilde{X}}= \left( \mathcal{O}_{\widetilde{X}}, \alpha^\ast\mathcal{T}(-1),M, \alpha^\ast\mathcal{O}(1) , \mathcal{O}_{L_1}, \dots, \mathcal{O}_{L_7}, \mathcal{B}_1, \mathcal{B}_2,\mathcal{B}_3, \mathcal{B}_4 \right)
\end{equation}
  has only non-trivial morphisms in degree $0$, except for the morphisms from 
 $\Ext^\ast(b_{i,j}, b_{i,l})=\C \oplus \C[-1]$ when $j <l$. Here,  $M$ denotes the left mutation $M=L_{\alpha^\ast \mathcal{O}(1)}\mathcal{O}_D$, defined by
the short exact sequence 
\begin{equation}
\label{eq:DefOfM}
0\to M\to \alpha^\ast \mathcal{O}(1) \xrightarrow{r} \mathcal{O}_D \to 0.
\end{equation}
Moreover,  the corresponding exceptional collection for $\mathcal{X}$
\begin{equation}
L\sigma^\prime_\mathcal{X}=\left( \mathcal{E}^\prime_p, 
\mathcal{E}_{p_1}, \mathcal{E}_{p_2}, \mathcal{E}_{p_3}, \mathcal{E}_{p_4},  \Phi(L\sigma_{\widetilde{X}}) \right)
\end{equation}
has only non-trivial morphism in degree $0$, except for 
the morphisms from $\Phi(b_{i,j})$ to $\Phi(b_{i,l})$ when $j<l$ (Lemma \ref{lem:M}), 
and the morphisms starting from $e_{i, j}$, which have degree $0$ or $1$ (Lemma \ref{lem:M2}).  

We collect the data of the dimensions of the $\Hom$-spaces between objects of $L\sigma_{\widetilde{X}}$ and $L\sigma^\prime_{\sX}$ in Table \ref{tab:MinRes}, Table \ref{tab:Ep} and Table \ref{tab:Epi}. 
In each table, the entry in position $(i,j)$ is either the integer $\dim \Hom(E_i,E_j)$, or a pair of integers $(\dim \Hom(E_i,E_j),\dim\Ext^1(E_i,E_j))$.

\begin{remark}
If $k=1$, the exceptional collection $L \sigma_{\widetilde{X}}$ is in fact strong. Moreover, by shifting by $-1$ any element of $\sE_{p_i}$, $i=1,2,3,4$, one obtains a full strong exceptional collection also for the stack $\sX$.  The dimensions of the $\Hom$-spaces of this collection are summarised in Table \ref{tab:k=1}. 

It may be interesting to find out whether, for all $k>0$ integer, one could obtain a full strong exceptional collection for $\mathcal{X}$ through a sequence of mutations and shifts of  $L\sigma^\prime_{\mathcal{X}}$.\end{remark}

\subsubsection{Basic facts} We recall some well-known facts about sheaves of pure dimension 1 on surfaces which we use repeatedly in the proofs of the lemmas. 
The first one is a Riemann-Roch computation. Here, by a curve on S we mean a connected effective 1-cycle.

\begin{lemma}
\label{lem:chi}
Let $S$ be a smooth projective surface, and let $V_1,V_2$ be sheaves of pure dimension one on $S$. For $i=1,2$, let $D_i$ be curves on $S$ such that $\ch{1}(V_i)=D_i \in H^*(S)$. Then 
\[ \chi(V_1,V_2)=\sum_{i\geq 0}(-1)^i\dim \Ext^i(V_1,V_2)=- D_1\cdot D_2. \]
\end{lemma}

\begin{proof}
Since $V_1$ and $V_2$ are torsion, their Chern characters satisfy $\ch{0}(V_i)=0$. Then the Hirzebruch--Riemann--Roch theorem shows the statement.
\end{proof}

The next Lemma follows directly from the definition of pure sheaf and simple homological algebra:

\begin{lemma} \label{lem:degree}
Let $V$, $W$ be sheaves of pure dimension one on $S$,  supported on curves $C$ and $D$, respectively. Then:
\begin{enumerate}[(i)]
\item if $C$ and $D$ have no common irreducible component, then $\Hom(V,W)=0$;
\item every map $V\to W$ factors through the restriction map $V\to V_{|C\cap D}$;
\item Let $D'=\overline{D\setminus (C\cap D)}$. Every map $V \to W$ factors through the kernel of the restriction map $W \to W_{|D'}$.
\end{enumerate}
\end{lemma}

For future reference, we also 
list the intersection numbers of $K_{\widetilde X}$ with the curves appearing in Figure \ref{fig:surfaces}: 
\begin{equation} \label{eq:canonical-intersection}
\begin{split}
	F\cdot K_{\widetilde{X}}= 4k-1 \qquad 
	C_{i,0}\cdot K_{\widetilde{X}}= 1 \qquad 	C_{i,j}\cdot K_{\widetilde{X}}= 0 \qquad 
		C_{i,k}\cdot K_{\widetilde{X}}= -1 \qquad 
\end{split}
\end{equation}
for $i=1,...,4$ and $j=1,...,k-1$.

\subsubsection{Morphisms of the collections  $\sigma_{\widetilde{X}}$ and $ \sigma^\prime_{\mathcal{X}}$}\label{ssec:morphismsOfCollections}

\begin{lemma} \label{lem:v-bundle and L_i}
Let $V$ be a vector bundle on $\P^2$. Then 
$\Ext^{\ast}(\alpha^\ast V, E )= \CC^{\mathrm{rank}(V)}$,
where $E$ is any object of the collection $(\mathcal{O}_{L_1}, \dots, \mathcal{O}_{L_7},
  \mathcal{O}_{D}, \mathcal{B}_1, \mathcal{B}_2,\mathcal{B}_3, \mathcal{B}_4 )$. 
 Moreover, $\mathcal{O}_{L_i}$ is mutually orthogonal to $\mathcal{O}_{L_j}$ for all $i \neq j$, and to any object of the collection $(\mathcal{O}_{D}, \mathcal{B}_1, \mathcal{B}_2,\mathcal{B}_3, \mathcal{B}_4 )$.
\end{lemma}

\begin{proof}
To prove the first part of the lemma note that, by adjunction, \[\Ext^{\ast}(\alpha^\ast V, E )=\Ext^{\ast}(V,\dR \alpha_\ast E )\] and $\dR \alpha_\ast E$ is the structure sheaf of a point.
The second part of the lemma follows from the fact that the objects involved have disjoint supports.
\end{proof}

\begin{lemma} \label{lem:D and Bi} 
Let $i=1,2,3,4$. Then:
\begin{enumerate}[1.]
\item for $j=1, \dots, k$, $\Ext^{\ast}(\mathcal{O}_D, b_{i,j})= \CC \oplus \CC[-1]$;
\item $\mathcal{B}_i$ and $\mathcal{B}_j$, $j \neq i$, are mutually orthogonal;
\item for $1 \leq j<l \leq k$, $\Ext^{\ast}( b_{i,j} , b_{i, l})= \CC \oplus \CC[-1]$	.
\end{enumerate}
\end{lemma}

\begin{proof}
To prove (1) and (3) use Lemma \ref{lem:chi} 
 and Lemma \ref{lem:degree}. 
We show for example the proof of (1) for $j=k$. It follows from Lemma \ref{lem:degree} (ii) that there is only one map $\sO_D \to b_{i,k}$. Moreover, by Lemma \ref{lem:chi}  \[\chi(\mathcal{O}_D, b_{i,k})=-D\cdot C_{i,k}=-(F+(k-1)\cdot C_{i,k-1}+k\cdot C_{i,k}) \cdot C_{i,k}=0\] On the other hand, $\Ext^2(\mathcal{O}_D, b_{i,k})=\Ext^2(\mathcal{O}_D, \sO_{C_{i,k}})\simeq \Hom(\sO_{C_{i,k}}(1), \sO_D)^*$  by Serre duality and \eqref{eq:canonical-intersection}, and the latter is trivial by Lemma \ref{lem:degree} (iii).  
Assertion (2) is immediate since $\mathcal{B}_i$ and $\mathcal{B}_j$ have disjoint supports.
\end{proof}

To carry out the rest of the computations, we will combine our result on the left adjoint functor $\Psi$ to $\Phi$ with the configuration of curves described in Theorem \ref{thm:min-resol}.  In particular, the local description from Examples \ref{ex:1/n11ImagesOfPsi} and \ref{ex:1/(2k+1)1kImagesOfPsi} gives

\begin{align}\label{eq:Psi(ei)}
\Psi(e_{i}) \simeq \begin{cases}
0 & \mbox{if } i \in {0, \dots, 4k-2}\\
\sO_{F}(-4k-1)[1] & \mbox{ if } i=4k-1 \\
\sO_{F}(-4k)  & \mbox{ if } i=4k
\end{cases}
\end{align}
and, writing $C_i=\sum_{j=0}^{k-1} C_{i,j}$,
 
\begin{align} \label{eq:psieij}
\Psi(e_{ij}) \simeq \begin{cases}
0 & \mbox{if } j \in {0, \dots, k-1}\\
\sO_{C_i}(C_i)[1] & \mbox{ if } j=k \\
\sO_{C_{i,2k-j}}(-1)  & \mbox{ if } j \in \{k+1, \dots, 2k-1\} \\
 \sO_{C_{i,0}}(-2) &  \mbox{ if } j=2k
\end{cases}
\end{align}

\begin{lemma}  \label{lem:ep} We have that:
\begin{enumerate}[1.]
\item for $2\leq l< 4k-1$ , $\Ext^*(e_l,\Phi(E))=0$ for all $E\in D^b(\Coh(\widetilde{X}))$;
\item for $V$ a vector bundle on $\P^2$, $\Ext^{\ast}(e_{4k-1},\Phi(\alpha^\ast(V)))=\CC^{\mathrm{rank(V)}}[-2]$;
\item   $\Ext^{\ast}(e_{4k-1},\Phi(\mathcal{O}_D))=\CC^2[-1]\oplus \CC[-2]$;
\item  for $i=1,2,3,4$, $j=1,\dots, k$, $\Ext^{\ast}(e_{4k-1},\Phi(b_{i,j}))=\CC[-1]$;
\item  for $V$ a vector bundle on $\P^2$, $\Ext^{\ast}(e_{4k},\Phi(\alpha^\ast(V)))=0$;
\item  $\Ext^{\ast}(e_{4k},\Phi(\mathcal{O}_D))=\CC$;
\item for $=1,2,3,4$, $j=1,\dots, k$, $\Ext^{\ast}(e_{4k},\Phi(b_{i,j}))=\CC[-1]$;
\item $\mathcal{E}_p$ 
and $( \Phi(\mathcal{O}_{L_1}), \dots,\Phi(\mathcal{O}_{L_7}) )$ are mutually orthogonal.
\end{enumerate}
\end{lemma}

\begin{proof}
In all the statements above, we use that $\Psi$ is left adjoint to $\Phi$, and conclude with techniques similar to the previous lemmas. We show (3) as an illustration: 
\[ \Ext^{\ast}(e_{4k-1},\Phi(\mathcal{O}_D))= \Ext^{\ast}_{\widetilde{X}}(\sO_F(-4k-1)[1],\mathcal{O}_D)  \]
by \eqref{eq:Psi(ei)}. The restriction of $\sO_D$ to $\overline{D\setminus F}$ has kernel $\sO_F(-4k)$, so we have
\[ \Hom_{\widetilde{X}}(\sO_F(-4k-1),\mathcal{O}_D)=\Hom_{\widetilde{X}}(\sO_F(-4k-1),\mathcal{O}_F(-4k))=\C^2 \]
by Lemma \ref{lem:degree} (iii). On the other hand,  
\[\Ext^2_{\widetilde{X}}(\sO_F(-4k-1),\mathcal{O}_D)\simeq \Hom(\sO_D, \sO_F(-2))^*=\Hom(\sO_F,\sO_F(-2))^*=0\]
by Serre duality, \eqref{eq:canonical-intersection}, and Lemma \ref{lem:degree} (ii). Since $\chi(\sO_F(-4k-1),\sO_D)=-F\cdot D = 1$ (using Lemma \ref{lem:chi}), we show (3).

\end{proof}

\begin{lemma} \label{lem:epi} Let $i=1,2,3,4$. Then:
\begin{enumerate}[1.]
\item for $V$ a vector bundle on $\P^2$, $\Ext^{\ast}(e_{i,j},\Phi(\alpha^\ast(V)))=0$ for all $j \neq 2k$;
\item $e_{i,j}$ and $\Phi(\mathcal{O}_{L_h})$ are mutually orthogonal for all $j \neq 2k$;
\item $\Ext^{\ast}(e_{i,2k}, \Phi(\mathcal{O}_{\widetilde{X}}))=0$; 
 \item $\Ext^{\ast}(e_{i,2k}, \Phi( \alpha^\ast( \mathcal{T}(-1) ) ) )=\Ext^{\ast}(e_{i,2k}, \Phi( \alpha^\ast(\mathcal{O}(1))  ))=\begin{cases}
\CC[-1] & \mbox{ if  \ } i=1,2\\
\CC^2[-1] & \mbox{ if  \ } i=3,4\\
\end{cases}
$

\item $\Ext^{\ast}(e_{i,2k},\Phi(O_{L_h}))=\CC^{a_{i,j}}[-1]$, where $a_{i,j}$ is the entry  $(i,j)$ of the matrix:
\[
A=\begin{pmatrix}
1 & 1 & 0 & 0 & 0 & 0 & 0  \\
0 & 0 & 1 & 1 & 0 & 0 & 0  \\
1 & 0 & 1 & 0 & 1 & 1 & 1  \\
0 & 1 & 0 & 1 & 1 & 1 & 1 
\end{pmatrix}
\]
\item   $\Ext^{\ast}(e_{i,j},\Phi(\mathcal{O}_D))=\begin{cases}
0 & \mbox{ if } j <2k\\
\CC[-1] & \mbox{ if } j=2k\\
\end{cases}
$
\item  $\Ext^{\ast}(e_{i,j},\Phi(b_{i, l}))=\begin{cases}
\CC & \mbox{ if } l=2k-j\\
\CC[-1]  & \mbox{ if } l =2k-j+1\\
0 & \mbox{ otherwise }
\end{cases}$

\item  $\mathcal{E}_{i}$ and $\Phi(\mathcal{B}_{j})$, $j \neq i$, are mutually orthogonal. 
\end{enumerate}
\end{lemma}
\begin{proof} 
Again, in all the statements above, we use that $\Psi$ is left adjoint to $\Phi$ and we conclude using ideas similar to those of the previous lemmas.
As an example, we show (4). 

By \eqref{eq:psieij}, Serre duality, \eqref{eq:canonical-intersection}, and adjunction,
\[\begin{split} 
\Ext^\ast(e_{i,2k}, \Phi( \alpha^\ast( \mathcal{T}(-1) ) ))&=\Ext^\ast(\sO_{C_{i,0}}(-2), \alpha^\ast(\mathcal{T}(-1)) )\\
&=\Ext^{2-\ast}(\alpha^\ast(\mathcal{T}(-1)), \sO_{C_{i,0}}(-1))^*\\
&= \Ext^{2-\ast}(\mathcal{T}(-1), \sO_{\alpha(C_{i,0})}(-1))^*
\end{split}\] Now, by the long exact sequence for the Hom functor applied to the Euler sequence, one finds that 
\[\Ext^{\ast}(\mathcal{T}(-1), \sO_{\alpha(C_{i,0})}(-1))= \Ext^{\ast-1}(\mathcal{O}(-1), \sO_{\alpha(C_{i,0})}(-1))\] since  $\mathrm{H}^\ast(\P^2,\sO_{\alpha(C_{i,0})}(-1) )=0$ because $C_{i,0}$ is rational. One concludes by observing that
\[
\Ext^{\ast}(\mathcal{O}(-1), \sO_{\alpha(C_{i,0})}(-1)) = 
\begin{cases}
\Ext^{\ast}(\mathcal{O}, \alpha(C_{i,0}))=\C & \mbox{ if } \alpha(C_{i,0}) \mbox{ is a line,}\\
\Ext^{\ast}(\mathcal{O}, \alpha(C_{i,0})(1))=\C^2 & \mbox{ if } \alpha(C_{i,0}) \mbox{ is a conic.}
\end{cases}
\]
Analogously, 
\[ \Ext^\ast(e_{i,2k}, \Phi(\alpha^\ast(\mathcal{O}(1))))=\Ext^\ast(\sO_{C_{i,0}}(-2), \alpha^\ast(\mathcal{O}(1)))= \Ext^{2-\ast}(\mathcal{O}(1), \sO_{\alpha(C_{i,0})}(-1))^*
\] Hence, 
$\Ext^\ast(e_{i,2k}, \Phi(\alpha^\ast(\mathcal{O}(1))))$ is $\CC[-1]$ if $\alpha(C_{i,0})$ is a line, $\CC^2[-1]$  if $\alpha(C_{i,0})$ is a conic.
\end{proof}

\subsubsection{Morphisms of the mutated collections}

\begin{lemma} \label{lem:M} Let $M$ be as in \eqref{eq:DefOfM}. Then:
\begin{enumerate}[1.]
\item $\Ext^{\ast}(\mathcal{O}_{\widetilde{X}},M)=\CC^2$
\item   $\Ext^{\ast}(\alpha^\ast(\mathcal{T}(-1)),M)=\CC$

\item $\Ext^{\ast}(M ,\alpha^\ast(\mathcal{O}(1)))= \CC$
\item for $i=1, \dots, 7$, $\Ext^{\ast}(M ,\mathcal{O}_{L_i})= \CC$
\item for $i=1,2,3,4$, $j=1, \dots, k$, $\Ext^{\ast}(M, b_{i,j})= \CC$
\end{enumerate}
	
\end{lemma}

\begin{proof}
All the computations follow from the long exact sequence of the $\Hom$ functor applied to \eqref{eq:DefOfM}, combined with Lemmas \ref{lem:v-bundle and L_i} and \ref{lem:D and Bi}. For example, composition with $r$ gives a surjection 
\[ \Hom(\alpha^*(\sT(-1)), \alpha^*(\sO(1))) \twoheadrightarrow \Hom(\alpha^*(\sT(-1)),\sO_D), \]
which immediately implies (2).
\end{proof}

\begin{lemma} \label{lem:M2}
\label{lem:M1} Let $M$ be as in \eqref{eq:DefOfM}. Then:
\begin{enumerate}[1.] 
\item 
$\Ext^{\ast}(e_l,\Phi(M))=\begin{cases}  
0 & \mbox{if } \ 2\leq l< 4k-1\\ 
\CC^2[-2] & \mbox{if } \  l=4k-1\\
\CC[-1] & \mbox{if } \ l=4k
\end{cases} $
\item for $i=1,2,3,4$ and for $k<j <2k$, $\Ext^{\ast}(e_{i,j}, \Phi(M))=0$;
\item  $\Ext^{\ast}(e_{i,2k}, \Phi(M))=\begin{cases}  
0 & \mbox{if } \  i=1,2\\ 
\CC[-1] & \mbox{if } \  i=3,4\\
\end{cases} $
	
\end{enumerate}
\end{lemma}
\begin{proof}
All the computations follow from 
the long exact sequence of the $\Hom$ functor applied to \eqref{eq:DefOfM}, together with Lemma \ref{lem:ep} and Lemma \ref{lem:epi}. For example, assertions (1) and (6) in Lemma \ref{lem:epi} immediately imply (2). 
\end{proof}
\smallskip


\subsection{Exceptional collections and tilting}\label{ssec:tilting}

Let $\sD$ be a triangulated category. We say that $T\in \sD$ \textit{generates} $\sD$ if, for every $E\in \sD$, we have that $\Hom(T,E)=0$ implies $E\simeq 0$, and that $T$ \textit{classically generates} $\sD$ if the smallest full triangulated subcategory containing $T$, closed under isomorphisms and direct summands, is $\sD$ itself.
An object $T\in \sD$ is \textit{partial tilting} if $\Hom(T,T[l])=0$ for all $l\neq 0$. We say that $T$ is a \textit{tilting} object if, moreover, it generates $\sD$.

Generally speaking, tilting objects produce equivalences between $\sD$ and derived categories of modules over their endomorphism algebra. More precisely, suppose $\sD=D^b(\Coh(Z))$ where $Z$ is a smooth Deligne-Mumford stack with finite stabilizers and projective coarse moduli space. Then we have:

\begin{proposition}
\label{prop:tilting_gives_equivalence}
Let $T$ be a tilting complex in $\sD$ and denote by $\Lambda=\Hom(T,T)$ its endomorphism algebra. Then, the functor $\dR\Hom_{\sO_{Z}}(T,-)$ defines an equivalence of triangulated categories
\[ \sD \xrightarrow{\sim} D^b(\mathrm{mod-}\Lambda) \]
with inverse $-\otimes^{\dL}_\Lambda T$.
\end{proposition}

\begin{proof}
For schemes, this is \cite[Theorem 7.6]{HVdB07}.
If $\sX$ is a smooth DM stack, $D(\QCoh(Z))$ is compactly generated (this is shown for the more general class of \textit{perfect} stacks in \cite[Prop. 3.9]{BZFN10}\footnote{Smooth DM stacks with finite stabilizers have affine diagonal, and hence are perfect. Our setting is analogous to that of \cite[Lemma 4.1]{AU15}, which follows a similar argument.}
). Then, $T$ generates $D(\QCoh(Z))$ if and only if it classically generates $\sD$ \cite[Theorem 2.1.2]{BVdB03}, and the standard proof (see e.g. \cite[\S 3]{Bae88}) suffices to conclude. 
\end{proof}

If $(E_1,...,E_n)$ is a strong full exceptional collection in $\sD$, then $\sE=\oplus_{i=1}^n E_i$ is a tilting complex. If the collection is not strong, $\Lambda_\sE=\Hom(\sE,\sE)$  is a $\mathrm{dg}$-algebra, and one may still obtain an equivalence, analogous to Prop. \ref{prop:tilting_gives_equivalence}, between $\sD$ and the derived category of $\mathrm{dg}$-modules over $\Lambda_\sE$. Alternatively, if the exceptional collection only admits maps in degrees $0$ and $1$, we obtain an explicit tilting complex from the so called \textit{universal extension trick}, which we briefly recall here:

\begin{lemma}
\label{lem:UnivExtTrick}
Assume that $(E_1,...,E_n)$ is a collection of objects of $\sD$ such that:
\begin{enumerate}[(i)]
\item \label{itm:self} each $E_i$ is partial tilting;
\item\label{itm:gen} the sum $\oplus_{i=1}^n E_i$ generates $D^b(\Coh(\sX))$;
\item \label{itm:up} if $i<j$, then $\Hom(E_i,E_j[l])=0$ for all $l\neq 0,1$;
\item \label{itm:down} if $i<j$, then $\Hom(E_j,E_i[l])=0$ for all $l\in \Z$.
\end{enumerate}
Then $\sD$ admits a tilting complex. 
\end{lemma}

\begin{proof} 

We direct the reader to \cite[Lemma 2.4]{Har21} for the details of the proof, we only recall the explicit construction here.

Proceed by induction on $n$: the case $n=1$ is clear. Assume now that $n>1$. Let $r$ denote the dimension of $\Ext^1(E_1,E_2)$, then the universal extension map produces a triangle
\begin{equation}
\label{eq:UnivExt}
 E_2 \to F \to E_1^{\oplus r}.  
\end{equation}
Put $E_2'=E_2\oplus F$ and $E_i'=E_{i}$ for $i=3,...,n$. Then,   $\{E_2', E_3',...,E_{n}'\}$ still satisfies the assumptions. 
This claim is proven as in \cite{Har21}. The only significant difference is that we do not assume the $E_i$ to be sheaves, and therefore need more care to control negative Ext groups.
\end{proof}

\begin{corollary}
\label{cor:ExcCollToTilting}
Suppose $(E_1,...,E_n)$ is a full exceptional collection of $\sD$, such that $\Hom(E_i,E_j[l])=0$ for all $l\neq 0,1$. Then $\sD$ admits a tilting complex. 
\end{corollary}

\begin{proof}
This is immediate from Lemma \ref{lem:UnivExtTrick}.
\end{proof}

Applying this to the case at hand:

\begin{proposition}
The exceptional collections $L\sigma_{\widetilde{X}}$ and $L\sigma'_\sX$ from Section \ref{ssec:HomSpaces} give rise to tilting complexes in $D(\widetilde{X})$ and $D(\sX)$. 
\end{proposition}


\begin{table}[ht!]
\centering
\begin{tabular}{l|ccccccc}
& $\sO_{\widetilde{X}}$   & $\alpha^\ast \sT(-1)$ & $M$   & $ \alpha^\ast \sO(1)$  & $\sO_{L_t}$ & $b_{i,j}$ & $b_{i,l}$   \\
\hline
$\sO_{\widetilde{X}}$   & $1$ & $3$ & $2$ & $3$ & $1$ & $1$ & $1$   \\
$\alpha^\ast \sT(-1)$ &     & $1$ & $1$   & $3$ & $2$ & $2$ & $2$   \\
$M$   &     &     & $1$ & $1$ & $1$ & $1$ & $1$   \\
$\alpha^\ast \sO(1)$  &     &     &     & $1$ & $1$ & $1$ & $1$   \\
$\sO_{L_t}$ &     &     &     &     & $1$ & $0$ & $0$   \\
$b_{i,j}$ &     &     &     &     &     & $1$ & $(1,1)$  \\
$b_{i,l}$ &     &     &     &     &     &     & $1$  
\end{tabular}
\caption{The  Hom-spaces of the collection $L\sigma_{\widetilde{X}}$.}
\label{tab:MinRes}
\end{table}


\begin{table}[ht!]
\centering
\begin{tabular}{l|c||cccccc}
          & $e_{4k-j}[-j-1]$ &  $\Phi(\sO_{\widetilde{X}})$   & $\Phi \alpha^\ast \sT(-1)$ & $\Phi M$   & $\Phi \alpha^\ast \sO(1)$  & $\Phi \sO_{L_t}$ & $\Phi b_{i,1}$  \\
                     \hline 
       $\vdots$ & \multirow{4}{*}{\thead{McKay quiver \\ (with shifts) \\ Example \ref{exa:1n11}}} & & 
              &                                                                                       \\  
$e_{4k-2}[-3]$ &      & $0$ & $0$ & $0$ & $0$ & $0$ & $0$  \\
$e_{4k-1}[-2]$ &      & $1$ & $2$ & $2$ & $1$ & $0$ & $1$  \\
$e_{4k}[-1]$   &      & $0$ & $0$ & $1$ & $0$ & $0$ & $1$ 
\end{tabular}
\caption{The Hom-spaces between objects in $\sE^\prime_p$ and from $\sE^\prime_p$ to $\Phi(L\sigma_{\widetilde{X}})$. The shifts of the objects in $\sE_p'$ make the arrows from the McKay quiver into maps of degree $0$.}
\label{tab:Ep}
\end{table}


\begin{table}[ht!]
\centering
\begin{tabular}{l|c||ccccccc}
           & $e_{i,2k-j}$ &  $\Phi(\sO_{\widetilde{X}})$   & $\Phi \alpha^\ast \sT(-1)$ & $\Phi M$   & $\Phi \alpha^\ast \sO(1)$  & $\Phi \sO_{L_t}$ &  $\Phi b_{i,2k-j}$ & $\Phi b_{i,2k-j+1}$   \\
           \hline
$e_{i,k+1}$     & \multirow{8}{*}{\thead{McKay quiver \\ Example \ref{exa:12k+1}}}  & $0$ & $0$ & $0$ & $0$ & $0$                            & $\cdots$  & $\cdots$  \\
       $\vdots$ &  &    &     &     &     &                                  &      &           \\
$e_{i,j}$    &  &  $0$ & $0$ & $0$ & $0$ & $0$                              &    $(1,0)$ & $(0,1)$   \\
   $\vdots$     &   &     &     &     &     &                                     &   &         \\
$e_{i,2k-1}$    & &    $0$ & $0$ & $0$ & $0$ & $0$                                       & $\cdots$  & $\cdots$  \\
\cline{1-1}\cline{3-9}
$e_{1,2k}$ 
  &       & $0$ & $(0,1)$ & $0$ & $(0,1)$ & $(0, a_{1,t})$ &             $\cdots$  & $\cdots$   \\
$e_{3,2k}$ 
 &       & $0$ & $(0,2)$ & $(0,1)$ & $(0,2)$ &   $(0, a_{3,t})$                                        & $\cdots$  & $\cdots$    
\end{tabular}
\caption{The Hom-spaces from between objects in $\sE_{p_i}$ and from  $\sE_{p_i}$ to $\Phi(L\sigma_{\widetilde{X}})$. The arrows of the McKay quiver involved in the first column are morphisms of degree $1$. Thus every $e_{i,j}$ admits both $\Hom $'s and $\Ext^1$'s to other objects of the collection.}
\label{tab:Epi}
\end{table}


\begin{table}[ht!]
\centering
\begin{tabular}{l|ccc|c||cccc|c|c}
      & $e_2[-3]$ & $e_3[-2]$ & $e_4[-1]$ & $e_{i,2}[-1]$ & $\Phi\alpha^*\sO_{\widetilde{X}}$   & $\Phi\alpha^*\sT(-1)$ & $\Phi M$   & $\Phi\alpha^*\sO(1)$  & $\Phi\sO_{L_t}$    & $\Phi\sO_{C_{i,1}}$                     \\
      \hline
$e_2[-3]$  & $1$  & $2$  & $1$  &  $0$& $0$ &$0$& $0$& $0$& $0$  & $0$                                                                                                \\\cline{5-11}
$e_3[-2]$  &      & $1$  & $2$  & \multirow{2}{*}{$O_{2\times 4}$}  & $1$ & $2$ & $2$ & $1$ & \multirow{2}{*}{$O_{2\times 7}$} &  $1$  \\
$e_4[-1]$  & &  & $1$  &  & $0$ & $0$ & $1$ & $0$ &  & $1$ \\
\cline{5-5}\cline{10-11}
$e_{1,2}[-1]$ &      &      &      & \multirow{4}{*}{$\mathrm{Id}_4$} & $0$ & $1$ & $0$ & $1$ & \multirow{4}{*}{$A$} & \multirow{4}{*}{$\mathrm{Id}_4$}  \\
$e_{2,2}[-1]$ &      &      &      &    & $0$ & $1$ & $0$ & $1$ &  &                    \\
$e_{3,2}[-1]$ & & & & & $0$ & $2$ & $1$ & $2$ &  & \\
$e_{4,2}[-1]$ & & & & & $0$ & $2$ & $1$ & $2$ &  &     
\\                    
\end{tabular}
\caption{The strong exceptional collection for $\mathcal{X}$ when $k=1$.}
\label{tab:k=1}
\end{table}



\newcommand{\etalchar}[1]{$^{#1}$}

\end{document}